\documentclass[english,11pt]{article}
\include{macros}
\usepackage[colorinlistoftodos]{todonotes}

\newcommand{\removed}[1]{}

\begin{document}

\title{An $\ell_p$ theory of PCA and spectral clustering}\blfootnote{Author names are sorted alphabetically.}

\author{Emmanuel Abbe\thanks{Institute of Mathematics, EPFL. Email: \texttt{emmanuel.abbe@epfl.ch}.}
	\and Jianqing Fan\thanks{Department of ORFE, Princeton University. Email: \texttt{jqfan@princeton.edu}.}
	\and Kaizheng Wang\thanks{Department of IEOR and Data Science Institute, Columbia University. Email: \texttt{kaizheng.wang@columbia.edu}.}
}

\date{This version: April 2022}

\maketitle

\begin{abstract}
Principal Component Analysis (PCA) is a powerful tool in statistics and machine learning.
While existing study of PCA focuses on the recovery of principal components and their associated eigenvalues, there are few precise characterizations of individual principal component scores that yield low-dimensional embedding of samples. That hinders the analysis of various spectral methods. In this paper, we  first develop an $\ell_p$ perturbation theory for a  hollowed version of PCA in Hilbert spaces which provably improves upon the vanilla PCA in the presence of heteroscedastic noises. Through a novel $\ell_p$ analysis of eigenvectors, we investigate entrywise behaviors of principal component score vectors and show that they can be approximated by linear functionals of the Gram matrix in $\ell_p$ norm, which includes $\ell_2$ and $\ell_\infty$ as special cases. For sub-Gaussian mixture models, the choice of $p$ giving optimal bounds depends on the signal-to-noise ratio, which further yields optimality guarantees for spectral clustering. For contextual community detection, the $\ell_p$ theory leads to simple spectral algorithms that achieve the information threshold for exact recovery and the optimal misclassification rate.
\end{abstract}


\section{Introduction}

\subsection{Overview}

Modern technologies generate enormous volumes of data that present new statistical and computational challenges. The high throughput data come inevitably with tremendous amount of noise, from which very faint signals are to be discovered. Moreover, the analytic procedures must be affordable in terms of computational costs. While likelihood-based approaches usually lead to non-convex optimization problems that are NP-hard in general, the method of moments provides viable solutions to the computation challenges.

Principal Component Analysis (PCA) \citep{Pea01} is arguably the most prominent tool of this type.
It significantly reduces the dimension of data using eigenvalue decomposition of a second-order moment matrix. Unlike the classical settings where the dimension $d$ is much smaller than the sample size $n$, nowadays it could be the other way around in numerous applications \citep{Rin08, Nov08, YRu01}. Reliability of the low-dimensional embedding is of crucial importance, as all downstream tasks are based on that. Unfortunately, existing theories often fail to provide sharp guarantees when both the dimension and noise level are high, especially in the absence of sparsity structures. The matter is further complicated by the use of nonlinear kernels for dimension reduction \citep{SSM97}, which is de facto PCA in some infinite-dimensional Hilbert space. 

In this paper, we investigate the spectral embedding returned by a hollowed version of PCA. Consider the signal-plus-noise model
\begin{align}
\bx_i = \bar\bx_i + \bz_i \in \RR^d,\qquad i \in [n].
\label{eqn-signal-plus-noise}
\end{align}
Here $\{ \bx_i \}_{i=1}^n $ are noisy observations of signals $\{ \bar\bx_i \}_{i=1}^n $ contaminated by $\{ \bz_i \}_{i=1}^n $. Define the data matrices $\bX = (\bx_1, \cdots , \bx_n)^{\top} \in \RR^{n \times d}$ and $\bar\bX = (\bar\bx_1, \cdots , \bar\bx_n)^{\top} \in \RR^{n \times d}$. Let $\bar\bG = \bar\bX \bar\bX^{\top} \in \RR^{n \times n}$ be the Gram matrix of $\{ \bar\bx_i \}_{i=1}^n $, and $\bG = \cH ( \bX \bX^{\top} )$ be the \textit{hollowed} Gram matrix of $\{ \bx_i \}_{i=1}^n $ where $\cH(\cdot)$ is the hollowing operator, zeroing out all diagonal entries of a square matrix. Denote by $\{ \lambda_j , \bu_j \}_{j=1}^n$ and $\{ \bar\lambda_j , \bar\bu_j \}_{j=1}^n$ the eigen-pairs of $\bG$ and $\bar\bG$, respectively, where the eigenvalues are sorted in descending order. While PCA computes the embedding by eigen-decomposition of $\bX \bX^{\top}$, here we delete its diagonal to enhance concentration and handle heteroscedasticity \citep{KGi00}. We seek an answer to the following fundamental question: how do the eigenvectors of $\bG$ relate to those of $\bar\bG$?

Roughly speaking, our main results state that
\begin{align}
\bu_j = \bG \bu_j / \lambda_j \approx \bG \bar\bu_j / \bar\lambda_j,
\label{intro-eqn-Lp}
\end{align}
where the approximation relies on the $\ell_p$ norm for a proper choice of $p$. In words, the eigenvector $\bu_j$ is a nonlinear function of $\bG$ but can be well approximated by the linear function $\bG \bar\bu_j / \bar\lambda_j$ in the $\ell_p$ norm where $p$ is given by the model's signal-to-noise ratio (SNR). 
This linearization facilitates the analysis and allows to quantify how the magnitude of the signal-to-noise ratio affects theoretical guarantees for signal recovery.

In many statistical problems such as mixture models, the vectors $\{ \bar{\bx}_i \}_{i=1}^n$ live in a low-dimensional subspace of $\RR^d$. Their latent coordinates reflect the geometry of the data, which can be decoded from eigenvalues and eigenvectors of $\bar\bG$. Our results show how well the spectral decomposition of $\bG$ reveals that of $\bar\bG$, characterizing the behavior of individual embedded samples. From there we easily derive the optimality of spectral clustering in sub-Gaussian mixture models and contextual stochastic block models, in terms of both the misclassification rates and the exact recovery thresholds. In particular, the linearization of eigenvector (\ref{intro-eqn-Lp}) helps develop a simple spectral method for contextual stochastic block models, efficiently combining the information from the network and the node attributes.

Our general results hold for Hilbert spaces. 
It is easily seen that construction of the hollowed Gram matrix $\bG$ and the subsequent steps only depend on pairwise inner products $\{ \langle \bx_i, \bx_j \rangle \}_{1 \leq i,j \leq n}$. This makes the ``kernel trick'' applicable \citep{CSh00}, and our analysis readily handles (a hollowed version of) kernel PCA. 

\subsection{A canonical example}\label{sec-intro-gmm}

We demonstrate the merits of the $\ell_p$ analysis using spectral clustering for a mixture of two Gaussians. Let $\by \in \{ \pm 1 \}^n$ be a label vector with i.i.d. Rademacher entries and $\bmu \in \RR^d$ be a deterministic mean vector, both of which are unknown. Consider the model
\begin{align}
\bx_i = y_i \bmu + \bz_i , \qquad i \in [n],
\label{eqn-intro-gmm}
\end{align}
where $\{ \bz_i \}_{i=1}^n$ are i.i.d.~$N( \mathbf{0} , \bI_d )$ vectors. The goal is to estimate $\by$ from $\{ \bx_i \}_{i=1}^n$. (\ref{eqn-intro-gmm}) is a special case of the signal-plus-noise model (\ref{eqn-signal-plus-noise}) with $\bar\bx_i = y_i \bmu$. Since $\PP ( y_i = 1 ) = \PP ( y_i = -1 ) = 1/2$, $\{ \bx_i \}_{i=1}^n$ are i.i.d. samples from a mixture of two Gaussians $\frac{1}{2} N(\bmu, \bI_d) + \frac{1}{2} N(-\bmu, \bI_d )$.

By construction, $\bar\bX = (\bar\bx_1,\cdots, \bar\bx_n)^{\top} = \by \bmu^{\top}$ and $\bar\bG = \| \bmu \|_2^2 \by \by^{\top}$ with $\bar\bu_1 = \by / \sqrt{n}$ and $\bar\lambda_1 = n \| \bmu \|_2^2$.   Hence, $\sgn(\bu_1)$ becomes a natural estimator for $\by$, where $\sgn(\cdot)$ is the entrywise sign function. A fundamental question is whether the empirical eigenvector $\bu_1$ is informative enough to accurately recover the labels in competitive regimes. To formalize the discussion, we denote by
\begin{align}
\mathrm{SNR} = \frac{\| \bmu \|_2^4}{ \| \bmu \|_2^2 + d / n }
\label{eqn-intro-snr}
\end{align}
the signal-to-noise ratio of model (\ref{eqn-intro-gmm}). Consider the challenging asymptotic regime where $n \to \infty$ and $1 \ll \mathrm{SNR}  \lesssim \log n$\footnote{In Theorem \ref{KPCA-thm-GMM} we derive results for the exact recovery of the spectral estimator, i.e. $\PP ( \sgn(\bu_1) = \pm \by ) \to 1$, when $\mathrm{SNR} \gg \log n$. Here we omit that case and discuss error rates.}. The dimension $d$ may or may not diverge. According to Theorem \ref{KPCA-thm-GMM}, the spectral estimator $\sgn(\bu_1)$ achieves the minimax optimal misclassification rate  
\begin{align}
e^{- \frac{1}{2}\mathrm{SNR}  [1 + o(1)]}.
\label{eqn-intro-rate}
\end{align}

In order to get this, we start from an $\ell_p$ analysis of $\bu_1$. Theorem \ref{KPCA-thm-GMM-Lp} shows that
\begin{align}
\PP \Big( \min_{s = \pm 1} \| s \bu_1 - \bG \bar\bu_1 / \bar\lambda_1 \|_p < \varepsilon_n \| \bar\bu_1 \|_p \Big) > 1 - C e^{-p}
\label{eqn-intro-Lp}
\end{align}
for $p = \mathrm{SNR}$, some constant $C>0$ and some deterministic sequence $\{ \varepsilon_n \}_{n=1}^{\infty}$ tending to zero. 
On the event $\| s \bu_1 - \bG \bar\bu_1 / \bar\lambda_1 \|_p < \varepsilon_n \| \bar\bu_1 \|_p$, we apply a Markov-type inequality to the entries of $(s \bu_1 - \bG \bar\bu_1 / \bar\lambda_1)$:
\begin{align}
& \frac{1}{n} |\{ i:~ |(s \bu_1 - \bG \bar\bu_1 / \bar\lambda_1)_i | > \sqrt{\varepsilon_n / n} \}|
 \leq
\frac{
 \frac{1}{ n } \sum_{i=1}^{n} | (s \bu_1 - \bG \bar\bu_1 / \bar\lambda_1)_i |^p
}{
(\sqrt{\varepsilon_n /n} )^p
}
\notag \\
&
  \overset{\mathrm{(i)}}{=} \bigg( \frac{
	\| s \bu_1 - \bG \bar\bu_1 / \bar\lambda_1 \|_p
}{
\sqrt{\varepsilon_n} \| \bar\bu_1\|_p
}
\bigg)^p
\leq \varepsilon_n^{p / 2},
\label{eqn-intro-Lp-outliers}
\end{align}
where $\mathrm{(i)}$ follows from $ \bar\bu_1 = \by / \sqrt{n}$ and $\| \bar\bu_1 \|_p^p = n (1/\sqrt{n})^p$. Hence all but an $\varepsilon_n^{\mathrm{SNR} / 2} $ fraction of $\bu_1$'s entries are well-approximated by those of $\bG \bar\bu_1 / \bar\lambda_1$. On the other hand, since the misclassification error is always bounded by 1, the exceptional event in (\ref{eqn-intro-Lp}) may at most contribute an $C e^{-\mathrm{SNR}}$ amount to the final error. Both $\varepsilon_n^{\mathrm{SNR} / 2} $ and $C e^{-\mathrm{SNR}}$ are negligible compared to the optimal rate $e^{-\mathrm{SNR} / 2}$ in (\ref{eqn-intro-rate}).  This helps us show that the $\ell_p$ bound (\ref{eqn-intro-Lp}) ensures sufficient proximity between $\bu_1$ and $\bG \bar\bu_1 / \bar\lambda_1$, and the analysis boils down to the latter term.

We now explain why $\bG \bar\bu_1 / \bar\lambda_1$ is a good target to aim at. Observe that
\begin{align}
( \bG \bar\bu_1 )_{i} = [ \cH ( \bX \bX^{\top} ) \bar\bu_1 ]_i
=  \sum_{j \neq i} \langle \bx_i , \bx_j \rangle y_j / \sqrt{n}
\propto  \langle \bx_i, \hat\bmu^{(-i)} \rangle,
\label{eqn-intro-lda}
\end{align}
where  
 $\hat\bmu^{(-i)} = \frac{1}{n-1} \sum_{j \neq i} \bx_j y_j$ is the leave-one-out sample mean.
 Consequently, the (unsupervised) spectral estimator $\sgn [ (\bu_1)_i ]$ for $y_i$ is approximated by $\sgn( \langle \bx_i, \hat\bmu^{(-i)} \rangle )$, which coincides with the (supervised) linear discriminant analysis \citep{Fis36} given additional labels $\{ y_j \}_{j \neq i}$. This oracle estimator turns out to capture the difficulty of label recovery. That is, $ \sgn( \bG \bar\bu_1 / \bar\lambda_1 ) $ achieves the optimal misclassification rate in (\ref{eqn-intro-rate}).

Above we provide high-level ideas about why the spectral estimator $\sgn(\bu_1)$ is optimal. Inequality (\ref{eqn-intro-Lp})  ties $\bu_1$ and its linearization $\bG \bar\bu_1 / \bar\lambda_1$ together. The latter is connected to the genie-aided estimator through (\ref{eqn-intro-lda}). As a side remark, the relation (\ref{eqn-intro-lda}) hinges on the fact that $\bG$ is hollowed. Otherwise there would be a square term $\langle \bx_i , \bx_i \rangle $ making things entangled.

%
%
%
%


\subsection{Related work}
Early works on PCA focus on classical settings where the dimension $d$ is fixed and the sample size $n$ goes to infinity \citep{And63}. Motivated by modern applications, in the past two decades there has been a surge of interest in high-dimensional PCA. Most papers in this direction study the consistency of empirical eigenvalues \citep{Joh01, BBP05} or Principal Component (PC) directions \citep{Pau07, Nad08, JMa09, BNa12, PWB16, wang2017asymptotics} under various spiked covariance models with $d$ growing with $n$. Similar results are also available for infinite-dimensional Hilbert spaces \citep{KGi00, ZBl06, KLo16}.
The analysis of PCs amounts to showing how the leading eigenvectors of $ \bX^{\top} \bX = \sum_{i=1}^{n} \bx_i \bx_i^{\top} \in \RR^{d\times d}$
recover those of $\EE (\bx_i \bx_i^{\top})$.
When it comes to dimension reduction, one projects the data onto these PCs and get PC scores. This is directly linked to leading eigenvectors of the Gram matrix $\bX \bX^{\top} \in \RR^{n\times n}$. In high-dimensional problems, the $n$-dimensional PC scores may still consistently reveal meaningful structures even if the $d$-dimensional PCs fail to do so \citep{CZh18}.

Analysis of PC scores is crucial to the theoretical study of spectral methods. However, existing results \citep{BBZ07, ARa19} in related areas cannot precisely characterize individual embedded samples under general conditions. This paper aims to bridge the gap by a novel analysis.
In addition, our work is orthogonal to those with sparsity assumptions \citep{JLu09, JWa16}. Here we are concerned with $\mathrm{(i)}$ the non-sparse regime where most components contribute to the main variability and $\mathrm{(ii)}$ the infinite-dimensional regime in kernel PCA where the sparsity assumption is not appropriate.


There is a vast literature on perturbation theories of eigenvectors. Most classical bounds are deterministic and use the $\ell_2$ norm or other orthonormal-invariant norms as error metrics. This includes the celebrated Davis-Kahan theorem \citep{DKa70} and its extensions \citep{Wed72};  see \cite{SSu90} for a review. Improved $\ell_2$-type results are available for stochastic settings \citep{OVW18}.
For many problems in statistics and machine learning, entrywise analysis is more desirable because that helps characterize the spectral embedding of individual samples. \cite{FWZ16}, \cite{EBW17}, \cite{CTP19} and \cite{DSu19} provide $\ell_{\infty}$ perturbation bounds in deterministic settings. Their bounds are often too conservative when the noise is stochastic. Recent papers \citep{KXi16, AFW17, MSC17, ZBo18, CFM19, Lei19} take advantage of the randomness to obtain sharp $\ell_{\infty}$ results for challenging tasks.

The random matrices considered therein are mostly Wigner-type, with independent entries or similar structures. On the contrary, our hollowed Gram matrix $\bG$ has Wishart-type distribution since its off-diagonal entries are inner products of samples and thus dependent. What is more, our $\ell_p$ bounds with $p$ determined by the signal strength are adaptive. If the signal is weak, existing $\ell_{\infty}$ analysis does not go through as strong concentration is required for uniform control of all the entries. However, our $\ell_p$ analysis still manages to control a vast majority of the entries. If the signal is strong, our results imply $\ell_{\infty}$ bounds. The $\ell_p$ eigenvector analysis in this paper shares some features with the study on $\ell_p$-delocalization \citep{ESY09}, yet the settings are very different. It would be interesting to establish further connections.

The applications in this paper are canonical problems in clustering and have been extensively studied. For the sub-Gaussian mixture model, the settings and methods in \cite{GVe18}, \cite{Nda18} and \cite{LZZ19} are similar to ours. The contextual network problem concerns grouping the nodes based on their attributes and pairwise connections, see \cite{BVR17}, \cite{DSM18} and \cite{YSa20} for more about the model. We defer detailed discussions on these to Sections \ref{KPCA-sec-GMM} and \ref{KPCA-sec-CSBM}.

\subsection{Organization of the paper}
We present the general setup and results for $\ell_p$ eigenvector analysis in Section \ref{KPCA-sec-Lp}; apply them to clustering under mixture models in Section \ref{KPCA-sec-GMM} and contextual community detection in Section \ref{KPCA-sec-CSBM}; show a sketch of main proofs in Section \ref{KPCA-sec-outlines}; and conclude the paper with possible future directions in Section \ref{KPCA-sec-discussions}.

\subsection{Notation}

We use $[n]$ to refer to $\{ 1, 2, \cdots, n \}$ for $n \in \ZZ_+$. Denote by $| \cdot |$ the absolute value of a real number or cardinality of a set. For real numbers $a$ and $b$, let $a \wedge b = \min \{ a, b \}$ and $a \vee b = \max \{ a, b \}$. For nonnegative sequences $\{ a_n \}_{n=1}^{\infty}$ and $\{ b_n \}_{n=1}^{\infty}$, we write $a_n \ll b_n$ or $a_n = o(b_n)$ if $b_n > 0$ and $a_n/b_n \to 0$; $a_n \lesssim b_n$ or $a_n = O(b_n)$ if there exists a positive constant $C$ such that $a_n \leq C b_n$; $a_n \gtrsim b_n$ or $a_n = \Omega (b_n)$ if $b_n \lesssim a_n$. In addition, we write $a_n \asymp b_n$ if $a_n \lesssim b_n$ and $b_n \lesssim a_n$. We let $\mathbf{1}_{ S }$ be the binary indicator function of a set $S$.

Let $\{ \be_j \}_{j=1}^d$ be the canonical bases of $\RR^d$, $\SSS^{d-1} = \{ \bx \in \RR^d:~ \| \bx \|_2 = 1 \}$ and $B(\bm{x},r) = \{ \by \in \RR^d:~\| \by - \bx \|_2 \leq r \}$. For a vector $\bx = (x_1,\cdots,x_d)^{\top} \in \RR^d$ and $p \geq 1$, define its $\ell_p$ norm $\| \bx \|_p = (\sum_{i=1}^{d} |x_i|^p )^{1/p}$. For $i \in [d]$, let $\bx_{-i}$ be the $(d-1)$-dimensional sub-vector of $\bx$ without the $i$-th entry. For a matrix $\bA \in \RR^{n\times m}$, we define its spectral norm $\| \bA \|_2 = \sup_{\|\bx\|_2 = 1 } \| \bA \bx \|_2$ and Frobenius norm $\| \bA \|_{\mathrm{F}} = ( \sum_{i,j} a_{ij}^2 )^{1/2}$. Unless otherwise specified, we use $\bA_{i}$ and $\ba_j$ to refer to the $i$-th row and $j$-th column of $\bA$, respectively. For $1 \leq p,q\leq \infty$, we define the $\ell_{q,p}$ norm as an entrywise matrix norm
\begin{align*}
\| \bA \|_{q,p} =\bigg[ \sum_{i=1}^{n}  \bigg( \sum_{j=1}^{m} |a_{ij}|^q \bigg)^{p/q}\bigg]^{1/p}.
\end{align*}
The notation is not to be confused with $(q,p)$-induced norm, which is not used in the current paper. In words, we concatenate the $\ell_q$ norms of the row vectors of $\bA$ into an $n$-dimensional vector and then compute its $\ell_p$ norm. A special case is $\| \bA \|_{2,\infty} = \max_{i\in[n]} \| \bA_i \|_2$.

Define the sub-Gaussian norms $\| X \|_{\psi_2} = \sup_{p \geq 1} \{ p^{-1/2} \EE^{1/p} |X|^p \}$ for random variable $X$ and $\|\bX\|_{\psi_2} = \sup_{\| \bu \|_2 = 1} \|\dotp{\bu}{\bX}\|_{\psi_2} $ for random vector $\bX$. Denote by $\chi^2_n$ refers to the $\chi^2$-distribution with $n$ degrees of freedom. $\overset{\PP}{\rightarrow}$ represents convergence in probability.
In addition, we adopt the following convenient notations from \cite{Wan19} to make probabilistic statements compact\footnote{In the reference above, $O_{\PP}(\cdot;~ \cdot)$ and $o_{\PP}(\cdot;~ \cdot)$ appear as $\hat O_{\PP}(\cdot;~ \cdot)$ and $\hat o_{\PP}(\cdot;~ \cdot)$. For simplicity we drop their hats in this paper.}.

\begin{definition}\label{KPCA-defn-0}
	Let $\{ X_n \}_{n=1}^{\infty}$, $\{ Y_n \}_{n=1}^{\infty}$ be two sequences of random variables and $\{ r_n \}_{n=1}^{\infty} \subseteq (0,+\infty)$ be deterministic. We write
	\begin{align*}
	X_n = O_{\PP} (Y_n;~ r_n)
	\end{align*}
	if there exists a constant $C_1>0$ such that
	\begin{align*}
	\forall C > 0,~~ \exists C' >0 \text{ and } N > 0, \text{ s.t. } \PP ( |X_n| \geq C' |Y_n| ) \leq C_1 e^{- C r_n }, \qquad \forall ~n \geq N.
	\end{align*}
	We write $X_n = o_{\PP} (Y_n;~ r_n)$ if $X_n = O_{\PP} ( w_n Y_n;~ r_n)$ holds for some deterministic sequence $\{ w_n \}_{n=1}^{\infty}$ tending to zero.
\end{definition}

Both the new notation $O_{\PP}(\cdot;~\cdot)$ and the conventional one $O_{\PP}(\cdot)$ help avoid dealing with tons of unspecified constants in operations. Moreover, the former is more informative as it controls the convergence rate of exceptional probabilities. This is particularly useful when we take union bounds over a growing number of random variables. If $\{Y_n\}_{n=1}^{\infty}$ are positive and deterministic, then $X_n=O_{\PP}(Y_n;~1)$ is equivalent to $X_n=O_{\PP}(Y_n)$. Similar facts hold for $o_{\PP}(\cdot;~\cdot)$ as well.

\section{Main results}\label{KPCA-sec-Lp}

\subsection{Basic setup}

Consider the signal-plus-noise model
\begin{align}
\bx_i = \bar\bx_i + \bz_i \in \RR^d,\qquad i \in [n].
\label{Lp-eqn-signal-plus-noise-euc}
\end{align}
For simplicity, we assume that the signals $\{ \bar{\bx}_i \}_{i=1}^n$ are deterministic and the noises $\{ \bz_i \}_{i=1}^n$ are the only source of randomness. The results readily extend to the case where the signals are random but independent of the noises.


Define the \textit{hollowed} Gram matrix $\bG \in \RR^{n\times n}$ of samples $\{ \bx_i \}_{i=1}^n$ through $G_{ij} = \langle \bx_i , \bx_j \rangle \mathbf{1}_{ \{ i \neq j \} }$, and the Gram matrix $\bar\bG \in \RR^{n\times n}$ of signals $\{ \bar\bx_i \}_{i=1}^n$ through $\bar G_{ij} = \langle \bar\bx_i , \bar\bx_j \rangle$.
Denote the eigenvalues of $\bG$ by $\lambda_1 \geq \cdots \geq \lambda_n$ and their associated eigenvectors by $\{ \bu_j \}_{j=1}^n$. Similarly, we define the eigenvalues $\bar\lambda_1 \geq \cdots \geq \bar\lambda_n$ and eigenvectors $\{ \bar\bu_j \}_{j=1}^n$ of $\bar\bG$.
Since $\bar\bG = \bar\bX \bar\bX^{\top} \succeq 0$, we have $\bar\lambda_j \geq 0$ for all $j \in [n]$.
By convention, $\lambda_0 = \bar\lambda_0 = + \infty$ and $\lambda_{n+1} = \bar\lambda_{n+1} = - \infty$. Some groups of eigenvectors may only be defined up to orthonormal transforms as we allow for multiplicity of eigenvalues.

Let $s$ and $r$ be two integers in $[n]$ satisfying  $0 \leq s \leq n - r$. Define $\bU = (\bu_{s+1},\cdots,\bu_{s + r})$, $\bar\bU = (\bar\bu_{s+1},\cdots, \bar\bu_{s + r})$, $\bLambda = \diag (\lambda_{s+1},\cdots,\lambda_{s + r})$ and $\bar\bLambda = \diag ( \bar\lambda_{s+1},\cdots, \bar\lambda_{s + r})$. In order to study how $\bU$ relates to $\bar\bU$, we adopt the standard notion of eigen-gap \citep{DKa70}:
\begin{align}
\bar{\Delta} = \min\{  \bar\lambda_{s} - \bar\lambda_{s+1} , \bar\lambda_{s+r} - \bar\lambda_{s+r+1} \}.
\end{align}
This is the separation between the set of target eigenvalues $\{ \bar\lambda_{s + j} \}_{j=1}^r$ and the rest, reflecting the signal strength. Define $\kappa = \bar\lambda_1 / \bar\Delta$, which plays the role of condition number.
Most importantly, we use a parameter $\gamma$ to characterize the signal-to-noise ratio and impose the regularity assumptions below. It is worth mentioning that we consider the asymptotic setting $n\to\infty$ throughout the paper to make the results clean and easy to read. They can be easily translated to finite-sample versions similar to those in \cite{AFW17}, since our tools such as concentration inequalities and spectral perturbation bounds are non-asymptotic by nature.

\begin{assumption}[Incoherence]\label{KPCA-assumption-incoherence-euc}
As $n \to \infty$ we have
	\begin{align*}
	\kappa \mu \sqrt{\frac{r}{n} } \leq \gamma \ll \frac{1}{\kappa \mu}
	\qquad \text{where} \qquad
	\mu = \max \bigg\{ \frac{\| \bar\bX \|_{2, \infty}}{\| \bar\bX \|_{2}} \sqrt{\frac{n}{r}} , ~ 1  \bigg\}.
	\end{align*}	
\end{assumption}
\begin{assumption}[Sub-Gaussianity]\label{KPCA-assumption-noise-euc}
	$\{ \bz_i \}_{i=1}^n$ are independent, zero-mean random vectors in $\RR^d$. There exists a constant $\alpha > 0$ and $\bSigma \succeq 0$ such that $\EE e^{\langle \bu,\bz_i \rangle} \leq e^{ \alpha^2 \langle \bSigma \bu, \bu \rangle / 2 }$ holds for all $\bu \in \RR^d$ and $i\in [n]$.
\end{assumption}
\begin{assumption}[Concentration]\label{KPCA-assumption-concentration-euc}
	$\sqrt{n} \max\{ (  \kappa \| \bSigma \|_{2} / \bar\Delta )^{1/2},~  \| \bSigma \|_{\mathrm{F}} / \bar\Delta \} \leq \gamma$.
\end{assumption}

By construction, $\bar\bX = (\bar\bx_1,\cdots,\bar\bx_n)^{\top}$ and $\| \bar\bX \|_{2,\infty} = \max_{i\in[n]} \| \bar\bx_i \|_2$. Assumption \ref{KPCA-assumption-incoherence-euc} regulates the magnitudes of $\{ \| \bar\bx_i \|_2 \}_{i=1}^n$ in order to control the bias induced by the hollowing step. It naturally holds under various mixture models. The incoherence parameter $\mu$ is similar to the usual definition \citep{CRe09} except for the facts that $ \bar\bX $ does not have orthonormal columns and $r$ is not its rank. When $r = 1$, we have $\mu = \sqrt{n} \| \bar\bX \|_{2, \infty} / \| \bar\bX \|_{2} \geq 1 $.
Assumption \ref{KPCA-assumption-noise-euc}  is a standard one on sub-Gaussianity \citep{KLo14}. Here $\{ \bz_i \}_{i=1}^n$ are independent but may not have identical distributions, which allows for heteroscedasticity. Assumption \ref{KPCA-assumption-concentration-euc} governs the concentration of $\bG$ around its population version $\bar\bG$.
To gain some intuition, we define $\bZ = (\bz_1,\cdots,\bz_n)^{\top} \in \RR^{n \times d}$ and observe that
\begin{align*}
\bG
& = \cH [ (\bar\bX + \bZ)  (\bar\bX + \bZ)^{\top} ] = \cH ( \bar\bX \bar\bX^{\top} ) + \cH( \bar\bX\bZ^{\top} + \bZ \bar\bX^{\top} ) + \cH( \bZ\bZ^{\top} ) \notag \\
& = \bar\bX \bar\bX^{\top} + ( \bar\bX\bZ^{\top} + \bZ \bar\bX^{\top} ) + \cH( \bZ\bZ^{\top} )  - \bar\bD,
\end{align*}
where $\bar\bD$ is the diagonal part of $\bar\bX \bar\bX^{\top} + \bar\bX\bZ^{\top} + \bZ \bar\bX^{\top}$. Hence
\begin{align*}
\| \bG - \bar\bG \|_2 \leq  \| \bar\bX\bZ^{\top} + \bZ \bar\bX^{\top} \|_2 + \| \cH( \bZ\bZ^{\top} ) \|_2 + \max_{i \in [n]} | (\bar\bX \bar\bX^{\top} + \bar\bX\bZ^{\top} + \bZ \bar\bX^{\top})_{ii} |.
\end{align*}
The individual terms above are easy to work with. For instance, we may control $\| \cH( \bZ\bZ^{\top} ) \|_2$ using concentration bounds for random quadratic forms such as Hanson-Wright-type inequalities \citep{CYa18}. The spectral and Frobenius norms of $\bSigma$ collectively characterize the effective dimension of the noise distribution. That gives the reason why Assumption \ref{KPCA-assumption-concentration-euc} is formulated as it is. It turns out that Assumptions \ref{KPCA-assumption-incoherence-euc}, \ref{KPCA-assumption-noise-euc} and \ref{KPCA-assumption-concentration-euc} lead to a matrix concentration bound $\| \bG - \bar\bG \|_2 = O_{\PP} ( \gamma \bar\Delta;~ n)$, paving the way for eigenvector analysis. Hence $\gamma^{-1}$ measures the signal strength, similar to the quantity in \cite{AFW17}.

%
%
%


\subsection{$\ell_{2,p}$ analysis of eigenspaces}

Note that $\{ \bu_{s+j} \}_{j=1}^r$ and $\{ \bar\bu_{s+j} \}_{j=1}^r$ are only identifiable up to sign flips, and things become even more complicated if some eigenvalues are identical. To that end, we need to align $\bU$ with $\bar\bU$ using certain orthonormal transform.
Define $\bH = \bU^{\top} \bar\bU \in \RR^{r\times r}$ and let $\tilde\bU \tilde\bLambda  \tilde\bV^{\top}$ denote its singular value decomposition, where $\tilde\bU, \tilde\bV \in \cO_{r\times r}$ and $\tilde\bLambda \in \RR^{r\times r}$ is diagonal with nonnegative entries. The orthonormal matrix $\tilde\bU \tilde \bV^\top $, denoted by the matrix sign function $\sgn ( \bH )$ in the literature \citep{Gro11}, is the best rotation matrix that aligns $\bU$ with $\bar \bU$ and will play an important role throughout our analysis. In addition, define $\bZ = (\bz_1,\cdots,\bz_n)^{\top} \in \RR^{n\times d}$ as the noise matrix.   Recall that for $\bA \in \RR^{n\times r}$ with row vectors $\{ \bA_i \}_{i=1}^n $, the $\ell_{2,p}$ norm is
\begin{align*}
\| \bA \|_{2,p} = \bigg( \sum_{i=1}^n \| \bA_{i} \|_2^p \bigg)^{1/p}.
\end{align*}

\begin{theorem}\label{KPCA-corollary-main}
Suppose that Assumptions \ref{KPCA-assumption-incoherence-euc}, \ref{KPCA-assumption-noise-euc} and \ref{KPCA-assumption-concentration-euc} hold. As long as $2 \leq p \lesssim (\mu\gamma)^{-2}$, we have
\begin{align*}
	&	\| \bU \sgn(\bH) - \bG\bar\bU \bar\bLambda^{-1}  \|_{2,p} = o_{\PP} ( \| \bar\bU \|_{2,p};~p),\\
	& \| \bU \sgn(\bH) - [ \bar\bU + \cH ( \bZ \bX^{\top} ) \bar\bU \bar\bLambda^{-1} ] \|_{2,p} = o_{\PP} ( \| \bar\bU \|_{2,p};~p),\\
	&	\| \bU \sgn(\bH) \|_{2,p} = O_{\PP} ( \| \bar\bU \|_{2,p};~p).
	\end{align*}
In addition, if $\kappa^{3/2} \gamma \ll 1$, then
\begin{align*}
&	\| \bU \bLambda^{1/2} \sgn(\bH) - \bG\bar\bU \bar\bLambda^{-1/2}  \|_{2,p} = o_{\PP} ( \| \bar\bU \|_{2,p} \| \bar\bLambda^{1/2} \|_2 ;~p) ,\\
&	\| \bU \bLambda^{1/2} \sgn(\bH) - [ \bar\bU \bar\bLambda^{1/2} + \cH (\bZ \bX^{\top}) \bar{\bU} \bar\bLambda^{-1/2} ] \|_{2,p} = o_{\PP} ( \| \bar\bU \|_{2,p} \| \bar\bLambda^{1/2} \|_2 ;~p).
\end{align*}
\end{theorem}

The first equation in Theorem \ref{KPCA-corollary-main} asserts that although $\bU$ is a highly nonlinear function of $\bG$, it can be well-approximated by a linear form $\bG \bar\bU \bar\bLambda^{-1}$ up to an orthonormal transform. This can be understood from the hand-waving deduction:
\begin{align*}
\bU = \bG \bU \bLambda^{-1} \approx \bG \bar\bU \bar\bLambda^{-1}.
\end{align*}
The second equation in Theorem \ref{KPCA-corollary-main} talks about the difference between $\bU$ and its population version $\bar\bU$. Ignoring the orthonormal transform $\sgn(\bH)$, we have that for a large fraction of $m \in [n]$, the following entrywise approximation holds
\begin{align}
\bU_m \approx
[\bar\bU + \cH ( \bZ \bX^{\top} ) \bar\bU \bar\bLambda^{-1}]_m
= \bar\bU_{m} + \bigg\langle \bz_m,  \sum_{j \neq m} \bx_j  \bar\bU_{j} \bar\bLambda^{-1} \bigg\rangle .
\label{KPCA-eqn-entrywise}
\end{align}
If we keep $\{ \bx_j \}_{j \neq m}$ fixed,  then the spectral embedding $\bU_m$ for the $m$-th sample is roughly linear in $\bz_m$ or equivalently $\bx_m$ itself. This relation is crucial for our analysis of spectral clustering algorithms.
The third equation in Theorem \ref{KPCA-corollary-main} relates to the delocalization property of $\bU$ to that of $\bar\bU$, showing that the mass of $\bU$ is spread out across its rows as long as $\bar\bU$ behaves in a similar way.

Many spectral methods use the rows of $\bU \in \RR^{n\times r}$ to embed the samples $\{ \bx_i \}_{i=1}^n \subseteq \RR^d$ into $\RR^r$ \citep{SMa00, NJW02} and perform downstream tasks. By precisely characterizing the embedding, the first three equations in Theorem \ref{KPCA-corollary-main} facilitate their analysis  under statistical models. 
In PCA, however, the embedding is defined by PC scores. Recall that the PCs are eigenvectors of the covariance matrix $\frac{1}{n} \bX^{\top} \bX \in \RR^{d \times d}$ and PC scores are derived by projecting the data onto them. Therefore, the PC scores in our setting correspond to the rows of $\bU \bLambda^{1/2}$ rather than $\bU$. The last two equations in Theorem \ref{KPCA-corollary-main} quantify their behavior. 

Theorem \ref{KPCA-corollary-main} is written to be easily applicable. It forms the basis of our applications in Sections \ref{KPCA-sec-GMM} and \ref{KPCA-sec-CSBM}. General results under relaxed conditions are given by Theorem \ref{KPCA-theorem-main}.

Let us now gain some intuition about the $\ell_{2,p}$ error metric. For large $p$, 
$\| \bA \|_{2,p} $ is small if a vast majority of the rows have small $\ell_2$ norms, but there could be a few rows that are large. Roughly speaking, the number of those outliers is exponentially small in $p$. We illustrate this using a toy example with $r = 1$, i.e., $\bA = \bx \in \RR^n$ is a vector and $\| \cdot \|_{2,p} = \| \cdot \|_p$. If $\| \bx \|_{p} \leq \varepsilon \| \mathbf{1}_n \|_{p}$ for some $\varepsilon > 0$, then Markov's inequality yields
\begin{align*}
\frac{1}{n} |\{ i:~ |x_i| > t \varepsilon \}| \leq 
\frac{ n^{-1} \| \bx \|_p^p }{(t \varepsilon)^p}
\leq \frac{ n^{-1} \varepsilon^p \| \mathbf{1}_n \|_p^p }{(t \varepsilon)^p} = t^{-p}, \qquad  \forall t > 0.
\end{align*}
Larger $p$ implies stronger bounds. In particular, the following fact states that when $p \gtrsim \log n$, an upper bound in $\ell_{2,p}$ yields one in $\ell_{2,\infty}$, controlling all the row-wise errors simultaneously.
\begin{fact}\label{KPCA-fact-p-log}
	$\| \bx \|_{\infty} \leq \| \bx \|_{c \log n} \leq e^{1/c} \| \bx \|_{\infty}$ for any $ n \in \ZZ_+$, $\bx \in \RR^n$, $c > 0$.
\end{fact}
Fact \ref{KPCA-fact-p-log} immediately follows from the relation
\[
\| \bx \|_{\infty} \leq \| \bx \|_p = \bigg(\sum_{i=1}^{n} |x_i|^p\bigg)^{1/p} \leq (n \| \bx \|_{\infty}^p)^{1/p} = n^{1/p} \| \bx \|_{\infty}, \qquad\forall p \geq 1.
\]
Recall that in \Cref{KPCA-corollary-main} we require $\mu \gamma \to 0$ (Assumption \ref{KPCA-assumption-incoherence-euc}) but the convergence rate can be arbitrarily slow. The largest $p$ is of order $(\mu\gamma)^{-2}$. Now we consider a stronger condition $\mu \gamma \lesssim 1 / \sqrt{\log n}$ so that we can take $p \asymp \log n$ and obtain $\ell_{2,\infty}$ approximation bounds.

\begin{corollary}\label{KPCA-corollary-main-inf}
	Suppose that Assumptions \ref{KPCA-assumption-incoherence-euc}, \ref{KPCA-assumption-noise-euc} and \ref{KPCA-assumption-concentration-euc} hold. As long as $\mu\gamma \lesssim 1 / \sqrt{\log n}$, we have
	\begin{align*}
	&	\| \bU \sgn(\bH) - \bG\bar\bU \bar\bLambda^{-1}  \|_{2,\infty} = o_{\PP} ( \| \bar\bU \|_{2,\infty};~\log n),\\
	& \| \bU \sgn(\bH) - [ \bar\bU + \cH ( \bZ \bX^{\top} ) \bar\bU \bar\bLambda^{-1} ] \|_{2,\infty} = o_{\PP} ( \| \bar\bU \|_{2,\infty};~\log n),\\
	&	\| \bU \sgn(\bH) \|_{2,\infty} = O_{\PP} ( \| \bar\bU \|_{2,\infty};~\infty).
	\end{align*}
	In addition, if $\kappa^{3/2} \gamma \ll 1$, then
	\begin{align*}
	&	\| \bU \bLambda^{1/2} \sgn(\bH) - \bG\bar\bU \bar\bLambda^{-1/2}  \|_{2,\infty} = o_{\PP} ( \| \bar\bU \|_{2,\infty} \| \bar\bLambda^{1/2} \|_2 ;~\log n) , \\
&	\| \bU \bLambda^{1/2} \sgn(\bH) - [ \bar\bU \bar\bLambda^{1/2} + \cH (\bZ \bX^{\top}) \bar{\bU} \bar\bLambda^{-1/2} ] \|_{2,\infty} = o_{\PP} ( \| \bar\bU \|_{2,\infty} \| \bar\bLambda^{1/2} \|_2 ;~\log n).
	\end{align*}
\end{corollary}

However, $p$ cannot be arbitrarily large in general. When the signal is weak, we can no longer obtain uniform error bounds as the above and should allow for exceptions. The quantity $(\mu\gamma)^{-2}$ in Theorem \ref{KPCA-corollary-main} measures the signal strength, putting a cap on the largest $p$ we can take. That makes the results adaptive and superior to the $\ell_{2,\infty}$ ones in \citep{AFW17}.



\subsection{Extension to Hilbert spaces}

Since $\bG \in \RR^{n\times n}$ is constructed purely based on pairwise inner products of samples, the whole procedure can be extended to kernel settings. Here we briefly discuss the kernel PCA \citep{SSM97}. Suppose that $\{ \bx_i \}_{i=1}^n $ are samples from some space $\cX$ and $K(\cdot,\cdot):~\cX \times \cX \to \RR$ is a symmetric and positive semi-definite kernel.
The kernel PCA is PCA based on a new Gram matrix $\bK \in \RR^{n\times n}$ with $K_{ij} = K(\bx_i , \bx_j)$. PCA is a special case of kernel PCA with $\cX = \RR^d$ and $K(\bx,\by) = \bx^{\top} \by$. Commonly-used nonlinear kernels include the Gaussian kernel $K(\bx,\by) = e^{- \eta \| \bx - \by \|_2^2 }$ with $\eta > 0$ and polynomial kernel. They offer flexible nonlinear embedding techniques which have achieved great success in machine learning \citep{CSh00}.

According to the Moore-Aronszajn Theorem \citep{Aro50}, there exists a reproducing kernel Hilbert space $\HH$ with inner product $\langle \cdot,\cdot \rangle $ and a function $\phi:~ \cX \to \HH$ such that $K(\bx, \by) =\langle \phi(\bx), \phi(\by) \rangle $ for any $\bx,\by \in \cX$. Hence, kernel PCA of $\{ \bx_i \}_{i=1}^n \subseteq \cX$ is de facto PCA of transformed data $\{ \phi(\bx_i) \}_{i=1}^n \subseteq \HH$.
The transform $\phi$ can be rather complicated since $\HH$ has infinite dimensions in general. Fortunately, the inner products $\{ \langle \phi(\bx_i) , \phi(\bx_j)  \rangle \}$ in $\HH$ can be conveniently computed in the original space $\cX$, which is $K_{ij}$.

Motivated by the kernel PCA, we extend the basic setup to Hilbert spaces. Let $\mathbb{H}$ be a real separable Hilbert space with inner product $\langle \cdot,\cdot \rangle $, norm $\| \cdot \|$, and orthonormal bases $\{ \bh_j \}$.
\begin{definition}[Basics of Hilbert spaces]
	A linear operator $\bA:\HH \to \HH$ is said to be bounded if its operator norm $\| \bA \|_{\mathrm{op}} = \sup_{ \| \bu \| = 1 } \| \bA \bu \|$ is finite.
	Define $\cL (\HH)$ as the collection of all bounded linear operators over $\mathbb{H}$. For any $\bA \in \cL (\HH)$, we use $\bA^*$ to refer to its adjoint operator and let $\Tr ( \bA ) = \sum_{j} \langle \bA \bh_j , \bh_j \rangle$.   Define
	\begin{align*}
	\cS_{+} (\HH) = \{ \bA \in \cL (\HH):~~\bA = \bA^*,~~\langle \bA \bx,\bx \rangle \geq 0,~\forall \bx \in \mathbb{H} \text{ and } \Tr ( \bA ) < \infty \}.
	\end{align*}
	Any $\bA \in \cS_{+}  (\HH)$ is said to be positive semi-definite. We use
	$\| \bA \|_{\mathrm{HS}} = \sqrt{ \Tr ( \bA^* \bA ) } = ( \sum_{j} \| \bA \bh_j \|^2 )^{1/2}$
	to refer to its Hilbert-Schmidt norm, and define $\bA^{1/2} \in \cT ( \mathbb{H} )$ as the unique operator such that $\bA^{1/2} \bA^{1/2} = \bA$.
\end{definition}

\begin{remark}
	When $\mathbb{H} = \RR^d$, we have $\cL(\cH) = \RR^{d\times d}$, $\Tr(\bA) = \sum_{i=1}^{d} A_{ii}$, $\| \cdot \|_{\mathrm{op}} = \| \cdot \|_2$ and $\| \cdot \|_{\mathrm{HS}} = \| \cdot \|_{\mathrm{F}}$. Further, $\cS_{+} (\HH)$ consists of all $d\times d$ positive semi-definite matrices.
\end{remark}

We now generalize model (\ref{Lp-eqn-signal-plus-noise-euc}) to the following one in $\HH$:
\begin{align}
\bx_i = \bar\bx_i + \bz_i \in \HH,\qquad i \in [n].
\label{Lp-eqn-signal-plus-noise}
\end{align}
When $\HH = \RR^d$, the data matrix $\bX = (\bx_1,\cdots,\bx_n)^{\top} \in \RR^{n \times d}$ corresponds to a linear transform from $\RR^d$ to $\RR^n$. For any general $\HH$, we can always define $\bX$ as a bounded linear operator from $\HH$ to $\RR^n$ through its action $\bh \mapsto (\langle \bx_1, \bh \rangle, \cdots, \langle \bx_n, \bh \rangle)$.
With slight abuse of notation, we formally write $\bX = (\bx_1,\cdots,\bx_n)^{\top}$, use $\| \bX \|_{\mathrm{op}}$ to refer to its norm, let $\| \bX \|_{2,\infty} = \max_{i\in[n]} \| \bx_i \|$, and do the same for $\bar\bX$ and $\bZ$. We generalize Assumptions \ref{KPCA-assumption-incoherence-euc}, \ref{KPCA-assumption-noise-euc} and \ref{KPCA-assumption-concentration-euc} accordingly.
\begin{assumption}[Incoherence]\label{KPCA-assumption-incoherence}
	As $n \to \infty$ we have
	\begin{align*}
	\kappa \mu \sqrt{\frac{r}{n} } \leq \gamma \ll \frac{1}{\kappa \mu}
	\qquad \text{where} \qquad
	\mu = \max \bigg\{ \frac{\| \bar\bX \|_{2, \infty}}{\| \bar\bX \|_{\mathrm{op}}} \sqrt{\frac{n}{r}} , ~ 1  \bigg\}.
	\end{align*}	
\end{assumption}

\begin{assumption}[Sub-Gaussianity]
	\label{KPCA-assumption-noise}
	$\{ \bz_i \}_{i=1}^n$ are independent, zero-mean random vectors in $\mathbb{H}$. There exists a constant $\alpha > 0$ and an operator $\bSigma \in \cT  (\HH)$ such that $\EE e^{\langle \bu,\bz_i \rangle} \leq e^{ \alpha^2 \langle \bSigma \bu, \bu \rangle / 2 }$ holds for all $\bu \in \mathbb{H}$ and $i\in [n]$.
\end{assumption}

\begin{assumption}[Concentration]\label{KPCA-assumption-concentration}
	$\sqrt{n} \max\{ (  \kappa \| \bSigma \|_{\mathrm{op}} / \bar\Delta )^{1/2},~  \| \bSigma \|_{\mathrm{HS}} / \bar\Delta \} \leq \gamma$.
\end{assumption}

Again, Assumption \ref{KPCA-assumption-incoherence} on incoherence holds for various mixture models. Assumption \ref{KPCA-assumption-noise} appears frequently in the study of sub-Gaussianity in Hilbert spaces \citep{KLo14}. For kernel PCA, Assumption \ref{KPCA-assumption-noise} automatically holds when the kernel is bounded, i.e. $K(\bx,\bx) \leq C$ for some constant $C$. Assumption \ref{KPCA-assumption-concentration} naturally arises in the study of Gram matrices and quadratic forms in Hilbert spaces \citep{CYa18}.
The same results in Theorem \ref{KPCA-corollary-main} continue to hold under the Assumptions \ref{KPCA-assumption-incoherence}, \ref{KPCA-assumption-noise} and \ref{KPCA-assumption-concentration}.
The proof is in Appendix \ref{KPCA-corollary-main-proof}.

\section{Mixture models}\label{KPCA-sec-GMM}

\subsection{Sub-Gaussian mixture model}\label{KPCA-sec-GMM-sGMM}

Sub-Gaussian and Gaussian mixture models serve as testbeds for clustering algorithms.
Maximum likelihood estimation requires well-specified models and often involves non-convex or combinatorial optimization problems that are hard to solve. The recent years have seen a boom in the study of efficient approaches.
The Lloyd's algorithm \citep{Llo82} with good initialization and its variants are analyzed under certain separation conditions \citep{KKa10, LZh16, Nda18, GZh19}.
Semi-definite programming (SDP) yields reliable results in more general scenarios \citep{ABC15, MVW17, Roy17, FCh18, GVe18, CYa18, CYa20}.
Spectral methods 
are more efficient in terms of computation and have attracted much attention \citep{VWa04, CZh18, LZZ19, SSH19}. However, much less is known about spectral methods compared with SDP.

We apply the $\ell_{p}$ theory of PCA to spectral clustering under a sub-Gaussian mixture model in a Hilbert space $\HH$. Suppose that we collect samples $\{ \bx_i \}_{i=1}^n \subseteq \HH$ from a mixture model
\begin{align}
\bx_i = \bmu_{y_i} + \bz_i , \qquad i \in [n].
\label{eqn-model-sGMM}
\end{align}
Here $\{ \bmu_j \}_{j=1}^K \subseteq \HH$ are cluster centers, $\{ y_i \}_{i=1}^n \subseteq [K]$ are true labels, and $\{ \bz_i \}_{i=1}^n \subseteq \HH$ are noise vectors satisfying Assumption \ref{KPCA-assumption-noise}. For simplicity, we assume that the centers and labels are deterministic. A conditioning argument extends the results to the case where they are independent of $\{ \bz_i \}_{i=1}^n$. Heteroscedasticity is allowed, as the covariance matrices of $\{ \bz_i \}_{i=1}^n$ may be different as long as they are uniformly dominated by some $\bSigma$.
The goal of clustering is to recover $\{ y_i \}_{i=1}^n$ from $\{ \bx_i \}_{i=1}^n$. Below is the spectral algorithm under investigation, based on PCA and approximate $k$-means. Here $r \in [K]$ is the target dimension of embedding.
\begin{enumerate}
	\item Compute the $r$ leading eigenvalues $\{ \lambda_j \}_{j=1}^r$ and their associated eigenvectors $\{ \bu_j \}_{j=1}^r$ of the hollowed Gram matrix $\cH(\bX \bX^{\top})$. Let $\bU = ( \bu_1,\cdots,\bu_r )$ and $\bLambda = \diag (\lambda_1,\cdots,\lambda_r)$.
	\item Conduct $(1+\varepsilon)$-approximate $k$-means clustering on the rows of $\bU \bLambda^{1/2}$, getting $\{ \hat\bmu_j \}_{j=1}^K \subseteq \RR^{r}$ and $\{ \hat{y}_i \}_{i=1}^n \subseteq [K]$ such that
	\begin{align*}
	\sum_{i=1}^{n} \| (\bU \bLambda^{1/2})_i - \hat\bmu_{\hat y_i} \|_2^2
	\leq (1 + \varepsilon) \min_{
		\substack{
			\{ \widetilde\bmu_j \}_{j=1}^K \subseteq \RR^K \\
			\{ \widetilde y_i \}_{i=1}^n \subseteq [K]
		}
	} \bigg\{ \sum_{i=1}^{n} \| (\bU \bLambda^{1/2})_i - \widetilde\bmu_{\widetilde y_i} \|_2^2 \bigg\} .
	\end{align*}
Return $\{ \hat{y}_i \}_{i=1}^n$ as the estimated labels.
\end{enumerate}

The rows of the PC score matrix $\bU \bLambda^{1/2}$ embed the $n$ samples into $\RR^r$, which greatly reduces the dimensionality. When $\varepsilon = 0$, $( \{ \hat\bmu_j \}_{j=1}^K , \{ \hat{y}_i \}_{i=1}^n )$ is an exact solution to the $k$-means program but may be NP-hard to compute in the worst case. Fortunately, for any constant $\varepsilon > 0$ there exists a linear-time algorithm that returns a $(1+\varepsilon)$-approximate solution \citep{KSS04}. 
In that case, the spectral algorithm above is computationally efficient as both steps run in nearly linear time. Our theory handles any constant $\varepsilon \geq 0$.

Define the misclassification rate of $\hat{\by} \in [K]^n$ as
\begin{align}
& \cM ( \hat{\by} , \by ) = n^{-1} \min_{\tau \in S_K}  |\{ i \in [n]:~ \hat{y}_i \neq \tau(y_i )  \}| .
\label{eqn-kmeans-error-rate}
\end{align}
Here $S_K$ is the set of all permutations of $[K]$. We will derive sharp bounds on $\EE  \cM ( \hat{\by} , \by )$. To facilitate presentation and highlight key ideas, we assume that $K,r$ are given and make some regularity assumptions. Estimation of $K$ and $r$ in general scenarios is left for future work.

\begin{assumption}[Regularities]\label{sGMM-assumption-regularity}
Let $\bB \in \RR^{K\times K}$ be the Gram matrix of $\{ \bmu_j \}_{j=1}^K$ with $B_{ij} = \langle \bmu_i, \bmu_j \rangle$. Suppose that $\mathrm{rank}(\bB) = r$ and there is a constant $\kappa_0$ that bounds
\begin{align*}
\frac{n}{ \min_{k \in [K]} |\{ i \in [n]:~y_i = k \}| }
,\qquad
\frac{ \lambda_1(\bB) }{  \lambda_r ( \bB) } 
\qquad\text{and}\qquad
\frac{ \max_{j \in [K]} \|  \bmu_j \| }{\min_{i \neq j} \|\bmu_i - \bmu_j \| }  
\end{align*}
from above. Here $\lambda_j(\cdot)$ denotes the $j$-th largest eigenvalue of a symmetric matrix.
\end{assumption}

The lower bound on the smallest cluster forces $K \leq \kappa_0$. Hence $K$ is a constant and all clusters have comparable sizes. 
The spectral assumption on $\bB$ holds if $\{ \bmu_j \}_{j=1}^K$ span a subspace of dimension $r \leq K$ but do not concentrate near any smaller subspace. Such condition is commonly used in the study of spectral methods for mixture models \citep{HKa13}. 
The last regularity condition in Assumption \ref{sGMM-assumption-regularity} is likely an artifact of proof. Our current results on the empirical embedding $\bU \bLambda^{1/2}$ (\Cref{KPCA-corollary-main}) controls its deviation from the truth using $\| \bar\bLambda \|_2$, which is related to $\max_{j \in [K]} \| \bmu_j \|$. We need such deviation to be smaller than the minimum separation $\min_{i \neq j} \|\bmu_i - \bmu_j \| $ in order to ensure the accuracy of clustering.

Before presenting the general results, we illustrate Assumption \ref{sGMM-assumption-regularity} by two examples. Suppose that $\HH = \RR^d$ for $d \geq K = 3$ and the 3 clusters are equally-sized. When $\bmu_j = \be_j$ for all $j \in \{ 1, 2, 3 \}$, we have $\bB =  \bI_3$, $r = 3$ and $\kappa_0 = 3$. When $\bmu_1 = \be_1$, $\bmu_2 = ( -1/2 , \sqrt{3}/2, 0)^{\top}$ and $\bmu_3 = ( -1/2 , -\sqrt{3}/2, 0)^{\top}$, we have $\bB = (3/2) \bI_3 - (1/2) \bm{1}_{3\times 3}$, $r = 2$ and $\kappa_0 = 3$.

\begin{theorem}\label{thm-sGMM-kmeans}
Consider the mixture model (\ref{eqn-model-sGMM}). Let Assumptions \ref{KPCA-assumption-noise}, \ref{sGMM-assumption-regularity} hold and $\varepsilon \geq 0$ be a constant. Define
$\bar{s}  = \min_{i \neq j} \|\bmu_i - \bmu_j \| $ and
\begin{align}
\mathrm{SNR} = \min \bigg\{ 
\frac{ \bar{s}^2 }{ \| \bSigma \|_{\mathrm{op}} },
\frac{n \bar{s}^4 }{ \| \bSigma \|_{\mathrm{HS}}^2 } 
\bigg\}.
\label{eqn-kmeans-SNR}
\end{align}
There exist constants $C > 0$ and $c >0$ such that the followings hold:
	\begin{enumerate}
	\item If $\mathrm{SNR} > C \log n$, then $\lim_{n \to \infty} \PP [ \cM ( \hat{\by} , \by ) = 0 ] = 1$;
	\item If $1 \ll \mathrm{SNR} \leq C \log n$, then $\limsup_{n\to\infty} \mathrm{SNR}^{-1} \log \EE \cM ( \hat{\by} , \by ) < - c$.
\end{enumerate}
\end{theorem}

The proof is in \Cref{sec-thm-sGMM-kmeans-proof}. Theorem \ref{thm-sGMM-kmeans} asserts that the spectral algorithm exactly recovers all the labels with high probability when $\mathrm{SNR}$ exceeds some constant multiple of $\log n$. When $\mathrm{SNR}$ is not that large but still diverges, we have an exponential bound $e^{-\Omega (\mathrm{SNR})}$ for the misclassification rate. To understand why the quantity $\mathrm{SNR}$ in (\ref{eqn-kmeans-SNR}) measures the signal-to-noise ratio, note that 
\begin{align}
\mathrm{SNR} & \asymp 
\frac{ \min_{j \neq k} \| \bmu_j - \bmu_k \|^2 }{ \| \bSigma \|_{\mathrm{op}} } \cdot \min \bigg\{ 
1,~
\frac{ \min_{j \neq k} \| \bmu_j - \bmu_k \|^2 }{ \| \bSigma \|_{\mathrm{op}} }
\cdot \frac{n}{r(\bSigma) }
\bigg\} .
\label{eqn-kmeans-SNR-2}
\end{align}
Here $r(\bSigma)  = \| \bSigma \|_{\mathrm{HS}}^2 / \| \bSigma \|_{\mathrm{op}}^2$ captures the effective rank of $\bSigma$. In the isotropic case with $\HH = \RR^d$ and $\bSigma = \bI_d$, we have $r(\bSigma) = d$. Thus $\mathrm{SNR}$ characterizes the strength of signal relative to the noise, together with the effect of dimension. It is equivalent to the signal-to-noise ratio in \cite{GVe18} and \cite{CYa18} when $K = O(1)$. 

The $\mathrm{SNR}$ differs from the classical notion of signal-to-noise ratio 
\begin{align}
\frac{ \min_{j \neq k} \| \bmu_j - \bmu_k \|^2 }{ \| \bSigma \|_{\mathrm{op}} }.
\label{eqn-kmeans-SNR-old}
\end{align}
frequently used for quantifying the misclassification rates \citep{LZh16, FCh18, LZZ19, SSH19, GZh19}. Those results hinge on an extra assumption
\begin{align}
\frac{ \min_{j \neq k} \| \bmu_j - \bmu_k \|^2 }{ \| \bSigma \|_{\mathrm{op}} } 
 \gg \max \bigg\{ 1,~ \frac{r(\bSigma)}{n} \bigg\},
\label{KPCA-eqn-def-snr-2}
\end{align}
or the one with $\gg$ replaced by $\gtrsim$. In that case, (\ref{eqn-kmeans-SNR-2}) shows that our $\mathrm{SNR}$ is equivalent to the classical one in (\ref{eqn-kmeans-SNR-old}). The error bound $\EE \cM (\hat\by, \by) = e^{-\Omega(\mathrm{SNR})}$ and the condition $\mathrm{SNR} = \Omega(\log n)$ for exact recovery in \Cref{thm-sGMM-kmeans} are optimal \citep{FCh18}.

In general, our assumption $\mathrm{SNR} \gg 1$ in \Cref{thm-sGMM-kmeans} translates to
\begin{align}
\frac{ \min_{j \neq k} \| \bmu_j - \bmu_k \|^2 }{ \| \bSigma \|_{\mathrm{op}} } 
\gg \max\bigg\{ 1,~
\sqrt{\frac{r(\bSigma)}{n}} \bigg\}.
\label{KPCA-eqn-def-snr-3-new}
\end{align}
It is weaker than (\ref{KPCA-eqn-def-snr-2}) when the noise has high effective dimensions $r(\bSigma) \gg n$. 

For the sub-Gaussian mixture model (\ref{eqn-model-sGMM}) with the regularity Assumption \ref{sGMM-assumption-regularity}, the results in \Cref{thm-sGMM-kmeans} are the best available in the literature. They have only been established for an SDP relaxation of $k$-means under sub-Gaussian mixture models in Euclidean spaces \citep{GVe18} and Hilbert spaces \citep{CYa18}. Our analysis of spectral method is powered by the $\ell_{2,p}$ approximation of the PC score matrix $\bU \bLambda^{1/2}$ in \Cref{KPCA-corollary-main}.
It would be interesting to precisely characterize the constants $C$ and $c$ in \Cref{thm-sGMM-kmeans}, relax the regularity conditions in Assumption \ref{sGMM-assumption-regularity}, and investigate the optimality of spectral method in more general regimes.

\cite{LZZ19} study the spectral algorithm without the hollowing step under the isotropic Gaussian mixture model, with $\HH = \RR^d$, $\bSigma = \bI_d$ and $\{ \bz_i \}_{i=1}^n$ being i.i.d.~from $N(\bm{0}, \bI_d )$. They prove an error bound that is exponential in $\min_{j \neq k} \| \bmu_j - \bmu_k \|^2$, with a sharp constant factor in the exponent. However, as we mentioned above, they require a strong condition (\ref{KPCA-eqn-def-snr-2}). On the other hand, our \Cref{thm-sGMM-kmeans} covers a much broader class of sub-Gaussian mixtures in Hilbert spaces. It only involves the effective dimension instead of the ambient one, which is possibly infinite. Our requirement (\ref{KPCA-eqn-def-snr-3-new}) is weaker than theirs.

\subsection{Gaussian mixture model}\label{KPCA-sec-GMM-GMM}

The symmetries and other structural properties of Gaussian mixture models allow for more precise characterizations than the above. While a main focus of interest is parameter estimation by likelihood-based methods \citep{DLR77} and methods of moments \citep{Pea94},
the problem of clustering is less explored.
Recently there is a surge of interest in sharp statistical guarantees, mostly under the isotropic Gaussian mixture model \citep{LZh16,CZh18,Nda18,LZZ19,CYa20}.
In another line of study, sparsity assumptions are adopted for high-dimensional regimes \citep{ASW13, JWa16}.
We study spectral clustering under the following model.

\begin{definition}[Gaussian mixture model]\label{KPCA-defn-GMM}
For $\by \in \{ \pm 1 \}^n$ and $\bmu \in \RR^d$ with $n, d \geq 2$, we write $\{ \bx_i \}_{i=1}^n \sim \mathrm{GMM} ( \bmu, \by )$ if
\begin{align}
\bx_i = y_i \bmu + \bz_i \in \RR^d ,\qquad i \in [n],
\label{KPCA-sec-GMM-1}
\end{align}
and $\{ \bz_i \}_{i=1}^n \subseteq \RR^d$ are i.i.d. $N(\mathbf{0} , \bI_d)$ vectors.
\end{definition}

It is natural to use the spectral estimator $\hat{\by} = \sgn(\bu_1)$ to recover $\by$, where $\bu_1$ is the leading eigenvector of $\bG = \cH ( \bX \bX^{\top} )$. This method can be viewed as a special case of the spectral algorithm  in Section~\ref{KPCA-sec-GMM-sGMM} that uses $k$-means for two classes and symmetric centroids. To gauge the misclassification rate of $\hat\by$, define
\begin{align}
& \cM ( \hat{\by} , \by ) = \min_{s = \pm 1}
|\{ 
i \in [n] :~ \hat y_i \neq s y_i 
\}|.
\label{eqn-gmm-mismatch}
\end{align}
Note that for the Gaussian mixture model \eqref{KPCA-sec-GMM-1}, $\bar{\bx}_i = y_i \bmu$, $\bar\bX = \by \bmu^{\top}$ and $\bar\bG = \bar\bX \bar\bX^{\top} = \| \bmu \|_2^2 \by \by^{\top}$. Then $\bar\lambda_1 = n \| \bmu \|_2^2$ and $\bar\bu_1 = \by / \sqrt{n}$.

\begin{theorem}\label{KPCA-thm-GMM}
Let $\{ \bx_i \}_{i=1}^n \sim \mathrm{GMM} ( \bmu, \by )$ and $n \to \infty$. Define
\begin{align}
\mathrm{SNR} = \frac{ \| \bmu \|_2^4 }{ \| \bmu \|_2^2 + d / n }.
\label{KPCA-snr-gmm}
\end{align}
	\begin{enumerate}
		\item If $\mathrm{SNR} > (2 + \varepsilon) \log n$ for some constant $\varepsilon > 0$, then $\lim_{n \to \infty} \PP [ \cM (  \hat\by, \by ) = 0 ] = 1$;
		\item If $1 \ll \mathrm{SNR} \leq 2 \log n$, then $\limsup_{n\to\infty} \mathrm{SNR}^{-1} \log \EE \cM ( \hat\by , \by ) \leq - 1 / 2$.
	\end{enumerate}
\end{theorem}

Theorem \ref{KPCA-thm-GMM} characterizes the spectral estimator with explicit constants. Here we do not impose any specific assumption on the dimension $d=d_n$ so long as $\mathrm{SNR} \to \infty$. It may be bounded or diverge at any rate. When $\mathrm{SNR}$ exceeds $2\log n$, $\sgn(\bu_1)$ exactly recovers all the labels (up to a global sign flip) with high probability. When $1 \ll \mathrm{SNR} \leq 2 \log n$, the misclassification rate is bounded from above by $e^{ - \mathrm{SNR} / [2 + o(1)] }$. According to \cite{Nda18}, both results are optimal in the minimax sense. The proof of Theorem \ref{KPCA-thm-GMM} is in Appendix \ref{KPCA-thm-GMM-proof}.

\cite{CZh18} prove that $\mathrm{SNR} \to \infty$ is necessary for any estimator to achieve vanishingly small misclassification rate and derive an upper bound $\EE \cM ( \sgn (\tilde{\bu}_1) , \by ) \lesssim 1/\mathrm{SNR}$ for $\tilde{\bu}_1$ being the leading eigenvector of the unhollowed Gram matrix $\bX \bX^{\top}$. \cite{Nda18} obtains exact recovery guarantees as well as an optimal exponential error bound for an iterative algorithm starting from $\sgn(\bu_1)$. Our analysis shows that the initial estimator is already good enough and no refinement is needed. \cite{CYa20} study the information threshold for exact recovery in multi-class setting and use an SDP to achieve that.

The $\mathrm{SNR}$ in (\ref{KPCA-snr-gmm}) is closely related to (indeed equivalent to in the order of magnitude) that in (\ref{eqn-kmeans-SNR-2}). One can immediately see this by setting $\bmu_1 = \bmu$, $\bmu_2 = - \bmu$ and $\bSigma = \bI_d$ in (\ref{eqn-kmeans-SNR-2}). The $\mathrm{SNR}$ precisely quantifies the signal-to-noise ratio for clustering and is always dominated by the classical one $\| \bmu \|_2^2$. When $d \gg n$, the condition $\mathrm{SNR} \to \infty$ is equivalent to
\begin{align}
\| \bmu \|_2 \gg (d/n)^{1/4}.
\label{KPCA-snr-gmm-1}
\end{align}
This is weaker than the commonly-used assumption
\begin{align}
\| \bmu \|_2 \gg \sqrt{d/n}
\label{KPCA-snr-gmm-2}
\end{align}
for clustering \citep{LZh16, LZZ19}, under which $\mathrm{SNR}$ is asymptotically equivalent to $\| \bmu \|_2^2$. Their discrepancy reflects an interesting high-dimensional phenomenon.

For the Gaussian mixture model in Definition \ref{KPCA-defn-GMM}, parameter estimation and clustering amount to recovering $\bmu \in \RR^d$ and $\by \in \{ \pm 1 \}^n$, respectively. A good estimate of $\bmu$ yields that of $\by$. Hence clustering should be easier than parameter estimation. The difference becomes more significant when $d \gg n$ as clustering targets fewer unknowns. To see this, we write $\bX = (\bx_1,\cdots,\bx_n)^{\top} \in \RR^{n\times d}$ and observe that
\[
\bX = \by \bmu^{\top} + \bZ,
\]
where $\bZ = (\bz_1,\cdots,\bz_n)^{\top} \subseteq \RR^{n \times d}$ has i.i.d. $N(0,1)$ entries. Clustering and parameter estimation correspond to estimating the left and right singular vectors of the signal matrix $\EE \bX$. According to the results by \cite{CZh18} on singular subspace estimation, (\ref{KPCA-snr-gmm-1}) and (\ref{KPCA-snr-gmm-2}) are sharp conditions for consistent clustering and parameter estimation. They ensure concentration of the Gram matrix $\bX \bX^{\top}$ and the covariance matrix $\frac{1}{n}\bX^{\top} \bX$.
When $(d/n)^{1/4} \ll \| \bmu \|_2 \ll \sqrt{d/n}$, consistent clustering is possible even without consistent estimation of the model parameter $\bmu$. Intuitively, there are many discriminative directions that can tell the classes apart but they are not necessarily aligned with the direction of $\bmu$.

Below we outline the proof of Theorem \ref{KPCA-thm-GMM}. The following $\ell_p$ approximation result for the regime $1 \ll \mathrm{SNR} \lesssim \log n$ helps illustrate main ideas. Its proof is deferred to Appendix \ref{KPCA-thm-GMM-Lp-proof}.

\begin{theorem}\label{KPCA-thm-GMM-Lp}
	Under the GMM model in Definition \ref{KPCA-defn-GMM} with $n \to \infty$ and $1 \ll \mathrm{SNR} \lesssim \log n$, there exist $\varepsilon_n \to 0$ and positive constants $C, N$ such that
	\begin{align*}
	\PP ( \| \bu_1 - \bG \bar\bu_1 / \bar\lambda_1 \|_{\mathrm{SNR}} < \varepsilon_n \| \bar\bu_1 \|_{\mathrm{SNR}} ) > 1 - C e^{-\mathrm{SNR}}, \qquad \forall n \geq N.
	\end{align*}
\end{theorem}

In a hand-waving way, the analysis right after (\ref{eqn-intro-Lp}) in the introduction suggests that the expected misclassification rate of $\sgn(\bu_1)$ differs from that of $\sgn(\bG \bar\bu_1 / \bar\lambda_1)$ by at most $O( e^{-\mathrm{SNR}} )$.
Then, it boils down to studying $\sgn(\bG \bar\bu_1 / \bar\lambda_1)$.
Note that
\begin{align*}
( \bG \bar\bu_1 / \bar\lambda_1 )_i  \propto
( \bG \by )_i = \sum_{j=1}^{n} [ \cH ( \bX \bX^{\top}) ]_{ij} y_j = \sum_{j \neq i} \langle \bx_i, \bx_j  y_j \rangle
= (n-1) \langle \bx_i, \hat\bmu^{(-i)} \rangle,~~ \forall i \in [n].
\end{align*}
Here $\hat\bmu^{(-i)} = \frac{1}{n-1} \sum_{j \neq i} \bx_j y_j$ is an estimate of $\bmu$ based on the samples $\{ \bx_j \}_{j \neq i}$ and their labels $\{ y_j \}_{j \neq i}$. It is straightforward to prove
\begin{align*}
\EE \cM ( \sgn(\bG \bar\bu_1 / \bar\lambda_1), \by ) = \frac{1}{n} \sum_{i=1}^{n} \PP [ \sgn(\langle \bx_i, \hat\bmu^{(-i)} \rangle) \neq y_i ] \leq e^{-\mathrm{SNR} / [2 + o(1)]}
\end{align*}
and get the same bound for $\EE \cM ( \sgn(\bu_1), \by )$. When $\mathrm{SNR} > (2+\varepsilon) \log n$, this leads to an $n^{-(1+\varepsilon/2)}$ upper bound for the misclassification rate, which implies exact recovery with high probability as any misclassified sample contributes $n^{-1}$ to the error rate. When $\mathrm{SNR} \leq 2 \log n$, we get the second part in Theorem \ref{KPCA-thm-GMM}. The proof is then finished.

The quantity $ \sgn( \langle \bx_i, \hat\bmu^{(-i)} \rangle )$ is the prediction of $y_i$ by linear discriminant analysis (LDA) given features $\{ \bx_i \}_{i=1}^n$ and additional labels $\{ y_j \}_{j \neq i}$. It resembles an oracle (or genie-aided) estimator that is usually linked to the fundamental limits of clustering \citep{ABH16, ZZh16}, which plays an important role in our analysis as well. By connecting $\bu_1$ with $ \bG \bar\bu_1 / \bar\lambda_1$ and thus $\{ \langle \bx_i, \hat\bmu^{(-i)} \rangle \}_{i = 1}^n$, Theorem \ref{KPCA-thm-GMM-Lp} already hints the optimality of $\sgn( \bu_1 )$. Our analysis may also apply to spectral algorithms in similar problems such as the bipartite stochastic block model \citep{NST21}.

Perhaps surprisingly, both the (unsupervised) spectral clustering and (supervised) LDA achieve the minimax optimal misclassification error $e^{-\mathrm{SNR} / [2 + o(1)]}$. The missing labels do not hurt much. This phenomenon is also observed by \cite{Nda18}. On the other hand, the Bayes classifier $\sgn( \langle \bmu , \bx \rangle )$ given the \emph{true} parameter $\bmu$ achieves error rate $1 - \Phi ( \| \bmu \|_2 )$, where $\Phi$ is the cumulative distribution function of $N(0,1)$. As $\| \bmu \|_2 \to \infty$, this is $e^{ - \| \bmu \|_2^2 / [2 + o(1)] }$ and it is always superior to the minimax error without the knowledge of $\bmu$. From there we get the followings for spectral clustering and LDA.
\begin{itemize}
\item If $\| \bmu \|_2 \gg \sqrt{d / n}$, then $\mathrm{SNR} = \| \bmu \|_2^2 [1 + o(1)]$ and both estimators achieve the Bayes error exponent;
\item If $\| \bmu \|_2 \leq C \sqrt{d / n}$ for some constant $C > 0$, then $\mathrm{SNR} \leq \| \bmu \|_2^2 / (1 + C^{-2})$ and both estimators achieve the minimax optimal exponent that is worse than the Bayes error exponent.
\end{itemize}

\section{Contextual stochastic block model}\label{KPCA-sec-CSBM}

\subsection{Problem setup}

Contextual network analysis concerns discovering interesting structures such as communities in a network with the help of node attributes. Large-scale applications call for computationally efficient procedures incorporating the information from both sources. For community detection in the contextual setting, various models and algorithms have been proposed and analyzed \citep{ZLZ16, WFe16, BVR17, MMa17, DSM18, MHC19, YSa20}. How to quantify the benefits of aggregation is a fundamental and challenging question.
We study community detection under a canonical model for contextual network data and prove the optimality of a simple spectral method.

To begin with, we present a binary version of the stochastic block model \citep{HLL83} that plays a central role in statistical network analysis \citep{Abb17}.
We use a label vector $\by = (y_1,\cdots,y_n)^{\top} \in \{\pm 1\}^n$ to encode the block (community) memberships of nodes. For any pair of nodes $i$ and $j$, we connect them with probability $\alpha$ if they are from the same block. Otherwise, the connection probability is $\beta$.

\begin{definition}[Stochastic Block Model]\label{KPCA-defn-SBM}
	For $n \in \ZZ_+$, $\by \in \{ \pm 1 \}^n$ and $0 < \alpha,  \beta < 1$, we write $\bA \sim \mathrm{SBM} ( \by, \alpha, \beta)$ if $\bA \in \{ 0, 1 \}^{n\times n}$ is symmetric, $A_{ii} = 0$ for all $i\in[n]$, $\{ A_{ij} \}_{1 \leq i < j \leq n}$ are independent, and
	\begin{align*}
	\PP ( A_{ij} = 1 ) = \begin{cases}
	\alpha & \mbox{ if } y_i  = y_j \\
	\beta & \mbox{ if } y_i \neq  y_j
	\end{cases},\qquad
	\forall i \neq j.
	\end{align*}
\end{definition}

In addition to the network, we also observe an attribute vector $\bx_i \in \RR^d$ of each node $i$ and postulate the Gaussian mixture model in Definition \ref{KPCA-defn-GMM}. Given the labels and other parameters, the network $\bA$ and node attributes $\{ \bx_i \}_{i=1}^n$ are assumed to be independent.
We borrow the name ``contextual stochastic block model'' from \cite{DSM18}.
More general versions can be found in \cite{BVR17}, \cite{DSM18} and \cite{YSa20}. In another line of research, the network $\bA$ is generated based on the covariates $\{ \bx_i \}_{i=1}^n$ \citep{WFe16, MMa17, MHC19}.

For simplicity, we impose uniform priors on the label vector $\by$ and the direction of separation vector $\bmu$. The two blocks are then approximately balanced.


\begin{definition}[Contextual Stochastic Block Model]\label{KPCA-defn-CSBM}
For $n \in \ZZ_+$, $0 <  \alpha , \beta < 1$, $d \geq 2$ and $R>0$, we write $( \by, \bmu, \bA, \{ \bx_i \}_{i=1}^n ) \sim \mathrm{CSBM} ( n, d, \alpha, \beta, R )$ if
\begin{enumerate}
\item the label vector $\by$ and separation vector $\bmu$ are independently generated from the uniform distributions over $\{ \pm 1 \}^n$ and $\{ \bu \in \RR^d :~ \| \bu \|_2 = R \}$, respectively;
\item given $\by$ and $\bmu$, the network $\bA$ and attributes $\{ \bx_i \}_{i=1}^n $ are independently generated from $\mathrm{SBM} ( \by, \alpha, \beta)$ and $\mathrm{GMM} ( \bmu, \by)$, respectively.
\end{enumerate}
\end{definition}

The goal of contextual community detection is to reconstruct $\by$ based on $\bA$ and $\{ \bx_i \}_{i=1}^n$.
We consider a commonly-used regime of the network where the connection probabilities $\alpha$, $\beta$ scale like $q_n/n$ for some $1 \ll q_n \leq \log n$ and differ by a constant factor. When $q_n \gg \log n$, one can easily recover the communities perfectly from $\bA$ \citep{Abb17}. When $q_n = O(1)$, it is not possible to achieve vanishingly small misclassification error \citep{ZZh16}. We are interested in the intermediate regime $1 \ll q_n \leq \log n$.
Meanwhile, recall that $\mathrm{SNR} = R^4 / (R^2 + d / n)$ in (\ref{KPCA-snr-gmm}) is the signal-to-noise ratio of the Gaussian mixture model. We take $\mathrm{SNR} \asymp q_n$ to ensure that the signal strengths of $\bA$ and $\{ \bx_i \}_{i=1}^n$ are comparable. There is no specific assumption on the dimension $d=d_n$. It may be bounded or diverge at any rate.

\begin{assumption}[Asymptotics]\label{KPCA-assumption-CSBM-q}
	Let $a$, $b$ and $c$ be positive constants. $( \by, \bmu, \bA, \{ \bx_i \}_{i=1}^n ) \sim \mathrm{CSBM} ( n, d, \alpha, \beta, R )$ with $1 \ll q_n \leq \log n$, $\alpha = \frac{a q_n}{n}$, $\beta = \frac{b q_n}{n}$ and $R^4 / (R^2 + d / n) = c q_n $.
\end{assumption}

\subsection{An aggregated spectral estimator}

On the one hand, Section \ref{KPCA-sec-GMM-GMM} shows that the leading eigenvector $\bu_1$ of the hollowed Gram matrix $\bG = \cH( \bX \bX^{\top} )$ is optimal for the Gaussian mixture model. From now on we rename it as $\bu_1(\bG)$ to avoid ambiguity. On the other hand, the second eigenvector $\bu_2(\bA)$ of $\bA$ estimates the labels under the stochastic block model \citep{AFW17}. To get some intuition, suppose that half of the entries in $\{ y_i \}_{i=1}^n$ are $+1$'s and the others are $-1$'s so that $ \mathbf{1}_n^{\top} \by = 0$. For such $\by$, it is easy to see from
\begin{align} \label{fan1}
\EE ( \bA | \by ) = \frac{\alpha+ \beta}{2} \mathbf{1}_n \mathbf{1}_n^{\top} + \frac{\alpha- \beta}{2} \by \by^{\top}
\end{align}
that its second eigenvector $\by/\sqrt{n}$ reveals the community structure. We propose an estimator for the integrated problem by aggregating the two individual spectral estimators $\bu_2(\bA)$ and $\bu_1(\bG)$. Without loss of generality, we assume $\langle \bu_2(\bA), \bu_1(\bG) \rangle \geq 0$ to avoid cancellation.

Let us begin the construction. The ideal `estimator'
\[
\hat y_i^{\mathrm{genie}} = \argmax_{ y = \pm 1 }  \PP ( y_i =  y | \bA , \bX, \by_{-i} ).
\]
is the best guess of $y_i$ given the network, attributes, and labels of all nodes (assisted by  Genie) except the $i$-th one. It is referred to as a genie-aided estimator or oracle estimator in the literature and is closely related to fundamental limits in clustering \citep{ABH16, ZZh16}, see Theorem \ref{KPCA-lemma-fundamental}. To mimic $\hat y_i^{\mathrm{genie}}$, we first approximate its associated odds ratio.

\begin{lemma}\label{KPCA-lemma-CSBM-linearization}
Under Assumption \ref{KPCA-assumption-CSBM-q}, for each given $i$, we have 
\begin{align*}
\bigg| \log \bigg( \frac{ \PP ( y_i =  1 | \bA , \bX, \by_{-i} ) }{ \PP ( y_i = -1 | \bA , \bX, \by_{-i} ) } \bigg) -
\bigg[ \bigg(
\log (a/b) \bA + \frac{2}{n + d / R^2 } \bG
\bigg) \by \bigg]_i
\bigg| = o_{\PP} (q_n;~q_n).
\end{align*}
\end{lemma}

The $i$-th coordinate of $\bA \by$ corresponds to the log odds ratio $ \log [ \PP ( y_i =  1 | \bA , \by_{-i} ) / \PP ( y_i = -1 | \bA , \by_{-i} ) ]$ for the stochastic block model \citep{ABH16}. From $A_{ii} = 0$ we see that $(\bA \by)_i = \sum_{j \neq i} A_{ij} y_j$ tries to predict the label $y_i$ via majority voting among the neighbors of node $i$. Similarly, $(\bG \by)_i$ relates to the log odds ratio $ \log [ \PP ( y_i =  1 | \bX , \by_{-i} ) / \PP ( y_i = -1 | \bX , \by_{-i} ) ]$ for the Gaussian mixture model. The overall log odds ratio is linked to a linear combination of $\bA \by$ and $\bG \by$ thanks to the conditional independence between $\bA$ and $\bX$ in Definition \ref{KPCA-defn-CSBM}. The proof of Lemma \ref{KPCA-lemma-CSBM-linearization} can be found in Appendix \ref{KPCA-lemma-CSBM-linearization-proof}.

Intuitively, Lemma \ref{KPCA-lemma-CSBM-linearization} reveals that
\begin{align*}
\sgn \bigg( \log (a/b) \bA \by + \frac{2}{n + d / R^2 } \bG \by \bigg) \approx (\hat y_1^{\mathrm{genie}}, \cdots, \hat y_n^{\mathrm{genie}})^{\top}
\end{align*}
The left-hand side still involves unknown parameters $a/b$, $R$ and $\by$.  Once these unknowns are consistently estimated, the substitution version of the left-hand side provides a valid estimator that mimics well the genie-aided estimator and hence is optimal.
Heuristics of linear approximation in Theorem \ref{KPCA-thm-GMM-Lp} above and \cite{AFW17} suggest
\begin{align*}
\bu_2(\bA) \approx \bA \bar\bu / \bar\lambda_A \qquad \text{and} \qquad \bu_1(\bG) \approx \bG \bar\bu / \bar\lambda_G.
\end{align*}
Here $\bar\bu = \by / \sqrt{n}$, $\bar\lambda_A = n(\alpha-\beta)/2$ is the second largest (in absolute value) eigenvalue of $\EE(\bA | \by)$ when $\alpha \neq \beta$ and the two blocks are equally-sized, and $\bar\lambda_G = n R^2$ is the leading eigenvalue of $\bar\bG = \bar\bX \bar\bX^{\top}$. Hence
\begin{align}
~~& \log (a/b) \bA \by + \frac{2}{n + d / R^2 } \bG \by \notag\\
& \approx \log (a/b) \sqrt{n} \bar\lambda_A \bu_2(\bA) + \frac{2}{n + d / R^2 } \sqrt{n} \bar\lambda_G \bu_1(\bG)  \notag \\
& \propto  \frac{n(\alpha-\beta)}{2} \log \bigg( \frac{\alpha}{\beta} \bigg) \bu_2(\bA) + \frac{2 R^4}{R^2 + d / n} \bu_1(\bG),
\label{eqn-CSBM-aggregation}
\end{align}
which yields a linear combination of $\bu_2(\bA)$ and $\bu_1(\bG)$. The coefficient in front of $\bu_1(\bG)$ is twice the $\mathrm{SNR}$ in (\ref{KPCA-snr-gmm}) for Gaussian mixture model.  Analogously, we may regard $\frac{n(\alpha-\beta)}{4} \log ( \alpha / \beta ) $ as a signal-to-noise ratio for the stochastic block model.

An legitimate estimator for $\by$ is obtained by replacing the unknown parameters $\alpha$, $\beta$ and $R$ in (\ref{eqn-CSBM-aggregation}) by their estimates.  When the two classes are balanced, i.e. $\by^{\top} \mathbf{1}_n = 0$,  \eqref{fan1} yields $\lambda_{1}[ \EE(\bA | \by) ] = n(\alpha + \beta) / 2$ and $\lambda_{2}[ \EE(\bA | \by) ] = n(\alpha - \beta) / 2$. Here $\lambda_j(\cdot)$ denotes the $j$-th largest (in absolute value) eigenvalue of a real symmetric matrix. Hence,
\begin{align*}
\frac{n(\alpha-\beta)}{2} \log \bigg( \frac{\alpha}{\beta} \bigg)
& = \lambda_{2}[ \EE(\bA | \by) ] \log \bigg( \frac{ \lambda_{1}[ \EE(\bA | \by) ] + \lambda_{2}[ \EE(\bA | \by) ]}{ \lambda_{1}[ \EE(\bA | \by) ] - \lambda_{2}[ \EE(\bA | \by) ]} \bigg) \\
& \approx \lambda_2(\bA) \log \bigg( \frac{\lambda_1(\bA) + \lambda_2(\bA)}{\lambda_1(\bA) - \lambda_2(\bA)} \bigg) .
\end{align*}
It can be consistently estimated using the plug-in method.
Similarly, using $\lambda_1(\bar\bG) = n R^2$, we have
\begin{align*}
\frac{2 R^4}{R^2 + d / n} = \frac{2 [\lambda_1(\bar\bG)/n]^2 }{ \lambda_1(\bar\bG) / n + d / n } \approx \frac{2 \lambda_1^2(\bG) }{n \lambda_1(\bG) + n d }.
\end{align*}

Based on these, we get an aggregated spectral estimator $\sgn(\hat\bu)$ with
\begin{align}
\hat\bu = \log \bigg( \frac{\lambda_1(\bA) + \lambda_2(\bA)}{\lambda_1(\bA) - \lambda_2(\bA)} \bigg)  \lambda_2(\bA) \bu_2(\bA) +  \frac{2 \lambda_1^2(\bG) }{n \lambda_1(\bG) + n d } \bu_1(\bG).
\label{eqn-CSBM-aggregation-1}
\end{align}
Our estimator uses a weighted sum of two individual estimators without any tuning parameter. 
When there are $K>2$ communities, it is natural to compute spectral embeddings using $(K-1)$ eigenvectors of $\bA$ and $\bG$, respectively. One may apply Procrustes analysis \citep{Wah65} to align the two embeddings and then use their linear combination for clustering. It would be nice to develop a tuning-free procedure similar to the above.

\cite{BVR17} propose a spectral method based on a weighted sum of the graph Laplacian matrix and $\bX \bX^{\top}$. \cite{YSa20} develop an SDP using a weighted sum of $\bA$ and a kernel matrix of $\{ \bx_i \}_{i=1}^n$. \cite{DSM18} study a belief propagation algorithm. Their settings are different from ours.

\subsection{Analysis of the estimator when $q_n = \log n$}

There are very few theoretical results on the information gain in combining the network and node attributes. \cite{BVR17} and \cite{YSa20} derive upper bounds for the misclassification error that depend on both sources of information. However, those bounds are not tight and cannot rigorously justify the benefits. \cite{DSM18} use techniques from statistical physics to derive an information threshold for weak recovery (i.e. better than random guessing) in some regimes. The threshold is smaller than those for the stochastic block model and the Gaussian mixture model. Their calculation is under the sparse regime where the maximum expected degree $n(\alpha + \beta) / 2$ of the network remains bounded as $n$ goes to infinity. They obtain a formal proof by taking certain large-degree limits.
To our best knowledge, the result below gives the first characterization of the information threshold for exact recovery and provides an efficient method achieving it by aggregating the two pieces of information.

We now investigate the aggregated spectral estimator (\ref{eqn-CSBM-aggregation-1}) under the Assumption \ref{KPCA-assumption-CSBM-q} with $q_n = \log n$. Our study shows that $\sgn(\hat\bu)$ achieves the information threshold for exact recovery as well as the optimal misclassification rate, both of which are better than
those based on a single form of data in terms of the mismatch $\cM$ in (\ref{eqn-gmm-mismatch}). To state the results, define
\begin{align}
I^*( a, b, c ) = \frac{(\sqrt{a} - \sqrt{b})^2 + c}{2}.
\label{KPCA-eqn-CSBM-h}
\end{align}

\begin{theorem}\label{KPCA-theorem-CSBM-eaxct}
	Let Assumption \ref{KPCA-assumption-CSBM-q} hold with $q_n = \log n$ and $a \neq b$. 
\begin{enumerate}
\item When $I^*( a, b, c )  > 1$, we have $\lim_{n \to \infty} \PP [\cM ( \sgn( \hat\bu ) , \by ) = 0] = 1$.
\item When $I^*( a, b, c ) < 1$, we have $\liminf_{n \to \infty} \PP [\cM ( \hat\by , \by ) > 0] > 0$ for any sequence of estimators $\hat\by = \hat\by_n ( \bA, \{ \bx_i \}_{i=1}^n )$.
\end{enumerate}
\end{theorem}

Theorem \ref{KPCA-theorem-CSBM-eaxct} asserts that $I^*(a,b,c)$ quantifies the signal-to-noise ratio and the phase transition of exact recovery takes place at $I^*( a, b, c ) = 1$.  When $c = 0$ (node attributes are uninformative), we have
$I^*(a, b, 0) = ( \sqrt{a} - \sqrt{b} )^2 / 2 $; the threshold reduces to that for the stochastic block model [$|\sqrt{a} - \sqrt{b} | = \sqrt{2}$ by \cite{ABH16}].  Similarly, when $a = b$ (the network is uninformative), we have
$I^*(a, a, c) = c/2$; the threshold reduces to that for the Gaussian mixture model [$c = 2$ by \cite{Nda18}]. The relation (\ref{KPCA-eqn-CSBM-h}) indicates that combining two sources of information adds up the powers of each part. The proof of Theorem \ref{KPCA-theorem-CSBM-eaxct} is deferred to Appendix \ref{KPCA-theorem-CSBM-eaxct-proof}.

Figure \ref{fig-CSBM} demonstrates the efficacy of our aggregated estimator $\sgn( \hat\bu )$. The two experiments use $c = 0.5$ and $c = 1.5$ respectively. We fix $n = 500$, $d = 2000$ and vary $a$ ($y$-axis), $b$ ($x$-axis) from $0$ to $8$. For each parameter configuration $(a,b,c)$, we compute the frequency of exact recovery (i.e. $\sgn(\hat\bu) = \pm \by$) over 100 independent runs. Light color represents high chance of success. The red curves $(\sqrt{a} - \sqrt{b})^2 + c = 2$ correspond to theoretical boundaries for phase transitions, which match the empirical results pretty well. Also, larger $c$ implies stronger signal in node attributes and makes exact recovery easier.

\begin{figure}[ht]
	\begin{minipage}[b]{0.45\linewidth}
		\centering
		\includegraphics[width=0.7\textwidth]{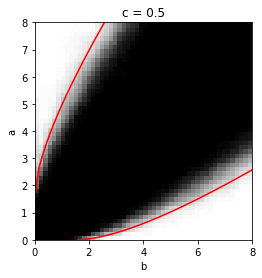}
	\end{minipage}
	\hspace{0.5cm}
	\begin{minipage}[b]{0.45\linewidth}
		\centering
		\includegraphics[width=0.7\textwidth]{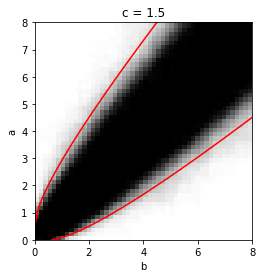}
	\end{minipage}
		\caption{Exact recovery for CSBM: $c = 0.5$ (left) and $c = 1.5$ (right).}
				\label{fig-CSBM}
\end{figure}

When $I^*(a,b,c) < 1$, exact recovery of $\by$ with high probability is no longer possible. In that case, we justify the benefits of aggregation using misclassification rates, by presenting an upper bound for $\sgn(\hat\bu)$ as well as a matching lower bound for all possible estimators. Their proofs can be found in Appendices \ref{KPCA-theorem-CSBM-error-rate-proof} and \ref{KPCA-thm-CSBM-lower-proof}.

\begin{theorem}\label{KPCA-theorem-CSBM-error-rate}
Let Assumption \ref{KPCA-assumption-CSBM-q} hold, $q_n = \log n$, $a \neq b$ and $I^*(a,b,c) \leq 1$. Then
	\begin{align*}
	\limsup_{n \to \infty} q_n^{-1} \log \EE \cM ( \sgn( \hat\bu ) , \by )  \leq - I^*( a, b, c ) .
	\end{align*}
\end{theorem}
\begin{theorem}\label{KPCA-thm-CSBM-lower}
Let Assumption \ref{KPCA-assumption-CSBM-q} hold. For any sequence of estimators $\hat\by = \hat\by_n ( \bA, \{ \bx_i \}_{i=1}^n )$,
	\begin{align*}
	\liminf_{n \to \infty} q_n^{-1} \log \EE \cM ( \hat{ \by } , \by ) \geq - I^*(a, b, c).
	\end{align*}
\end{theorem}

Theorems \ref{KPCA-theorem-CSBM-error-rate} and \ref{KPCA-thm-CSBM-lower} imply that in the $\log n$-regime, the aggregated spectral estimator $\sgn(\hat\bu)$ achieves the optimal misclassification rate:
\[
\EE \cM ( \sgn(\hat\bu) , \by ) = n^{-I^*(a,b,c)+o(1)}.
\]
When $c = 0$, it reduces to the optimal rate $n^{ - ( \sqrt{a} - \sqrt{b} )^2 / 2 + o(1) }$ for the stochastic block model (Definition \ref{KPCA-defn-SBM}) and when  $a=b$, the result reduces $n^{ - c / 2 + o(1) }$ for the Gaussian mixture model (Definition \ref{KPCA-defn-GMM}), respectively. It is easy to show that they are achieved by $\bu_2(\bA)$ \citep{AFW17} and $\bu_1(\bG)$ (Theorem \ref{KPCA-thm-GMM}), which are asymptotically equivalent to our aggregated estimator $\hat\bu$ in extreme cases $c \to 0$ and $a \to b$, respectively. In other words, our result and procedure encompass those for the stochastic block model and Gaussian mixture model as two specific examples.

\subsection{A modified estimator for the general case}

While \Cref{KPCA-thm-CSBM-lower} establishes a lower bound $e^{- q_n [ I^*(a,b,c) + o(1) ] }$ for misclassification under Assumption \ref{KPCA-assumption-CSBM-q} without restricting $q_n = \log n$, our aggregated spectral estimator $\sgn(\hat\bu)$ is only analyzed for the $\log n$-regime. If the network becomes sparser ($q_n \ll \log n$), the empirical eigenvalues $\lambda_1(\bA)$ and $\lambda_2(\bA)$ no longer concentrate around  $\frac{\alpha + \beta}{2}$ and $\frac{\alpha - \beta}{2}$ \citep{FOf05}. The eigenvector analysis of $\bu_2(\bA)$ in \cite{AFW17} breaks down. Consequently, the estimator (\ref{eqn-CSBM-aggregation-1}) fails to approximate the (scaled) vector of log odds $n^{-1/2} \log (a/b) \bA \by + n^{-1/2}  \frac{2}{n + d / R^2 } \bG \by  $.
 
Fortunately, the $\ell_p$ results for $\bu_1(\bG)$ in the current paper continue to hold, and $\hat{\by}_G = \sgn[\bu_1(\bG)]$ faithfully recovers $\by$. Hence, we need only to modify the first term in (\ref{eqn-CSBM-aggregation-1}) concerning the network $\bA$, which aims to approximate $n^{-1/2} \log (a/b) \bA \by $. To approximate $\log(a/b)$, we resort to $\bm{1}^{\top} \bA \bm{1} / n \approx \frac{\alpha + \beta}{2} = \frac{( a+b ) q_n}{n}$ and $\by \bA \by / n \approx \frac{\alpha - \beta}{2} = \frac{( a-b ) q_n}{n}$ so that $a/b \approx \frac{\bm{1}^{\top} \bA \bm{1} + \by \bA \by}{\bm{1}^{\top} \bA \bm{1} - \by \bA \by}$.  Thus, we propose a new estimator $\sgn(\tilde\bu)$ with
\begin{align}
\tilde\bu = \frac{1}{\sqrt{n}} \log \bigg( \frac{
\bm{1}^{\top} \bA \bm{1} + \hat{\by}_G^{\top} \bA \hat{\by}_G
}{
\bm{1}^{\top} \bA \bm{1} - \hat{\by}_G^{\top} \bA \hat{\by}_G
} \bigg) \bA  \hat{\by}_G 
 \ +  \frac{2 \lambda_1^2(\bG) }{n \lambda_1(\bG) + n d } \bu_1(\bG),
\label{eqn-CSBM-aggregation-2}
\end{align}
where $\hat{\by}_G $ estimates $\by$. 
The new estimator $\sgn(\tilde\bu)$ achieves the fundamental limit $e^{- q_n [ I^*(a,b,c) + o(1) ] }$ even if $q_n \ll \log n$. See \Cref{thm-CSBM-q} below and its proof in \Cref{sec-proof-thm-CSBM-q}.

\begin{theorem}\label{thm-CSBM-q}
Let Assumption \ref{KPCA-assumption-CSBM-q} hold and $a \neq b$. We have
	\begin{align*}
	\limsup_{n \to \infty} q_n^{-1} \log \EE \cM ( \sgn( \tilde\bu ) , \by )  \leq - I^*( a, b, c ) .
	\end{align*}
In addition, if $q_n = \log n$ and $ I^*( a, b, c ) > 1$ then $\lim_{n \to \infty} \PP [\cM ( \sgn( \tilde\bu ) , \by ) = 0] = 1$.
\end{theorem}

\section{Proof ideas}\label{KPCA-sec-outlines}

To illustrate the key ideas behind the $\ell_p$ analysis in Theorem \ref{KPCA-corollary-main}, we use a simple rank-1 model
\begin{align}
\bx_i = \bmu  y_i + \bz_i \in \RR^d , \qquad i \in [n],
\label{KPCA-eqn-Lp-sketch}
\end{align}
where $\by = (y_1,\cdots,y_n)^{\top} \subseteq \{ \pm 1 \}^n$ and $\bmu \in \RR^d$ are deterministic; $\{ \bz_i \}_{i=1}^n$ are independent and $\bz_i \sim N(\mathbf{0} , \bSigma_i)$ for some $\bSigma_i \succ 0$.  We assume  further $\bSigma_i \preceq C \bI_d$ for all $i \in [n]$ and some constant $C>0$.

Model  (\ref{KPCA-eqn-Lp-sketch}) is a heteroscedastic version of the Gaussian mixture model in Definition \ref{KPCA-defn-GMM}. We have $\bar\bx_i = y_i \bmu$, $\bar\bX = (\bar\bx_1 , \cdots, \bar\bx_n)^{\top} = \by \bmu^{\top}$, $\bar\bG = \bar\bX \bar\bX^{\top} = \| \bmu \|_2^2 \by \by^{\top}$, $\bar\lambda_1 = n \| \bmu \|_2^2$ and $\bar\bu_1 = \by / \sqrt{n}$. For simplicity, we suppress the subscript $1$ in $\bu_1$, $\bar\bu_1$, $\lambda_1$ and $\bar\lambda_1$. The goal is to show that for $p$ that satisfies our technical condition,  
\begin{align}
\min_{c = \pm 1} \| c \bu - \bG \bar\bu / \bar\lambda \|_{p} = o_{\PP} ( \| \bar\bu \|_p ;~ p ).
\label{KPCA-outlines-1}
\end{align}
For simplicity, we assume that $\bu$ is already aligned with $\bG\bar\bu / \bar\lambda$ and the optimal $c$ above is $1$.

\subsection{Benefits of hollowing}

The hollowing procedure conducted on the Gram matrix has been commonly used in high-dimensional PCA and spectral methods \citep{KGi00, MSu18, Nda18, CLC19}. 
When the noises $\{ \bz_i \}_{i=1}^n$ are strong and heteroscedastic, it drives $\bG$ closer to $\bar\bG$ and thus ensures small angle between $\bu$ and $\bar\bu$. Such $\ell_2$ proximity is the starting point of our refined $\ell_p$ analysis.

Observe that
\begin{align*}
& \langle \bx_i , \bx_j \rangle = \langle \bar\bx_i ,\bar\bx_j \rangle +  \langle \bar\bx_i ,\bar\bz_j \rangle +  \langle \bz_i ,\bar\bx_j \rangle +  \langle \bz_i ,\bz_j \rangle, \\
& \EE \langle \bx_i , \bx_j \rangle = \langle \bar\bx_i ,\bar\bx_j \rangle + \EE \| \bz_i \|_2^2 \mathbf{1}_{ \{ i = j \} }.
\end{align*}
Hence the diagonal and off-diagonal entries of the Gram matrix behave differently. In high-dimensional and heteroscedastic case, the difference in noise levels $\{ \EE \| \bz_i \|_2^2 \}_{i=1}^n$ could have a severe impact on the spectrum of Gram matrix $\bX\bX^{\top}$. 
In particular, the following lemma shows that the leading eigenvector of $\bX \bX^{\top}$ could be asymptotically perpendicular to that of $\bar\bX \bar\bX^{\top}$, while $\cH ( \bX \bX^{\top})$ is still faithful. The proof is in Appendix \ref{thm-hollowing-proof}.


\begin{lemma}\label{thm-hollowing}
	Consider the model (\ref{KPCA-eqn-Lp-sketch}) with $\bSigma_1 = 2 \bI_d $ and $\bSigma_2 = \cdots = \bSigma_n = \bI_d$. Let $\hat\bu$ and $\bu$ be the leading eigenvectors of the Gram matrix $\bX \bX^{\top}$ and its hollowed version $\cH ( \bX \bX^{\top})$. Suppose that $n \to \infty$ and $(d/n)^{1/4} \ll  \| \bmu \|_2  \ll \sqrt{d/n}$. We have $| \langle \hat\bu, \bar\bu \rangle | \overset{\PP}{\to} 0$ and $| \langle \bu, \bar\bu \rangle | \overset{\PP}{\to} 1$.
\end{lemma}

Figure \ref{fig-hollowing} visualizes the entries of eigenvectors $\bar\bu$ (black), $\hat \bu$ (red) and $\bu$ (blue) in a typical realization with $n = 100$, $d = 500$, $\| \bmu\|_2 = 3$ and $\by = ( \mathbf{1}_{n/2}^{\top} , - \mathbf{1}_{n/2}^{\top})^{\top}$. The population eigenvector $\bar\bu$ perfectly reveals class labels, and the eigenvector $\bu$ of the hollowed Gram matrix is aligned with that. Without hollowing, the eigenvector $\hat\bu$ is localized due to heteroscedasticity and fails to recover the labels. The error rates of $\sgn(\hat\bu)$ and $\sgn(\bu)$ are $48\%$ and $3\%$, respectively.

\begin{figure}[t]
	\centering
	\includegraphics[width=0.4\textwidth]{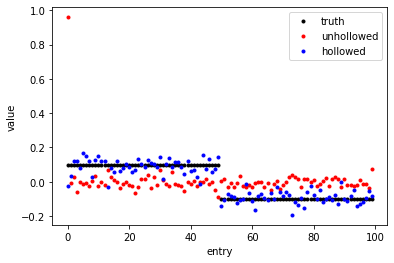}
	\caption{Benefits of hollowing.  Presented are the components of true eigenvector $\by/\sqrt{n}$ (black, $n=100$), the leading eigenvectors $\bu$ (red) and $\hat{\bu}$ (blue) of the Gram matrix and its hollowed version.}
	\label{fig-hollowing}
\end{figure}

With the help of hollowing, we obtain the following results on spectral concentration. See Appendix \ref{lemma-hollowing-proof} for the proof.

\begin{lemma}\label{lemma-hollowing}
	Consider the model (\ref{KPCA-eqn-Lp-sketch}). When $n \to \infty$ and $\| \bmu \|_2 \gg \max \{ 1, (d/n)^{1/4} \}$, we have $\| \bG - \bar\bG \|_2 = o_{\PP} ( \bar\lambda ;~n )$, $|\lambda - \bar\lambda| = o_{\PP} ( \bar\lambda ;~n )$ and $\min_{c = \pm 1} \| c \bu - \bar\bu \|_2 = o_{\PP} (1 ;~ n)$.
\end{lemma}

It is worth pointing out that hollowing inevitably creates bias as the diagonal information of $\bar\bG$ is lost.  Under incoherence conditions on the signals $\{ \bar\bx_i \}_{i=1}^n$ (Assumption \ref{KPCA-assumption-incoherence-euc}), this effect is under control. It becomes negligible when the noise is strong. While the simple hollowing already suffices for our need, general problems may benefit from more sophisticated procedures such as the heteroscedastic PCA in \cite{ZCW18}.

\subsection{Moment bounds and the choice of $p$}

As hollowing has been shown to tackle heteroscedasticity, from now on we focus on the homoscedastic case
\[
\bSigma_1 = \cdots = \bSigma_n = \bI_d
\]
to facilitate presentation. We want to approximate $\bu$ with $\bG \bar\bu / \bar\lambda$. By definition, 
\begin{align*}
\| \bu - \bG \bar\bu / \bar\lambda \|_p = \| \bG \bu / \lambda - \bG \bar\bu / \bar\lambda \|_p
& \leq \| \bG (\bu - \bar\bu) \|_p / |\lambda| + \| \bG \bar\bu \|_p |\lambda^{-1} - \bar\lambda^{-1}| .
\end{align*}
The spectral concentration of $\bG$ (Lemma \ref{lemma-hollowing}) forces $1 / |\lambda| =  O_{\PP} (\bar\lambda^{-1} ;~ n)$ and $|\lambda^{-1} - \bar\lambda^{-1}| = o_{\PP} (\bar\lambda^{-1} ;~ n)$. In order to get (\ref{KPCA-outlines-1}), it suffices to choose some $p \lesssim n$ such that
\begin{align}
& \| \bG (\bu - \bar\bu) \|_p = o_{\PP} ( \bar\lambda \| \bar\bu \|_p ;~ p ), \label{KPCA-outlines-2}\\
& \| \bG \bar\bu \|_p = O_{\PP} ( \bar\lambda \| \bar\bu \|_p ;~ p ). \label{KPCA-outlines-3}
\end{align}

The target (\ref{KPCA-outlines-3}) sheds light on the choice of $p$. Let $\bar\bZ = (\bz_1,\cdots,\bz_n)^{\top}$ and observe that
\begin{align*}
\bG &
 = \cH ( \bX \bX^{\top}) 
 = \cH(\bar\bX \bar\bX^{\top}) + \cH( \bar\bX  \bZ^{\top} ) + \cH( \bZ  \bar\bX^{\top} ) + \cH (\bZ \bZ^{\top}).
\end{align*}
As an example, we show how to obtain
\begin{align}
\| \cH( \bZ  \bar\bX^{\top} ) \bar\bu \|_p = O_{\PP} ( \bar\lambda \| \bar\bu \|_p ;~ p ).
\label{eqn-illustration-p}
\end{align}
By Markov's inequality, a convenient and sufficient condition for \eqref{eqn-illustration-p} is
\begin{align}
\EE^{1/p} \| \cH (\bZ \bar\bX^{\top}) \bar\bu \|_p^p \lesssim \bar\lambda \| \bar\bu \|_p = n \| \bmu \|_2^2 \cdot n^{1/p - 1/2}.
\label{KPCA-outlines-4}
\end{align}
We now establish \eqref{KPCA-outlines-4}.
The facts $[\cH (\bZ  \bar\bX^{\top})]_{ij} = \langle \bz_i , y_j \bmu \rangle \mathbf{1}_{ \{ i \neq j \} }$ and $\bar\bu = \by / \sqrt{n}$ yield
\begin{align}
[ \cH (\bZ \bar\bX^{\top}) \bar\bu ]_i = \sum_{j \neq i} \langle \bz_i ,  y_j \bmu \rangle y_j / \sqrt{n} = \frac{n-1}{\sqrt{n}} \langle \bz_i ,  \bmu \rangle, \qquad \forall i \in [n].
\label{eqn-illustration-p-1}
\end{align}
By $\{ \bz_i \}_{i=1}^n$ are i.i.d. $N(\mathbf{0} , \bI_d)$ random vectors, we have $\langle \bz_i ,  \bmu \rangle \sim N(0, \| \bmu \|_2^2)$. By moment bounds for Gaussian distribution \citep{Ver10}, $\sup_{q \geq 1} \{  q^{-1/2} \EE^{1/q} |\langle \bz_i ,  \bmu \rangle|^q \} \leq c \| \bmu \|_2$ holds for some constant $c$. Then
\begin{align}
\EE \| \cH (\bZ \bar\bX^{\top}) \bar\bu \|_p^p = \sum_{i=1}^{n} \EE | [ \cH (\bZ \bar\bX^{\top}) \bar\bu ]_i |^p \leq n (c \| \bmu\|_2 \sqrt{np} )^p.
\label{eqn-illustration-p-0}
\end{align}
We can achieve (\ref{KPCA-outlines-4}) if $p \lesssim \| \bmu \|_2^2$. Hence $p$ cannot be arbitrarily large. Moment bounds are used throughout the proof. The final choice of $p$ depends on the most stringent condition.

Moments bounds are natural choices for $\ell_p$ control and they adapt to the signal strength. In contrast, the $\ell_{\infty}$ analysis in \cite{AFW17} targets quantities like $\| \bG \bar\bu \|_{\infty}$ and $\| \cH( \bZ  \bar\bX^{\top} ) \bar\bu \|_{\infty}$ by applying concentration inequality to each entry and taking union bounds. We now demonstrate why $\ell_{\infty}$ analysis requires stronger signal than the $\ell_p$ one. Similar to (\ref{eqn-illustration-p}), suppose that we want to prove
\begin{align}
\| \cH( \bZ  \bar\bX^{\top} ) \bar\bu \|_{\infty} = O_{\PP} ( \bar\lambda \| \bar\bu \|_{\infty}) = O_{\PP} (n \| \bmu \|_2^2 / \sqrt{n} ) = O_{\PP} ( \sqrt{n} \| \bmu \|_2^2).
\label{eqn-illustration-p-2}
\end{align}
According to (\ref{eqn-illustration-p-1}), we have $\cH( \bZ  \bar\bX^{\top} ) \bar\bu \sim N ( \bm{0} , \frac{ \| \bmu \|_2^2 (n-1)^2 }{n} \bI_n )$. By Inequality A.3 in \cite{Cha14}, $\EE \| \cH( \bZ  \bar\bX^{\top} ) \bar\bu \|_{\infty} \asymp \sqrt{n \log n} \| \bmu \|_2$ and there exists a constant $c>0$ such that
\begin{align}
\PP [\| \cH( \bZ  \bar\bX^{\top} ) \bar\bu \|_{\infty} > \sqrt{n \log n} \| \bmu \|_2 ] \to 1.
\label{eqn-illustration-p-3}
\end{align}
The $\sqrt{\log n}$ factor is the price of uniform control of $n$ coordinates, as opposed to $\sqrt{p}$ in the adaptive $\ell_p$ bound (\ref{eqn-illustration-p-0}). If $\| \bmu \|_2 \ll \sqrt{\log n}$, (\ref{eqn-illustration-p-3}) contradicts the desired result (\ref{eqn-illustration-p-2}). Then the $\ell_{\infty}$ analysis breaks down.

\subsection{Leave-one-out analysis}

Finally we come to (\ref{KPCA-outlines-2}). Let $\bG_i$ denote the $i$-th row of $\bG$. By definition,
\begin{align*}
& \| \bG (\bu - \bar\bu) \|_{p} = \bigg(
\sum_{i=1}^{n} | \bG_i (\bu - \bar\bu) |^p
\bigg)^{1/p}.
\end{align*}
We need to study $| \bG_i (\bu - \bar\bu) |$ for each individual $i \in [ n ]$. By Cauchy-Schwarz inequality, the upper bound
\[
| \bG_i (\bu - \bar\bu) | \leq \| \bG_i \|_2 \| \bu - \bar\bu \|_2
\]
always holds. Unfortunately, it is too large to be used, as we have not exploited the weak dependence between $\bG_i$ and $\bu$. We should resort to probabilistic analysis for tighter control.

For any $i \in [n]$, we construct a new data matrix
\[
\bX^{(i)} = (\bx_1,\cdots,\bx_{i-1} , \mathbf{0} , \bx_{i+1} , \cdots , \bx_n )^{\top} = (\bI_n - \be_i \be_i^{\top}) \bX
\]
by deleting the $i$-th sample. Then 
\begin{align*}
& \bG_i = ( \langle \bx_i, \bx_1 \rangle , \cdots, \langle \bx_i, \bx_{i-1} \rangle , 0 , 
\langle \bx_i, \bx_{i+1} \rangle ,\cdots , \langle \bx_i, \bx_{n} \rangle )
= \bx_i^{\top}  \bX^{(i)\top}, \\
&  \bG_i (\bu - \bar\bu)  = \langle \bx_i , \bX^{(i)\top} (\bu - \bar\bu) \rangle .
\end{align*}
Recall that $\bu$ is the eigenvector of the whole matrix $\bG$ constructed by $n$ independent samples. It should not depend too much on any individual $\bx_i$. Also, $\bX^{(i)\top} $ is independent of $\bx_i$. Hence the dependence between $\bx_i$ and $\bX^{(i)\top} (\bu - \bar\bu)$ is weak. We would like to invoke sub-Gaussian concentration inequalities to control their inner product.

To decouple them in a rigorous way, we construct leave-one-out auxiliaries $\{ \bG^{(i)} \}_{i=1}^n \subseteq \RR^{n \times n}$ where
\begin{align*}
\bG^{(i)} = \cH ( \bX^{(i)} \bX^{(i)\top} ) = \cH [ (\bI - \be_i \be_i^{\top}) \bX \bX^{\top} (\bI - \be_i \be_i^{\top}) ]
\end{align*} 
is the hollowed Gram matrix of the dataset $\{ \bx_1,\cdots,\bx_{i-1} , \mathbf{0} , \bx_{i+1}, \cdots,\bx_{n} \}$ with $\bx_i$ zeroed out. Equivalently, $\bG^{(i)}$ is obtained by zeroing out the $i$-th row and column of $\bG$.
Let $\bu^{(i)}$ be the leading eigenvector of $\bG^{(i)}$. Then
\begin{align*}
|\bG_i (\bu - \bar\bu)| = |\langle \bx_i , \bX^{(i)\top} (\bu - \bar\bu) \rangle |
\leq \underbrace{ |\langle \bx_i , \bX^{(i)\top} (\bu^{(i)} - \bar\bu) \rangle | }_{\varepsilon_1} + \underbrace{ |\langle \bx_i , \bX^{(i)\top} (\bu - \bu^{(i)}) \rangle | }_{\varepsilon_2}.
\end{align*}
We have the luxury of convenient concentration inequalities for $\varepsilon_1$ as $\bx_i$ and $\bX^{(i)\top} (\bu^{(i)} - \bar\bu)$ are completely independent. In addition, we can safely apply the Cauchy-Schwarz inequality to $\varepsilon_2$ because $\bu^{(i)}$ should be very similar to $\bu$.

The leave-one-out technique is a powerful tool in random matrix theory \citep{ESY09} and high-dimensional statistics \citep{JMo18,Elk18}. \cite{ZBo18}, \cite{AFW17} and \cite{CFM19} apply it to $\ell_{\infty}$ eigenvectors analysis of Wigner-type random matrices. Here we focus on $\ell_p$ analysis of Wishart-type matrices with dependent entries.

\section{Discussion}\label{KPCA-sec-discussions}

We conduct a novel $\ell_p$ analysis of PCA and derive linear approximations of eigenvectors. The results yield optimality guarantees for spectral clustering in several challenging problems. Meanwhile, this study leads to new research directions that are worth exploring.
First, we hope to extend the analysis from Wishart-type matrices to more general random matrices. One example is the normalized Laplacian matrix frequently used in spectral clustering. Second, our general results hold for Hilbert spaces and they are potentially useful in the study of kernel PCA, such as quantifying the performances of different kernels. Third, the linearization of eigenvectors provides tractable characterizations of spectral embedding that serve as the starting point of statistical inference. Last but not least, it would be nice to generalize the results for contextual community detection to multi-class and imbalanced settings. That is of great practical importance.

\section*{Acknowledgements}
EA was supported by the NSF CAREER Award CCF-1552131. JF was supported by the ONR grant N00014-19-1-2120 and NSF grants DMS-2052926, DMS-1712591, and DMS-2053832. KW was supported by a startup fund from Columbia University and the NIH grant 2R01-GM072611-15 when he was a student at Princeton University.

\newpage
\appendix

\section{Useful facts}

Here we list some elementary results about operations using the new notations $O_{\PP}(\cdot;~\cdot)$ and $o_{\PP}(\cdot;~\cdot)$. Most of them can be found in \cite{Wan19}.

\begin{fact}\label{KPCA-fact-0}  The following two statements hold.
\begin{enumerate}
\item $X_n = O_{\PP} (Y_n;~ r_n)$ is equivalent to the following: there exist positive constants $C_1$, $C_2$ and $N$, a non-decreasing function $f:~[C_2,+\infty) \to (0,+\infty)$ satisfying $\lim_{x \to +\infty} f(x) = +\infty$, and a positive deterministic sequence $\{ R_n \}_{n=1}^{\infty}$ tending to infinity such that
	\begin{align*}
	\PP ( |X_n| \geq t |Y_n| ) \leq C_1 e^{- r_n f(t) }, \qquad \forall ~n \geq N,~ C_2 \leq  t \leq R_n. 
	\end{align*}
\item When $X_n = o_{\PP} (Y_n;~ r_n)$, we have
\begin{align*}
\lim_{n \to \infty} r_n^{-1} \log \PP ( |X_n| \geq c |Y_n| ) = - \infty
\end{align*}
for any constant $c > 0$. Here we adopt the convention $\log 0 = -\infty$.
\end{enumerate}
\end{fact}

\begin{fact}[Truncation]\label{KPCA-fact-truncation}
	If $X_n \mathbf{1}_{ \{ | Z_n | \leq | W_n | \} } = O_{\PP} (Y_n ;~ r_n)$ and $Z_n = o_{\PP} ( W_n;~ r_n )$, then
	\begin{align*}
	X_n = O_{\PP} (Y_n ;~ r_n).
	\end{align*}
\end{fact}
Fact \ref{KPCA-fact-truncation} directly follows from Fact \ref{KPCA-fact-0} above and Lemma 4 in \cite{Wan19}.

\begin{fact}\label{KPCA-fact-moment-tail}
	If $\EE^{1/r_n} |X_n|^{r_n} \lesssim Y_n$ or $\EE^{1/r_n} |X_n|^{r_n} \ll Y_n$ for deterministic $Y_n$, then $X_n=O_{\PP} ( Y_n ;~r_n )$ or $X_n=o_{\PP} ( Y_n ;~r_n )$, respectively.

\end{fact}

\begin{fact}[Lemma 2 in \cite{Wan19}]\label{KPCA-fact-arithmetic}
	If $X_n = O_{\PP} (Y_n ;~ r_n )$ and $W_n = O_{\PP} (Z_n ;~ s_n )$, then
	\begin{align*}
	& X_n + W_n = O_{\PP} ( |Y_n| + |Z_n| ;~ r_n \wedge s_n ) , \\
	& X_n W_n = O_{\PP} ( Y_n Z_n ;~ r_n \wedge s_n ).
	\end{align*}
\end{fact}

\begin{fact}[Lemma 3 in \cite{Wan19}]\label{KPCA-fact-transform}
	We have the followings:
	\begin{enumerate}
		\item if $X_n = {O}_{\PP} (Y_n ;~ r_n)$, then $|X_n|^{\alpha} = {O}_{\PP} ( |Y_n|^{\alpha} ;~ r_n)$ for any $\alpha > 0$;
		\item if $X_n = {o}_{\PP} (1 ;~ r_n)$, then $f(X_n) = {o}_{\PP} ( 1 ;~ r_n )$ for any $f:~\RR\to \RR$ that is continuous at $0$.
	\end{enumerate}
\end{fact}

\begin{definition}[A uniform version of $O_{\PP}(\cdot,~\cdot)$]\label{KPCA-defn-1}
	Let $\{ \Lambda_n \}_{n=1}^{\infty}$ be a sequence of finite index sets. For any $n \geq 1$, $\{ X_{n\lambda} \}_{\lambda \in \Lambda_n}$, $\{ Y_{n\lambda} \}_{\lambda \in \Lambda_n}$ are two collections of random variables; $\{ r_{n\lambda} \}_{\lambda \in \Lambda_n} \subseteq (0,+\infty)$ are deterministic. We write
	\begin{align}
	\{ X_{n\lambda} \}_{\lambda \in \Lambda_n} = O_{\PP} ( \{ Y_{n\lambda} \}_{\lambda \in \Lambda_n};~ \{ r_{n\lambda} \}_{\lambda \in \Lambda_n} )
	\label{eqn-definition-O-uniform}
	\end{align}
	if there exist positive constants $C_1$, $C_2$ and $N$, a non-decreasing function $f:~[C_2,+\infty) \to (0,+\infty)$ satisfying $\lim_{x \to +\infty} f(x) = +\infty$, and a positive deterministic sequence $\{ R_n \}_{n=1}^{\infty}$ tending to infinity such that
	\begin{align*}
	\PP ( |X_n| \geq t |Y_n| ) \leq C_1 e^{- r_n f(t) }, \qquad \forall ~n \geq N,~ C_2 \leq  t \leq R_n. 
	\end{align*}
	When $Y_{n\lambda} = Y_n$ and/or $r_{n\lambda} = r_n$ for all $n$ and $\lambda$, we may replace
	$ \{ Y_{n\lambda} \}_{\lambda \in \Lambda_n}$ and/or $\{ r_{n\lambda} \}_{\lambda \in \Lambda_n} $ in (\ref{eqn-definition-O-uniform}) by $Y_n$ and/or $r_n$ for simplicity.
\end{definition}

\begin{fact}\label{KPCA-fact-union}
	If $r_n \gtrsim \log |\Lambda_n|$, then $\{ X_{n\lambda} \}_{\lambda \in \Lambda_n } = O_{\PP} ( \{ Y_{n\lambda} \}_{\lambda \in \Lambda_n } ;~ r_n)$ implies that
	\[
	\max_{\lambda \in \Lambda_n} |X_{n\lambda}| = O_{\PP} ( \max_{\lambda \in \Lambda_n} Y_{n\lambda} ;~ r_n ).
	\]
\end{fact}

\section{More on $\ell_{2,p}$ analysis of eigenspaces}\label{KPCA-sec-general}

In this section, we provide a generalized version of Theorem \ref{KPCA-corollary-main} and its proof.
Instead of Assumption \ref{KPCA-assumption-incoherence}, we use a weaker version of that (Assumption \ref{KPCA-assumption-spectral}) at the cost of a more nested regularity condition for $p = p_n$ (Assumption \ref{KPCA-assumption-Lp}). Assumptions \ref{KPCA-assumption-noise} and \ref{KPCA-assumption-concentration} are still in use.

\begin{assumption}[Incoherence]\label{KPCA-assumption-spectral}
	$n \to \infty$ and $ \| \bar\bG \|_{2, \infty} / \bar\Delta \leq \gamma \ll 1 / \kappa$.
\end{assumption}

\begin{assumption}[Regularity of $p = p_n$]\label{KPCA-assumption-Lp}
	$\sqrt{n p} \| \bar\bX \bSigma^{1/2} \|_{2,p} \lesssim \bar\Delta \| \bar\bU\|_{2,p}$ and
\[
n^{1/p}  \sqrt{r p} \max\{ \| \bSigma \|_{\mathrm{HS}},~\sqrt{n} \| \bSigma \|_{\mathrm{op}}  \}
\lesssim  \bar\Delta \| \bar\bU \|_{2,p}.
\]
\end{assumption}


\begin{theorem}\label{KPCA-theorem-main}
	Let Assumptions \ref{KPCA-assumption-noise}, \ref{KPCA-assumption-concentration}, \ref{KPCA-assumption-spectral} and \ref{KPCA-assumption-Lp} hold. We have
	\begin{align*}
	&\| \bU \sgn(\bH) \|_{2,p} = O_{\PP} \left(  \| \bar\bU \|_{2,p} +  \gamma \bar\Delta^{-1} \| \bar\bG \|_{2,p} ;~p \wedge n \right) , \\
	&\| \bU \sgn(\bH) - \bG\bar\bU \bar\bLambda^{-1} \|_{2,p} = O_{\PP} \left( \kappa \gamma \| \bar\bU \|_{2,p} + \gamma \bar\Delta^{-1} \| \bar\bG \|_{2,p} ;~p \wedge n \right)
	,\\&
	\| \bU \bLambda^{1/2} \sgn(\bH) - \bG\bar\bU \bar\bLambda^{-1/2} \|_{2,p} = O_{\PP} ( \kappa^{3/2} \gamma \bar\Delta^{1/2} \| \bar\bU \|_{2,p} + \kappa^{1/2} \gamma \bar\Delta^{-1/2} \| \bar\bG \|_{2,p} ;~p\wedge n).
	\end{align*}
\end{theorem}

\subsection{Proof of Theorem \ref{KPCA-theorem-main}}\label{KPCA-sec-proof-theorem-main}

The following lemmas provide useful intermediate results, whose proofs can be found in Sections \ref{KPCA-proof-KPCA-lemma-L2} and \ref{KPCA-sec-proof-lemma-Lp_prelim}.

\begin{lemma}\label{KPCA-lemma-L2}
	Let Assumptions \ref{KPCA-assumption-noise}, \ref{KPCA-assumption-concentration} and \ref{KPCA-assumption-spectral} hold. We have $\| \bG -  \bar\bG \|_2 = O_{\PP} ( \gamma \bar\Delta;~n )$, $\| \bLambda - \bar\bLambda \|_2 = O_{\PP}( \gamma \bar\Delta;~n )$ and $ \| \bU \bU^{\top} - \bar\bU \bar\bU^{\top}  \|_2 = O_{\PP}( \gamma;~n )$.
\end{lemma}

\begin{lemma}\label{KPCA-lemma-Lp_prelim}
	Let Assumptions \ref{KPCA-assumption-noise}, \ref{KPCA-assumption-concentration}, \ref{KPCA-assumption-spectral} and \ref{KPCA-assumption-Lp} hold. We have
	\begin{align*}
	& \| \bG \bar\bU -  \bar\bU \bar\bLambda - \cH ( \bZ \bX^{\top} ) \bar\bU \|_{2,p}
	=  (\gamma + \sqrt{r/n} ) O_{\PP}( \bar\Delta \| \bar\bU \|_{2,p};~p), \\
	& \|\cH(\bZ \bX^{\top}) \bar\bU \|_{2,p} = O_{\PP} \Big(
	\sqrt{n p} \| \bar\bX \bSigma^{1/2} \|_{2,p} +
	n^{1/p} \sqrt{r p} \max\{  \| \bSigma \|_{\mathrm{HS}},~\sqrt{n} \| \bSigma \|_{\mathrm{op}}  \}
	;~p \wedge n\Big) , \\
	& \| \bG \bar\bU \bar\bLambda^{-1} \|_{2,p} = O_{\PP} ( \| \bar\bU \|_{2,p};~p \wedge n).
	\end{align*}
\end{lemma}

We now prove Theorem \ref{KPCA-theorem-main}. Let $\bar\gamma = \| \bG - \bar\bG \|_2 / \bar\Delta$. It follows from Lemma 1 in \cite{AFW17} that when $\bar\gamma \leq 1/10$,
\begin{align*}
&\| \bU \bH - \bG\bar\bU \bar\bLambda^{-1} \|_{2,p} \leq 6 \bar\gamma \bar\Delta^{-1} \| \bG \bar\bU \|_{2,p} + 2 \bar\Delta^{-1} \| \bG ( \bU \bH - \bar\bU ) \|_{2,p}.
\end{align*}
By Lemma \ref{KPCA-lemma-L2} and $\gamma \to 0$ in Assumption \ref{KPCA-assumption-spectral},
$\bar\gamma = O_{\PP}(\gamma;~n) = o_{\PP}(1;~n)$. Lemma \ref{KPCA-lemma-Lp_prelim} asserts that $ \| \bG \bar\bU \|_{2,p} \leq \| \bG \bar\bU \bar\bLambda^{-1} \|_{2,p} \| \bar\bLambda \|_2 =  O_{\PP} ( \kappa \bar\Delta \| \bar\bU \|_{2,p};~p \wedge n)$, respectively. Hence
\begin{align}
\| \bU \bH - \bG\bar\bU \bar\bLambda^{-1} \|_{2,p} 
&= O_{\PP}( \kappa \gamma \| \bar\bU \|_{2,p} ;~p \wedge n) +   \| \bG ( \bU \bH - \bar\bU ) \|_{2,p} O_{\PP}( \bar\Delta^{-1} ;~n ),
\label{KPCA-eqn-thm-main}\\
\| \bU \bH \|_{2,p} 
& \leq \| \bG\bar\bU \bar\bLambda^{-1} \|_{2,p} + \| \bU \bH - \bG\bar\bU \bar\bLambda^{-1} \|_{2,p} = O_{\PP}(  \| \bar\bU \|_{2,p} ;~p \wedge n) \notag\\
& + \| \bG ( \bU \bH - \bar\bU ) \|_{2,p} O_{\PP}( \bar\Delta^{-1} ;~n ) .
\label{KPCA-eqn-thm-main-0}
\end{align}

	We construct leave-one-out auxiliaries $\{ \bG^{(m)} \}_{m=1}^n \subseteq \RR^{n \times n}$ where $\bG^{(m)}$ is obtained by zeroing out the $m$-th row and column of $\bG$. Mathematically, we define a new data matrix
	\[
	\bX^{(m)} = (\bx_1,\cdots,\bx_{m-1} , \mathbf{0} , \bx_{m+1} , \cdots , \bx_n )^{\top} = (\bI_n - \be_m \be_m^{\top}) \bX
	\]
	by deleting the $m$-th sample and
	\[
	\bG^{(m)} = \cH ( \bX^{(m)} \bX^{(m)\top} ) = \cH [ (\bI_n - \be_m \be_m^{\top}) \bX \bX^{\top}
	(\bI_n - \be_m \be_m^{\top}) ].
	\]
	Let $\{ \bu^{(m)}_j \}_{j=1}^n$ be the eigenvectors of $\bG^{(m)}$, $\bU^{(m)} = (\bu^{(m)}_{s+1} , \cdots, \bu^{(m)}_{s+r} ) \in \RR^{n\times r}$ and $\bH^{(m)} = \bU^{(m)\top} \bar\bU $. The construction is also used by \cite{AFW17} in entrywise eigenvector analysis.

By Minkowski's inequality,
\begin{align}
& \| \bG ( \bU \bH - \bar\bU ) \|_{2,p} 
 \leq \bigg( \sum_{m=1}^{n} [ \| \bG_m (\bU \bH - \bU^{(m)} \bH^{(m)} ) \|_2 + \|  \bG_m (\bU^{(m)} \bH^{(m)} - \bar\bU) \|_2 ]^p \bigg)^{1/p} \notag\\
& \leq \bigg( \sum_{m=1}^{n} \| \bG_m (\bU \bH - \bU^{(m)} \bH^{(m)} ) \|_2^p \bigg)^{1/p}  +  \bigg( \sum_{m=1}^{n} \| \bG_m (\bU^{(m)} \bH^{(m)} - \bar\bU) \|_2^p \bigg)^{1/p}.
\label{KPCA-eqn-thm-main-1}
\end{align}

The first term on the right hand side of (\ref{KPCA-eqn-thm-main-1}) corresponds to leave-one-out perturbations. When $\max\{ \| \bar\bG \|_{2,\infty}, \| \bG - \bar\bG \|_2 \} \kappa \leq \bar\Delta /32$, Lemma 3 in \cite{AFW17} forces
\begin{align*}
& \| \bU \bU^{\top} - \bU^{(m)} (\bU^{(m)})^{\top} \|_2 \leq 3 \kappa \| ( \bU \bH )_{m} \|_2,\qquad \forall m \in [n], \notag \\
& \max_{m\in[n]} \| \bU^{(m)} \bH^{(m)} - \bar\bU \|_2 \leq 6 \max\{ \| \bar\bG \|_{2,\infty}, \| \bG - \bar\bG \|_2 \}  / \bar\Delta .
\end{align*}
The fact $\| \bar\bG \|_{2,\infty} \leq \gamma \bar\Delta$, the result $\| \bG - \bar\bG \|_2 = O_{\PP} (\gamma \bar\Delta;~n)$ in Lemma \ref{KPCA-lemma-L2}, and Assumption \ref{KPCA-assumption-spectral} imply that
\begin{align}
&\| \bG \|_{2,\infty} \leq \| \bar\bG \|_{2,\infty} + \| \bG - \bar\bG \|_2 = O_{\PP} ( \gamma\bar\Delta;~n),\notag \\
&\bigg( \sum_{m=1}^{n} \| \bU \bU^{\top} - \bU^{(m)} (\bU^{(m)})^{\top} \|_2^p \bigg)^{1/p}
= O_{\PP} ( \kappa \| \bU \bH \|_{2,p};~n ), \notag\\
& \max_{m\in[n]} \| \bU^{(m)} \bH^{(m)} - \bar\bU \|_2 = O_{\PP} ( \gamma;~n).
\label{KPCA-eqn-thm-main-1.5}
\end{align}
The definitions $\bH = \bU^{\top} \bar\bU $ and $\bH^{(m)} = ( \bU^{(m)})^{\top} \bar\bU $ yield
\begin{align*}
\| \bU \bH - \bU^{(m)} \bH^{(m)} \|_2 = \| ( \bU \bU^{\top} - \bU^{(m)} (\bU^{(m)})^{\top} ) \bar\bU \|_2 \leq \| \bU \bU^{\top} - \bU^{(m)} (\bU^{(m)})^{\top} \|_2.
\end{align*}
Based on these estimates,
\begin{align}
& \bigg( \sum_{m=1}^{n} \| \bG_m (\bU \bH - \bU^{(m)} \bH^{(m)} ) \|_2^p \bigg)^{1/p} 
\leq \| \bG \|_{2,\infty} \bigg( \sum_{m=1}^{n} \| \bU \bH - \bU^{(m)} \bH^{(m)}\|_2^p \bigg)^{1/p} \notag\\
& \leq \| \bG \|_{2,\infty} \bigg( \sum_{m=1}^{n} \| \bU \bU^{\top} - \bU^{(m)} (\bU^{(m)})^{\top} \|_2^p \bigg)^{1/p}
= O_{\PP} ( \kappa \gamma \bar\Delta \| \bU \bH \|_{2,p};~n )  \notag\\
& = O_{\PP} ( \kappa \gamma \bar\Delta \| \bar\bU \|_{2,p};~p\wedge n ) +  O_{\PP} ( \kappa \gamma  \| \bG ( \bU \bH - \bar\bU) \|_{2,p} ;~n).
\label{KPCA-eqn-thm-main-2}
\end{align}
The last equality follows from (\ref{KPCA-eqn-thm-main-0}). We use (\ref{KPCA-eqn-thm-main-1}), (\ref{KPCA-eqn-thm-main-2}) and $\kappa \gamma = o(1)$ from Assumption \ref{KPCA-assumption-spectral} to derive
\begin{align*}
\| \bG ( \bU \bH - \bar\bU) \|_{2,p} \leq  \bigg( \sum_{m=1}^{n} \| \bG_m (\bU^{(m)} \bH^{(m)} - \bar\bU) \|_2^p \bigg)^{1/p} +  O_{\PP} ( \kappa \gamma \bar\Delta \| \bar\bU \|_{2,p};~p\wedge n ).
\end{align*}
By plugging this into (\ref{KPCA-eqn-thm-main}) and (\ref{KPCA-eqn-thm-main-0}) and using $\kappa \gamma = o(1)$, we obtain that
\begin{align}
&\| \bU \bH - \bG\bar\bU \bar\bLambda^{-1} \|_{2,p} =  O_{\PP}( \kappa \gamma \| \bar\bU \|_{2,p} ;~p \wedge n) +  \bigg( \sum_{m=1}^{n} \| \bG_m (\bU^{(m)} \bH^{(m)} - \bar\bU) \|_2^p \bigg)^{1/p}
O_{\PP} (\bar\Delta^{-1};~n), \label{KPCA-eqn-thm-main-3} \\
&\| \bU\bH \|_{2,p} = O_{\PP}(\| \bar\bU \|_{2,p} ;~p \wedge n) + \bigg( \sum_{m=1}^{n} \| \bG_m (\bU^{(m)} \bH^{(m)} - \bar\bU) \|_2^p \bigg)^{1/p}
O_{\PP} (\bar\Delta^{-1};~n). \label{KPCA-eqn-thm-main-3.5}
\end{align}

We now control the second term in \eqref{KPCA-eqn-thm-main-1}.
From the decompositions
\begin{align*}
& \bG = \cH [ (\bar\bX + \bZ)  (\bar\bX + \bZ)^{\top} ]  = \cH(\bar\bX \bar\bX^{\top} + \bar\bX\bZ^{\top} + \bZ \bar\bX^{\top} ) +  \cH( \bZ \bZ^{\top} ),
\end{align*}
we have
\begin{align}
&\bigg( \sum_{m=1}^{n} \| \bG_m (\bU^{(m)} \bH^{(m)} - \bar\bU) \|_2^p \bigg)^{1/p}
\notag\\ &
 \leq~ 
 \| \cH(\bar\bX \bar\bX^{\top} + \bar\bX\bZ^{\top} + \bZ \bar\bX^{\top} ) \|_{2,p} \max_{m\in[n]} \| \bU^{(m)} \bH^{(m)} - \bar\bU \|_2 \notag\\
& + \bigg( \sum_{m=1}^{n} \| [  \cH(\bZ \bZ^{\top}) ]_m (\bU^{(m)} \bH^{(m)} - \bar\bU) \|_2^p \bigg)^{1/p}.
\label{KPCA-eqn-thm-main-4}
\end{align}

We now work on the first term on the right hand side of (\ref{KPCA-eqn-thm-main-4}). Define $\bM \in \R^{n\times n}$ through $M_{ij}  = \| ( \bar\bX \bZ^{\top})_{ij} \|_{\psi_2}$. Then $\EE M_{ij} = 0$ and $M_{ij} = \| \langle \bar\bx_i, \bz_j \rangle \|_{\psi_2} \lesssim \| \bSigma^{1/2} \bar\bx_i \|$, where $\lesssim$ only hides a universal constant. 
\begin{align*}
 \| \bM \|_{2,p}
& = \bigg[ \sum_{i=1}^{n} \bigg(
\sum_{j=1}^{n} |M_{ij}|^2
\bigg)^{p/2}
\bigg]^{1/p}
\lesssim  \bigg[ \sum_{i=1}^{n} \bigg(
\sum_{j=1}^{n}  \| \bSigma^{1/2} \bar\bx_i \|^2
\bigg)^{p/2}
\bigg]^{1/p}
= \sqrt{n} \| \bar\bX \bSigma^{1/2} \|_{2,p},\\
  \| \bM^{\top} \|_{2,p}
&= \bigg[ \sum_{j=1}^{n} \bigg(
\sum_{i=1}^{n} |M_{ij}|^2
\bigg)^{p/2}
\bigg]^{1/p} \lesssim  \bigg[ \sum_{j=1}^{n} \bigg(
\sum_{i=1}^{n}  \| \bSigma^{1/2} \bar\bx_i \|^2
\bigg)^{p/2}
\bigg]^{1/p}
\notag\\ & 
= n^{1/p} \| \bar\bX \bSigma^{1/2} \|_{2,2}
\leq \sqrt{n} \| \bar\bX \bSigma^{1/2} \|_{2,p}.
\end{align*}
By Lemma \ref{KPCA-lemma-Lp-gaussian} and $p \geq 2$,
\begin{align*}
&\| \bar\bX \bZ^{\top} \|_{2,p} = O_{\PP} (
\sqrt{p} \| \bM \|_{2,p};~p
) = O_{\PP} (
\sqrt{np} \| \bar\bX \bSigma^{1/2} \|_{2,p};~p
), \\
&\|\bZ  \bar\bX^{\top} \|_{2,p} = O_{\PP} (
\sqrt{p} \| \bM^{\top} \|_{2,p};~p
) = O_{\PP} (
\sqrt{np} \| \bar\bX \bSigma^{1/2} \|_{2,p};~p
).
\end{align*}
These estimates and $\sqrt{n p} \| \bar\bX \bSigma^{1/2} \|_{2,p} \lesssim \bar\Delta \| \bar\bU\|_{2,p}$ in Assumption \ref{KPCA-assumption-Lp} yield
\begin{align*}
&\| \cH ( \bar\bX \bZ^{\top} + \bZ \bar\bX^{\top} ) \|_{2,p} \leq
\|  \bar\bX \bZ^{\top} + \bZ \bar\bX^{\top} \|_{2,p} = O_{\PP} (  \bar\Delta \| \bar\bU \|_{2,p} ;~p).
\end{align*}
This and (\ref{KPCA-eqn-thm-main-1.5}) lead to
\begin{align}
& \| \cH(\bar\bX \bar\bX^{\top} + \bar\bX\bZ^{\top} + \bZ \bar\bX^{\top} ) \|_{2,p} \max_{m\in[n]} \| \bU^{(m)} \bH^{(m)} - \bar\bU \|_2 \notag\\
& = O_{\PP} ( \gamma ( \| \bar\bX \bar\bX^{\top} \|_{2,p} + \bar\Delta \| \bar\bU \|_{2,p} ) ;~p \wedge n).
\label{KPCA-eqn-thm-main-5}
\end{align}
We use (\ref{KPCA-eqn-thm-main-3}), (\ref{KPCA-eqn-thm-main-4}) and (\ref{KPCA-eqn-thm-main-5}) to get
\begin{align}
\| \bU \bH - \bG\bar\bU \bar\bLambda^{-1} \|_{2,p} =~& O_{\PP}( \kappa \gamma \| \bar\bU \|_{2,p} ;~p \wedge n) + O_{\PP}(  \gamma \bar\Delta^{-1} \| \bar\bX \bar\bX^{\top} \|_{2,p} ;~p \wedge n)
\notag\\ &
+ \bigg( \sum_{m=1}^{n} \| [  \cH(\bZ \bZ^{\top}) ]_m (\bU^{(m)} \bH^{(m)} - \bar\bU) \|_2^p \bigg)^{1/p} O_{\PP} (\bar\Delta^{-1};~n).
\label{KPCA-eqn-thm-main-6}
\end{align}

By construction, $\bU^{(m)} \bH^{(m)} - \bar\bU \in \R^{n\times r}$ is independent of $\bz_m$. We invoke Lemma \ref{KPCA-lemma-Z-product} to get
\begin{align}
& \bigg( \sum_{m=1}^{n} \| [  \cH(\bZ \bZ^{\top}) ]_m (\bU^{(m)} \bH^{(m)} - \bar\bU) \|_2^p \bigg)^{1/p} =
\bigg(
\sum_{m=1}^{n} \bigg\| \sum_{j\neq m}  \langle \bz_m,\bz_j \rangle (\bU^{(m)} \bH^{(m)} - \bar\bU)_j \bigg\|_2^p 
\bigg)^{1/p} \notag\\
& = n^{1/p} \max_{m\in[n]}  \| \bU^{(m)} \bH^{(m)} - \bar\bU  \|_2   O_{\PP} \left(
\sqrt{r p} \max\{ \| \bSigma \|_{\mathrm{HS}},~\sqrt{n} \| \bSigma \|_{\mathrm{op}} \}
;~p \wedge n \right) \notag\\
&
= O_{\PP} (  \gamma \bar\Delta \| \bar\bU \|_{2,p} ;~p\wedge n ),
\label{KPCA-eqn-thm-main-8}
\end{align}
where we also used (\ref{KPCA-eqn-thm-main-1.5}) and Assumption \ref{KPCA-assumption-Lp}.

We use (\ref{KPCA-eqn-thm-main-6})
and (\ref{KPCA-eqn-thm-main-8}) to derive
\begin{align}
&\| \bU \bH - \bG\bar\bU \bar\bLambda^{-1} \|_{2,p} =  O_{\PP}( \kappa \gamma \| \bar\bU \|_{2,p} ;~p \wedge n) + O_{\PP}(  \gamma \bar\Delta^{-1} \| \bar\bG \|_{2,p} ;~p \wedge n).
\label{KPCA-eqn-thm-main-9}
\end{align}
Consequently, Lemma \ref{KPCA-lemma-Lp_prelim} yields
\begin{align}
\| \bU \bH \|_{2,p} & \leq \| \bG\bar\bU \bar\bLambda^{-1} \|_{2,p} + \| \bU \bH - \bG\bar\bU \bar\bLambda^{-1} \|_{2,p} \notag\\
& = O_{\PP}(  \| \bar\bU \|_{2,p} ;~p \wedge n) + O_{\PP}(  \gamma \bar\Delta^{-1} \| \bar\bG \|_{2,p} ;~p \wedge n) .
\label{KPCA-eqn-thm-main-10}
\end{align}

Lemma 2 in \cite{AFW17} and the result $\| \bG - \bar\bG \|_2 = O_{\PP} (\gamma \bar\Delta;~n)$ in Lemma \ref{KPCA-lemma-L2} imply that $\| \bH - \sgn(\bH) \|_2 = O_{\PP} ( \gamma^2;~n)$. As $\sgn(\bH)$ is orthonormal, we have $\| \bH^{-1} \|_2 = O_{\PP} (1,~n)$ and
\begin{align}
\| \bU \sgn(\bH) - \bU \bH \|_{2,p} 
& \leq \| \bU \bH  \bH^{-1} ( \sgn(\bH) - \bH ) \|_{2,p} \notag\\
& \leq \| \bU \bH \|_{2,p} \| \bH^{-1}\|_2 \| \sgn(\bH) - \bH \|_2 = \| \bU \bH \|_{2,p} O_{\PP} ( \gamma^2;~n ).
\label{KPCA-eqn-thm-main-11}
\end{align}
The tail bounds for $\| \bU \sgn(\bH) \|_{2,p}$ and $\| \bU \sgn(\bH) - \bG \bar\bU \bar\bLambda^{-1} \|_{2,p}$ in Theorem \ref{KPCA-theorem-main} follow from (\ref{KPCA-eqn-thm-main-9}), (\ref{KPCA-eqn-thm-main-10}) and (\ref{KPCA-eqn-thm-main-11}).

Finally we use the results above to control $\| \bU \bLambda^{1/2} \sgn(\bH) - \bG \bar\bU \bar\bLambda^{-1/2} \|_{2,p}$. By Lemma \ref{KPCA-lemma-L2},
$\| \bLambda - \bar\bLambda \|_2 \leq \| \bG - \bar\bG \|_2 = O_{\PP} (\gamma \bar\Delta;~n) = o_{\PP} (\bar\Delta;~n)$.
Hence $n^{-1} \log \PP ( \| \bLambda - \bar\bLambda \|_2 \geq \bar\Delta / 2 ) \to -\infty$.
When $\| \bG - \bar\bG \|_2 < \bar\Delta / 2$, we have $\bLambda \succ (\bar\Delta / 2) \bI$, and $\bLambda^{1/2}$ is well-defined. It remains to show that
\begin{align}
& \| \bU \bLambda^{1/2} \bar\bH - \bG \bar\bU \bar\bLambda^{-1/2} \|_{2,p} \mathbf{1}_{ \{ \| \bG - \bar\bG \|_2 < \bar\Delta / 2 \} } \notag\\
& = O_{\PP} ( \kappa^{3/2} \gamma \bar\Delta^{1/2} \| \bar\bU \|_{2,p} + \kappa^{1/2} \gamma \bar\Delta^{-1/2} \| \bar\bG \|_{2,p} ;~p\wedge n).
\label{KPCA-eqn-thm-main-12}
\end{align}

Define $\bar\bH = \sgn(\bH)$. When $\| \bG - \bar\bG \|_2 < \bar\Delta / 2$ happens, we use triangle's inequality to derive
\begin{align*}
\| \bU \bLambda^{1/2} \bar\bH - \bG \bar\bU \bar\bLambda^{-1/2} \|_{2,p}
& \leq \| \bU \bar\bH ( \bar\bH^{\top} \bLambda^{1/2} \bar\bH - \bar\bLambda^{1/2} ) \|_{2,p}
+ \| ( \bU \bar\bH - \bG \bar\bU \bar\bLambda^{-1} ) \bar\bLambda^{1/2} \|_{2,p} \notag \\
& \leq \| \bU \bar\bH \|_{2,p} \|  \bar\bH^{\top} \bLambda^{1/2} \bar\bH - \bar\bLambda^{1/2} \|_2
+ \| \bU \bar\bH - \bG \bar\bU \bar\bLambda^{-1} \|_{2,p} \| \bar\bLambda\|_2^{1/2}.
\end{align*}
It is easily seen from $\| \bar\bLambda \|_2 \leq \kappa \bar\Delta$ that
\begin{align*}
\| \bU \bar\bH - \bG \bar\bU \bar\bLambda^{-1} \|_{2,p} \| \bar\bLambda\|_2^{1/2}
= O_{\PP} ( \kappa^{3/2} \gamma \bar\Delta^{1/2} \| \bar\bU \|_{2,p} + \kappa^{1/2} \gamma \bar\Delta^{-1/2} \| \bar\bG \|_{2,p} ;~p\wedge n).
\end{align*}
Hence
\begin{align}
& \| \bU \bLambda^{1/2} \bar\bH - \bG \bar\bU \bar\bLambda^{-1/2} \|_{2,p}  \mathbf{1}_{ \{ \| \bG - \bar\bG \|_2 < \bar\Delta / 2 \} } 
= O_{\PP} ( \kappa^{3/2} \gamma \bar\Delta^{1/2} \| \bar\bU \|_{2,p} + \kappa^{1/2} \gamma \bar\Delta^{-1/2} \| \bar\bG \|_{2,p} ;~p\wedge n) \notag\\
&+ O_{\PP} (  \| \bar\bU \|_{2,p} +  \gamma \bar\Delta^{-1} \| \bar\bG \|_{2,p} ;~p\wedge n) \cdot \|  \bar\bH^{\top} \bLambda^{1/2} \bar\bH - \bar\bLambda^{1/2} \|_2  \mathbf{1}_{ \{ \| \bG - \bar\bG \|_2 < \bar\Delta / 2 \} } .
\label{KPCA-eqn-thm-main-13}
\end{align}

Note that $ \bar\bH^{\top} \bLambda^{1/2} \bar\bH = ( \bar\bH^{\top} \bLambda \bar\bH)^{1/2}$. In view of the perturbation bound for matrix square roots \citep[Lemma 2.1]{Sch92},
\begin{align*}
\|  \bar\bH^{\top} \bLambda^{1/2} \bar\bH - \bar\bLambda^{1/2} \|_2
& \leq \frac{
	\| \bar\bH^{\top} \bLambda \bar\bH - \bar\bLambda \|_2
}{ \lambda_{\min} ( \bar\bH^{\top} \bLambda^{1/2} \bar\bH ) + \lambda_{\min} (\bar\bLambda^{1/2}) }
\leq \frac{
	\| \bLambda \bar\bH - \bar\bH \bar\bLambda \|_2
}{ 2 \bar\Delta^{1/2} } \notag \\
& \lesssim (
\| \bLambda \bH - \bH \bar\bLambda \|_2
+ \| \bLambda (\bar\bH - \bH) \|_2 +  \| (\bar\bH - \bH) \bar\bLambda \|_2
) /  \bar\Delta^{1/2} \notag\\
&\lesssim \| \bLambda \bH - \bH \bar\bLambda \|_2 /  \bar\Delta^{1/2} 
+   O_{\PP} ( \kappa \gamma^2 \bar\Delta^{1/2} ;~n )
\end{align*}
as long as $\| \bG - \bar\bG \|_2 < \bar\Delta / 2$. Here we used $\| \bH - \bar\bH \|_2 = O_{\PP} (\gamma^2;~n)$ according to Lemma \ref{KPCA-lemma-L2} as well as Lemma 2 in \cite{AFW17}.

From $\bU^{\top} \bG = \bLambda \bU^{\top}$ and $\bar\bG \bar\bU = \bar\bU\bar\bLambda$ we obtain that
\begin{align*}
\bLambda \bH - \bH \bar\bLambda 
=  \bLambda \bU^{\top} \bar\bU - \bU^{\top} \bar\bU \bar\bLambda
= \bU^{\top} \bG \bar\bU - \bU^{\top} \bar\bG \bar\bU = \bU^{\top} (\bG - \bar\bG ) \bar\bU 
\end{align*}
and $\| \bLambda \bH - \bH \bar\bLambda \|_2 \leq \| \bG - \bar\bG \|_2 = O_{\PP} ( \gamma \bar\Delta;~n)$.
As a result,
\begin{align*}
\|  \bar\bH^{\top} \bLambda^{1/2} \bar\bH - \bar\bLambda^{1/2} \|_2 \mathbf{1}_{ \{ \| \bG - \bar\bG \|_2 < \bar\Delta / 2 \} } = O_{\PP} ( \gamma  \bar\Delta^{1/2};~n ),
\end{align*}
where we also used $\kappa \gamma = o(1)$ in Assumption \ref{KPCA-assumption-spectral}. Plugging this into (\ref{KPCA-eqn-thm-main-13}), we get the desired bound (\ref{KPCA-eqn-thm-main-12}) and thus complete the proof of Theorem \ref{KPCA-theorem-main}.

\subsection{Proof of Lemma \ref{KPCA-lemma-L2}}\label{KPCA-proof-KPCA-lemma-L2}
Note that
\begin{align}
\bG 
& = \cH [ (\bar\bX + \bZ)  (\bar\bX + \bZ)^{\top} ] = \cH ( \bar\bX \bar\bX^{\top} ) + \cH( \bar\bX\bZ^{\top} + \bZ \bar\bX^{\top} ) + \cH( \bZ\bZ^{\top} ) \notag \\
& = \bar\bX \bar\bX^{\top} + ( \bar\bX\bZ^{\top} + \bZ \bar\bX^{\top} ) + \cH( \bZ\bZ^{\top} )  - \bar\bD,
\label{KPCA-ineq-L2-decomposition}
\end{align}
where $\bar\bD$ is the diagonal part of $\bar\bX \bar\bX^{\top} + \bar\bX\bZ^{\top} + \bZ \bar\bX^{\top}$, with $\bar \bD_{ii} = \| \bar\bx_i \|^2 + 2 \langle \bar\bx_i , \bz_i \rangle$.
From $\| \langle \bar\bx_i , \bz_i \rangle \|_{\psi_2} \lesssim  \| \bSigma^{1/2} \bar\bx_i \|$ we get $\{ | \langle \bar\bx_i , \bz_i \rangle | \}_{i=1}^n =  O_{\PP} ( \{ \| \bSigma^{1/2} \bar\bx_i \| \sqrt{n} \}_{i=1}^n ;~ n )$. By Fact \ref{KPCA-fact-union},
\[
\max_{i\in[n]} | \langle \bar\bx_i , \bz_i \rangle | = O_{\PP} \Big( \max_{i\in[n]} \| \bSigma^{1/2} \bar\bx_i \| \sqrt{n};~ n \Big)
\]
and
\begin{align}
\| \bar\bD \|_2 &= \max_{i\in[n]} |\bar\bD_{ii}|
=  \max_{i\in[n]} \| \bar\bx_i\|^2 + O_{\PP}\left(
\max_{i\in[n]}  \| \bSigma^{1/2} \bar\bx_i \| \sqrt{n};~n\right) \notag\\
& = \| \bar\bX\|_{2,\infty}^2 + O_{\PP}\left( 
\| \bar\bX \|_{2,\infty} ( n \| \bSigma \|_{\mathrm{op}} )^{1/2};~n\right) \notag\\
& \leq \| \bar\bX \bar\bX^{\top} \|_{2,\infty} + O_{\PP}\left( 
\| \bar\bX \bar\bX^{\top} \|_{2}^{1/2} ( n \| \bSigma \|_{\mathrm{op}} )^{1/2};~n\right)
 \notag\\
& 
=  \| \bar\bG \|_{2,\infty} + O_{\PP}\left( 
( n \kappa \bar\Delta \| \bSigma \|_{\mathrm{op}} )^{1/2};~n\right)
.
\label{KPCA-ineq-L2-D-1}
\end{align}

Note that $\| \bZ \bar\bX^{\top} \|_2 = \sup_{ \bu ,\bv \in  \SSS^{n-1} } \bu^{\top} \bZ \bar\bX^{\top}  \bv$.
Since $\{ \bz_i^{\top} \bar\bX^{\top}  \bv \}_{i=1}^n$ are zero-mean, independent and
\begin{align*}
\| \bz_i^{\top} \bar\bX^{\top}  \bv \|_{\psi_2} \lesssim \| \bSigma^{1/2} \bar\bX^{\top}  \bv \| \leq \|  \bar\bX \bSigma^{1/2} \|_{\mathrm{op}}
\leq ( \| \bar\bG \|_2  \| \bSigma \|_{\mathrm{op}} )^{1/2} =  ( \kappa \bar\Delta  \| \bSigma \|_{\mathrm{op}} )^{1/2},
\end{align*}
we have
\[
\|  \bu^{\top} \bZ \bar\bX^{\top} \bv \|_{\psi_2} = \bigg\| \sum_{i=1}^{n} u_i \bz_i^{\top} \bar\bX^{\top} \bv \bigg\|_{\psi_2} \lesssim \bigg( \sum_{i=1}^{n} u_i^2 \| \bz_i^{\top} \bar\bX^{\top} \bv \|_{\psi_2}^2 \bigg)^{1/2} \lesssim  ( \kappa \bar\Delta  \| \bSigma \|_{\mathrm{op}} )^{1/2}.
\]
A standard covering argument \citep[Section 5.2.2]{Ver10} yields
\[
\| \bZ \bar\bX^{\top} \|_2 = O_{\PP} ( ( n \kappa \bar\Delta  \| \bSigma \|_{\mathrm{op}} )^{1/2} ;~n ).
\]
The same tail bound also holds for $\|  \bar\bX \bZ^{\top} \|_2$.

From these estimates, (\ref{KPCA-ineq-L2-decomposition}), (\ref{KPCA-ineq-L2-D-1}) and Lemma \ref{KPCA-lem-concentration-gram} we obtain that
\begin{align*}
\| \bG -  \bar\bX \bar\bX^{\top} \|_2 = O_{\PP} \left(
\| \bar\bG \|_{2,\infty}^2 + ( n \kappa \bar\Delta  \| \bSigma \|_{\mathrm{op}} )^{1/2} + \max\{
\sqrt{n} \| \bSigma \|_{\mathrm{HS}},~ n \| \bSigma \|_{\mathrm{op}}
\};~ n
\right).
\end{align*}
By Assumptions \ref{KPCA-assumption-spectral} and \ref{KPCA-assumption-concentration}, we have $ n \kappa  \| \bSigma \|_{\mathrm{op}} \leq \bar\Delta$. Hence $n \| \bSigma \|_{\mathrm{op}} \leq ( n \kappa \bar\Delta  \| \bSigma \|_{\mathrm{op}} )^{1/2}$ and $\| \bG -  \bar\bX \bar\bX^{\top} \|_2 = O_{\PP} ( \gamma \bar\Delta;~n )$.

Finally, Weyl's inequality \citep{SSu90} and Davis-Kahan theorem \citep{DKa70} assert that
$\| \bLambda - \bar\bLambda \|_2 \leq \| \bG - \bar\bG \|_2 = O_{\PP}( \gamma \bar\Delta;~n )$ and $\| \bU \bU^{\top} - \bar\bU \bar\bU^{\top} \|_2 \lesssim \| \bG - \bar\bG \|_2 / \bar\Delta = O_{\PP}( \gamma;~n )$.

\subsection{Proof of Lemma \ref{KPCA-lemma-Lp_prelim}}\label{KPCA-sec-proof-lemma-Lp_prelim}
Observe that
\begin{align*}
\bG 
& = \cH ( \bX \bX^{\top} ) = \cH [ (\bar\bX + \bZ)  \bX^{\top} ]  = \bar\bX \bar\bX^{\top}  + [ \cH(\bar\bX \bar\bX^{\top}) - \bar\bX \bar\bX^{\top} ] + \cH( \bar\bX\bZ^{\top}) + \cH( \bZ \bX^{\top} ) .
\end{align*}
From $ \bar\bX \bar\bX^{\top} \bar\bU = \bar\bG \bar\bLambda = \bar\bU \bar\bLambda $ we get
\begin{align*}
\| \bG \bar\bU - \bar\bU  \bar\bLambda - \cH ( \bZ \bX^{\top} ) \bar\bU \|_{2,p} 
& = \| \bG \bar\bU -\bar\bX \bar\bX^{\top} \bar\bU  - \cH ( \bZ \bX^{\top} ) \bar\bU \|_{2,p}  \\
& = \| [ \cH ( \bar\bX \bar\bX^{\top} ) - \bar\bX \bar\bX^{\top} + \cH (\bar\bX \bZ^{\top}) ] \bar\bU \|_{2,p} 
\\&
\leq\left( \sum_{m=1}^{n} ( \| \bar\bx_m \|^2 \|\bar \bU_m\|_2 )^p \right)^{1/p}
+ \|\cH(\bar\bX \bZ^{\top}) \bar\bU \|_{2,p}.
\end{align*}
On the one hand,  we have
\begin{align*}
\sum_{m=1}^{n} ( \| \bar\bx_m \|^2 \|\bar \bU_m\|_2 )^p \leq \max_{m\in[n]}\| \bar\bx_m \|^{2p} \sum_{m=1}^{n} \|\bar \bU_m\|_2^p 
= \| \bar\bX \|_{2,\infty}^{2p} \| \bar\bU \|_{2,p}^p
\leq ( \gamma \bar\Delta \| \bar\bU \|_{2,p} )^p,
\end{align*}
where we used $\| \bar\bX \|_{2,\infty}^2 \leq \| \bar\bX \bar\bX^{\top} \|_{2,\infty} \leq \gamma \bar\Delta $ in Assumption \ref{KPCA-assumption-spectral}. On the other hand, $\{ \bz_j \}_{j\neq m}$ are independent, $\| \langle \bar\bx_m, \bz_j  \rangle \|_{\psi_2} \lesssim \| \bSigma^{1/2} \bar\bx_m \|$, $\bar\bU = (\bar\bu_1,\cdots,\bar\bu_r)$ and $\| \bar\bu_j \|_2 = 1$ for $j \in [r]$.
Then
\begin{align*}
& \| [ \cH( \bar\bX \bZ^{\top} )  ]_{m} \bar\bu_j \|_{\psi_2}
= \| ( \bar\bX \bZ^{\top} )_{m} (\bI - \be_m \be_m^{\top}) \bar\bu_j \|_{\psi_2} \\
& = \bigg\| \sum_{k \neq m} \bar u_{jk} \langle \bar\bx_m, \bz_j \rangle \bigg\|_{\psi_2} \lesssim \| \bSigma^{1/2} \bar\bx_m \|,\qquad j\in [r],~  m \in [n].
\end{align*}
Lemma \ref{KPCA-lemma-Lp-gaussian} forces $\|\cH(\bar\bX \bZ^{\top}) \bar\bU \|_{2,p} = O_{\PP} ( \sqrt{p} \| \bM \|_{2,p};~p )$, where $M_{ij} = \| \bSigma^{1/2} \bar\bx_i \|$. Hence
\begin{align*}
& \| \bM \|_{2,p} = \bigg[ \sum_{i=1}^{n}  \bigg( \sum_{j=1}^{r} \| \bSigma^{1/2} \bar\bx_i \|^2 \bigg)^{p/2} \bigg]^{1/p} = \sqrt{r} \| \bar\bX \bSigma^{1/2} \|_{2,p},\\
&\|\cH(\bar\bX \bZ^{\top}) \bar\bU \|_{2,p} = O_{\PP} ( \sqrt{rp} \| \bar\bX \bSigma^{1/2} \|_{2,p};~p )
= O_{\PP} ( \sqrt{r/n} \bar\Delta \| \bar\bU\|_{2,p} ;~p ),
\end{align*}
where the last equality follows from Assumption \ref{KPCA-assumption-Lp}.
By combining the two parts we get 
\begin{align}
\| \bG \bar\bU - \bar\bU \bar\bLambda  - \cH ( \bZ \bX^{\top} ) \bar\bU \|_{2,p}
& = (\gamma + \sqrt{r/n} ) O_{\PP}( \bar\Delta \| \bar\bU \|_{2,p};~p), \notag\\
\| \bG \bar\bU \bar\bLambda^{-1} - \cH ( \bZ \bX^{\top} ) \bar\bU \bar\bLambda^{-1} \|_{2,p}
& \leq \| \bG \bar\bU -  \bar\bU \bar\bLambda - \cH ( \bZ \bX^{\top} ) \bar\bU \|_{2,p} \| \bar\bLambda^{-1} \|_2 + \| \bar\bU \|_{2,p}  \notag\\
& = O_{\PP} ( \| \bar\bU \|_{2,p} ;~p).
\label{KPCA-ineq-proof-lemma-Lp-prelim-0}
\end{align}

To study $\cH ( \bZ \bX^{\top} ) \bar\bU$, we decompose it into $\cH ( \bZ \bar\bX^{\top} ) \bar\bU + \cH ( \bZ \bZ^{\top} ) \bar\bU$. Note that
\begin{align*}
&[ \cH(\bZ \bar\bX^{\top}) \bar\bU ]_{mj} =  ( \bZ \bar\bX^{\top} )_{m} (\bI- \be_m \be_m^{\top})  \bar\bu_j
= \langle \bz_m,  \bar\bX^{\top} (\bI- \be_m \be_m^{\top})  \bar\bu_j \rangle,\\
& \| [ \cH(\bZ \bar\bX^{\top})  \bar\bU ]_{mj} \|_{\psi_2} \lesssim \| \bSigma^{1/2} \bar\bX^{\top} (\bI- \be_m \be_m^{\top})  \bar\bu_j \| .
\end{align*}
Lemma \ref{KPCA-lemma-Lp-gaussian} forces $\|\cH(\bZ \bar\bX^{\top}) \bar\bU \|_{2,p} = O_{\PP} ( \sqrt{p} \| \bM \|_{2,p};~p )$, where $M_{ij} =  \| \bSigma^{1/2} \bar\bX^{\top} (\bI- \be_m \be_m^{\top})  \bar\bu_j \|$. From
\begin{align*}
& \sum_{j=1}^{r}
\| \bSigma^{1/2} \bar\bX^{\top} (\bI- \be_m \be_m^{\top})  \bar\bu_j \|^2
= \bigg\langle 
(\bI- \be_m \be_m^{\top}) \bar\bX \bSigma \bar\bX^{\top} (\bI- \be_m \be_m^{\top}) ,~
\sum_{j=1}^{r} \bar\bu_j \bar\bu_j^{\top} \bigg\rangle \\
& \leq \Tr ( \bar\bX \bSigma \bar\bX^{\top} )
= \| \bar\bX \bSigma^{1/2} \|_{2,2}^2
\end{align*}
we get
\begin{align}
& \| \bM \|_{2,p} 
= \bigg[
\sum_{m=1}^{n}
\bigg(
\sum_{j=1}^{r}
\| \bSigma^{1/2} \bar\bX^{\top} (\bI- \be_m \be_m^{\top})  \bar\bu_j \|^2
\bigg)^{p/2}
\bigg]^{1/p}
= n^{1/p} \| \bar\bX \bSigma^{1/2} \|_{2,2}
\leq n^{1/2} \| \bar\bX \bSigma^{1/2} \|_{2,p}, \notag\\
& \|\cH(\bZ \bar\bX^{\top}) \bar\bU \|_{2,p} = O_{\PP} ( \sqrt{n p} \| \bar\bX \bSigma^{1/2} \|_{2,p};~p ) = O_{\PP} (  \bar\Delta \| \bar\bU\|_{2,p} ;~p ),
\label{KPCA-ineq-proof-lemma-Lp-prelim-1}
\end{align}
where we used Assumption \ref{KPCA-assumption-Lp} to get the last equality.

Note that $\| \bar\bU \|_2 = 1$ and $\| [ \cH(\bZ \bZ^{\top}) \bar\bU ]_{m} \|_2 
= \| \sum_{j\neq m}  \langle
\bz_m,  \bz_j \rangle \bar\bU_j \|_2$, $\forall m \in [n]$. Lemma \ref{KPCA-lemma-Z-product} asserts that
\begin{align}
& \| \cH(\bZ \bZ^{\top}) \bar\bU \|_{2,p}
 = 
\bigg(
\sum_{m=1}^{n} \bigg\| \sum_{j\neq m}  \langle \bz_m,\bz_j \rangle \bar\bU_j \bigg\|_2^p 
\bigg)^{1/p}
\notag\\
& = n^{1/p} \| \bar\bU \|_2^p  O_{\PP} \left(
\sqrt{r p} \max\{ \| \bSigma \|_{\mathrm{HS}},~\sqrt{n} \| \bSigma \|_{\mathrm{op}} \}
;~p \wedge n \right) \notag\\
& = O_{\PP} \Big(
n^{1/p} \sqrt{r p} \max\{  \| \bSigma \|_{\mathrm{HS}},~\sqrt{n} \| \bSigma \|_{\mathrm{op}}  \}
;~p \wedge n\Big)
= O_{\PP} ( \bar\Delta \| \bar\bU \|_{2,p};~p\wedge n ).
\label{KPCA-ineq-proof-lemma-Lp-prelim-2}
\end{align}
The last equality is due to Assumption \ref{KPCA-assumption-Lp}. Then we complete the proof using (\ref{KPCA-ineq-proof-lemma-Lp-prelim-0}),
(\ref{KPCA-ineq-proof-lemma-Lp-prelim-1}) and (\ref{KPCA-ineq-proof-lemma-Lp-prelim-2}).


\section{Proof of Theorem \ref{KPCA-corollary-main}}\label{KPCA-corollary-main-proof}

We will invoke Theorem \ref{KPCA-theorem-main} to prove Theorem \ref{KPCA-corollary-main} in the Hilbert setting (under Assumptions \ref{KPCA-assumption-incoherence}, \ref{KPCA-assumption-noise} and \ref{KPCA-assumption-concentration}). We claim that Assumption \ref{KPCA-assumption-Lp} holds, $p \lesssim n$ and
\begin{align}
\gamma  \| \bar\bG \|_{2,\infty} / \bar\Delta \ll  \sqrt{r/n} .
\label{KPCA-corollary-main-1}
\end{align}
In that case, Theorem \ref{KPCA-theorem-main} asserts that
\begin{align}
&\| \bU \sgn(\bH) - \bG\bar\bU \bar\bLambda^{-1} \|_{2,p} = O_{\PP} \left( \kappa \gamma \| \bar\bU \|_{2,p} + \gamma \bar\Delta^{-1} \| \bar\bG \|_{2,p} ;~p \right) \label{KPCA-eqn-corollary-main-2} 
,\\
&\| \bU \sgn(\bH) \|_{2,p} = O_{\PP} \left(  \| \bar\bU \|_{2,p} +  \gamma \bar\Delta^{-1} \| \bar\bG \|_{2,p} ;~p \right) , \label{KPCA-eqn-corollary-main-1} \\
&
\| \bU \bLambda^{1/2} \sgn(\bH) - \bG\bar\bU \bar\bLambda^{-1/2} \|_{2,p} = O_{\PP} ( \kappa^{3/2} \gamma \bar\Delta^{1/2} \| \bar\bU \|_{2,p} + \kappa^{1/2} \gamma \bar\Delta^{-1/2} \| \bar\bG \|_{2,p} ;~p\wedge n).  \label{KPCA-eqn-corollary-main-3} 
\end{align}

When $2 \leq p < \infty$, we have $n^{-1/2} \| \bv \|_2 \leq n^{-1/p} \| \bv \|_p \leq \| \bv \|_{\infty}$, $\forall \bv \in \RR^n$. This inequality and (\ref{KPCA-corollary-main-1}) force that
\begin{align*}
\gamma  \| \bar\bG \|_{2,p} \leq \gamma  n^{1/p} \| \bar\bG \|_{2,\infty}  \ll  n^{1/p} \bar\Delta \sqrt{r/n} =  n^{1/p} \bar\Delta n^{-1/2} \| \bar\bU \|_{2,2} \leq  \bar\Delta \| \bar\bU \|_{2,p}.
\end{align*}
Hence $\gamma \bar\Delta^{-1} \| \bar\bG \|_{2,p} = o( \| \bar\bU \|_{2,p})$. The first, third and fourth equation in Theorem \ref{KPCA-corollary-main} directly follow from (\ref{KPCA-eqn-corollary-main-2}), (\ref{KPCA-eqn-corollary-main-1}), (\ref{KPCA-eqn-corollary-main-3}) and $\kappa \gamma \ll 1/ \mu \lesssim 1$ in Assumption \ref{KPCA-assumption-incoherence}.

To control $\| \bU \sgn(\bH) - [ \bar\bU + \cH ( \bZ \bX^{\top} ) \bar\bU \bar\bLambda^{-1} ] \|_{2,p} $, we invoke Lemma \ref{KPCA-lemma-Lp_prelim} to get
\begin{align*}
& \| \bG \bar\bU -  \bar\bU \bar\bLambda - \cH ( \bZ \bX^{\top} ) \bar\bU \|_{2,p}
=  (\gamma + \sqrt{r/n} ) O_{\PP}( \bar\Delta \| \bar\bU \|_{2,p};~p) = o_{\PP}( \bar\Delta \| \bar\bU \|_{2,p};~p).
\end{align*}
Then 
\begin{align*}
& \| \bU \sgn(\bH) - [ \bar\bU + \cH ( \bZ \bX^{\top} ) \bar\bU \bar\bLambda^{-1} ] \|_{2,p} \\
& \leq \| \bU \sgn(\bH) - \bG\bar\bU \bar\bLambda^{-1}  \|_{2,p}
+ \|  \bG\bar\bU \bar\bLambda^{-1} - [ \bar\bU + \cH ( \bZ \bX^{\top} ) \bar\bU \bar\bLambda^{-1} ]  \|_{2,p} \\
& \leq  \| \bU \sgn(\bH) - \bG\bar\bU \bar\bLambda^{-1}  \|_{2,p}
+ \|  \bG\bar\bU  - [ \bar\bU \bar\bLambda + \cH ( \bZ \bX^{\top} ) \bar\bU ]  \|_{2,p} \| \bar\bLambda^{-1} \|_2  \\
& = o_{\PP} ( \| \bar\bU \|_{2,p};~p\wedge n) .
\end{align*}
Similarly, we obtain from (\ref{KPCA-eqn-corollary-main-3}) and Lemma \ref{KPCA-lemma-Lp_prelim}that
\begin{align*}
&\| \bU \bLambda^{1/2} \sgn(\bH) - [ \bar\bU \bar\bLambda^{1/2} + \cH (\bZ \bX^{\top}) \bar{\bU} \bar\bLambda^{-1/2} ] \|_{2,p} 
 = o_{\PP} ( \| \bar\bU \|_{2,p} \| \bar\bLambda^{1/2} \|_2 ;~p \wedge n).
\end{align*}

So far we have get all the desired results in Theorem \ref{KPCA-corollary-main}, provided that Assumption \ref{KPCA-assumption-Lp}, $p \lesssim n$ and (\ref{KPCA-corollary-main-1}) hold. 

\begin{claim}\label{claim-p-n}
$p \lesssim n$ under Assumption \ref{KPCA-assumption-incoherence}.
\end{claim}
\begin{proof}
It is easy seen that
\begin{align*}
p \overset{\mathrm{(i)}}{\lesssim} (\mu \gamma)^{-2} \overset{\mathrm{(ii)}}{\leq} \gamma^{-2} 
\overset{\mathrm{(iii)}}{\leq} (\kappa \mu \sqrt{r / n})^{-2}
= \frac{n}{r \kappa^2 \mu^2} \overset{\mathrm{(iv)}}{\leq} n,
\end{align*}
where we used $\mathrm{(i)}$ the condition on $p$; $\mathrm{(ii)}$ $\mu \geq 1$; $\mathrm{(iii)}$ Assumption \ref{KPCA-assumption-incoherence}; $\mathrm{(iv)}$ $r \geq 1$, $\kappa \geq 1$ and $\mu \geq 1$.
\end{proof}

To verify (\ref{KPCA-corollary-main-1}), we start from
\begin{align}
\| \bar\bG \|_{2,\infty} & = \| \bar\bX \bar\bX^{\top} \|_{2,\infty} 
\leq \| \bar\bX \|_{2,\infty} \| \bar\bX \|_{\mathrm{op}}
= \frac{\| \bar\bX \|_{2,\infty} }{\| \bar\bX \|_{\mathrm{op}}} \cdot \| \bar\bX\|_{\mathrm{op}}^2 \notag\\
& \overset{\mathrm{(i)}}{\leq} ( \mu \sqrt{r/n} )  (\kappa \bar\Delta)
=  \kappa \mu  \sqrt{r/n} \cdot \bar\Delta,
\label{ineq-conversion-1}
\end{align}
where $\mathrm{(i)}$ is due to $\mu \geq (\| \bar\bX \|_{2,\infty} / \| \bar\bX \|_{\mathrm{op}} ) \sqrt{n/r}$ and $\| \bar\bX \|_{\mathrm{op}}^2 = \| \bar\bG \|_{\mathrm{op}}= \kappa \bar\Delta$.
Assumption \ref{KPCA-assumption-incoherence} forces $\gamma \geq \kappa  \mu \sqrt{r/n} $ and
\begin{align}
\| \bar\bG \|_{2,\infty} / \bar\Delta \leq \gamma .
\label{ineq-conversion-2}
\end{align}
In addition, (\ref{ineq-conversion-1}) and the condition $\gamma \ll (\kappa \mu)^{-1}$ in Assumption \ref{KPCA-assumption-incoherence} imply (\ref{KPCA-corollary-main-1})

It remains to check Assumption \ref{KPCA-assumption-Lp}, which can be implied by the followings.

\begin{claim}\label{claim-assumptions}
Under Assumptions \ref{KPCA-assumption-incoherence} and \ref{KPCA-assumption-concentration}, we have
\begin{align}
& \sqrt{np} \| \bar{\bX} \bSigma^{1/2} \|_{2,p} \leq \sqrt{p} \mu \gamma \bar\Delta \| \bar\bU \|_{2,p} , \label{eqn-claim-a-1} \\
& n^{1/p}  \sqrt{r p} \max\{ \| \bSigma \|_{\mathrm{HS}},~\sqrt{n} \| \bSigma \|_{\mathrm{op}}  \}
\leq \sqrt{p} \gamma \bar\Delta  \| \bar\bU \|_{2,p}.
\label{eqn-claim-a-2}
\end{align}
Therefore, if $p \lesssim (\mu \gamma)^{-2}$ then $\sqrt{np} \| \bar{\bX} \bSigma^{1/2} \|_{2,p}  \lesssim \bar\Delta \| \bar\bU \|_{2,p}$ and $n^{1/p}  \sqrt{r p} \max\{ \| \bSigma \|_{\mathrm{HS}},~\sqrt{n} \| \bSigma \|_{\mathrm{op}}  \}
\lesssim \bar\Delta \| \bar\bU \|_{2,p} $.
\end{claim}

\begin{proof}
To prove (\ref{eqn-claim-a-1}), we first prove an inequality in $\| \cdot \|_{2,\infty}$ and then convert it to $\| \cdot \|_{2,p}$ using
\begin{align}
n^{-1/2} \| \bv \|_2 \leq n^{-1/p} \| \bv \|_p \leq \| \bv \|_{\infty}
,\qquad \forall \bv \in \RR^n
\label{KPCA-corollary-main-conversion}
\end{align}
By elementary calculation, 
\begin{align*}
\frac{ \bar\Delta \sqrt{r} }{ n \| \bar\bX \bSigma^{1/2} \|_{2,\infty}  }
& \overset{\mathrm{(i)}}{\geq}   \frac{ \bar\Delta \sqrt{r} }{ n \| \bar\bX \|_{2,\infty} \| \bSigma\|_{\mathrm{op}}^{1/2}  } 
= \bigg( \frac{ \bar\Delta }{ \kappa n \| \bSigma\|_{\mathrm{op}} }  \bigg)^{1/2} \frac{ \sqrt{\kappa r \bar\Delta / n} }{ \| \bar\bX \|_{2,\infty} } \\
& \overset{\mathrm{(ii)}}{=} \bigg( \frac{ \bar\Delta }{ \kappa n \| \bSigma\|_{\mathrm{op}} }  \bigg)^{1/2} \bigg( \frac{ \| \bar\bX \|_{\mathrm{op}} }{ \| \bar\bX \|_{2,\infty} } \sqrt{\frac{r}{n}} \bigg) 
\overset{\mathrm{(iii)}}{\geq}  \bigg( \frac{ \bar\Delta }{ \kappa n \| \bSigma\|_{\mathrm{op}} }  \bigg)^{1/2} \frac{1}{\mu}
\overset{\mathrm{(iv)}}{\geq} \frac{1}{\mu \gamma} .
\end{align*}
where we used $\mathrm{(i)}$ $\| \bar\bX \bSigma^{1/2} \|_{2,\infty} \leq \| \bar\bX \|_{2,\infty} \| \bSigma\|_{\mathrm{op}}^{1/2} $; $\mathrm{(ii)}$ $\kappa \bar\Delta = \| \bar\bG \|_{\mathrm{op}}  = \| \bar\bX \|_{\mathrm{op}}^2$; $\mathrm{(iii)}$ $\mu \geq \frac{ \| \bar\bX \|_{2,\infty}  }{ \| \bar\bX \|_{\mathrm{op}} } \sqrt{\frac{n}{r}}$; $\mathrm{(iv)}$ $\gamma \geq (\kappa n \| \bSigma \|_{\mathrm{op}} / \bar\Delta)^{1/2}$ in Assumption \ref{KPCA-assumption-concentration}.
We use (\ref{KPCA-corollary-main-conversion}) to get
\begin{align*}
& \sqrt{n p} \| \bar\bX \bSigma^{1/2} \|_{2,p}
 \leq \sqrt{n p} n^{1/p} \| \bar\bX \bSigma^{1/2} \|_{2,\infty}
\lesssim \sqrt{np}  n^{1/p} \frac{ \sqrt{r} \bar\Delta \mu \gamma}{n} = \bar\Delta   n^{1/p} \sqrt{r/n} (\sqrt{p} \mu \gamma) \\
& =  \bar\Delta  n^{1/p} n^{-1/2} \| \bar \bU \|_{2,2} (\sqrt{p} \mu \gamma) \leq \bar\Delta  n^{1/p} n^{-1/p} \| \bar \bU \|_{2,p} (\sqrt{p} \mu \gamma)
=  \bar\Delta  \| \bar \bU \|_{2,p} (\sqrt{p} \mu \gamma)
\lesssim \bar\Delta  \| \bar \bU \|_{2,p}.
\end{align*}
The last inequality is due to $p \lesssim (\mu \gamma)^{-2}$.

We finally prove (\ref{eqn-claim-a-2}).
By the conversion (\ref{KPCA-corollary-main-conversion}),
\begin{align*}
& n^{1/p}  \sqrt{r p} \max\{ \| \bSigma \|_{\mathrm{HS}},~\sqrt{n} \| \bSigma \|_{\mathrm{op}}  \}
 = n^{1/p } \| \bar\bU \|_{2,2} \sqrt{p}  \max\{ \| \bSigma \|_{\mathrm{HS}},~\sqrt{n} \| \bSigma \|_{\mathrm{op}}  \} \\
& \leq n^{1/2} \| \bar\bU \|_{2,p}  \sqrt{p}  \max\{ \| \bSigma \|_{\mathrm{HS}},~\sqrt{n} \| \bSigma \|_{\mathrm{op}}  \}
= \| \bar\bU \|_{2,p}  \sqrt{p}  \max\{ \sqrt{n}  \| \bSigma \|_{\mathrm{HS}},~ n \| \bSigma \|_{\mathrm{op}}  \}.
\end{align*}
By Assumption \ref{KPCA-assumption-concentration}, we have $\sqrt{n} \| \bSigma \|_{\mathrm{HS}} \leq \gamma \bar{\Delta}$, $n \| \bSigma \|_{\mathrm{op}} \leq \gamma^2  \bar\Delta / \kappa$ and
\[
\max\{ \sqrt{n} \| \bSigma \|_{\mathrm{HS}}
, n \| \bSigma \|_{\mathrm{op}} \} \leq \max \{  \gamma \bar{\Delta} , \gamma^2  \bar\Delta / \kappa \} \leq \gamma \bar{\Delta}.
\]
Then (\ref{eqn-claim-a-2}) directly follows.
\end{proof}

\section{Proofs of Section \ref{KPCA-sec-GMM-sGMM}}\label{KPCA-sec-GMM-sGMM-proof}

\subsection{Proof of \Cref{thm-sGMM-kmeans}}\label{sec-thm-sGMM-kmeans-proof}

We will invoke \Cref{KPCA-corollary-main} for Hilbert spaces to study the spectral embedding $\bU \bLambda^{1/2}$. To begin with, define $\bY \in \{ 0, 1\}^{n\times K}$ through $Y_{ij} = \bm{1}_{ \{ y_i = j \} }$. From $\bar\bx_i = \bmu_{\by_i}$ it is easy to see that $\bar\bG = \bY \bB \bY^{\top}$. We first analyze the $r$ leading eigenvalues of $\bar\bG$.

On the one hand, we have $\| \bY \|_{\mathrm{F}}^2 = n$ and
\begin{align}
\| \bar\bG \|_2 \leq \| \bY \|_2^2 \| \bB \|_2 \leq n \| \bB\|_2.
\label{eqn-thm-sGMM-kmeans-1}
\end{align}
On the other hand, under Assumption \ref{sGMM-assumption-regularity} we have $\mathrm{rank}(\bar\bG) \leq \mathrm{rank}(\bB)  = r $. Denote by $\sigma_j(\cdot)$ the $j$-th largest singular value of a matrix. There exists $\bQ \in \RR^{K\times r}$ such that $\bB = \bQ \bQ^{\top}$ and $\sqrt{ \lambda_j(\bB) } = \sigma_j (\bQ)$. For any distinct $i,j \in [K]$,
\[
\bar{s}^2 \leq \| \bmu_i - \bmu_j \|^2 = (\be_i - \be_j)^{\top} \bB (\be_i - \be_j) \leq 2 \| \bB \|_2.
\]
Therefore,
\begin{align}
 \lambda_r (\bar\bG) &= \lambda_r (\bY \bQ \bQ^{\top} \bY^{\top}) = \sigma_r^2 ( \bY \bQ  ) \geq [ \sigma_K (\bY ) \cdot \sigma_r (\bQ) ]^2 = \lambda_K (\bY^{\top} \bY) \cdot \lambda_r(\bB) \\
& \geq \lambda_K (\bY^{\top} \bY) \cdot \| \bB\|_2 / \kappa_0
 \geq \frac{ \bar{s}^2 \lambda_K (\bY^{\top} \bY) }{ 2 \kappa_0 } .
\label{eqn-thm-sGMM-kmeans-2}
\end{align}
and 
\[
\kappa = \| \bar\bG \|_2 / \lambda_r (\bar\bG) \leq \frac{n \| \bB\|_2}{\lambda_K (\bY^{\top} \bY) \cdot \| \bB\|_2 / \kappa_0} 
= \kappa_0 n / \lambda_K (\bY^{\top} \bY).
\]
Note that $\lambda_K (\bY^{\top} \bY) = \min_{k \in [K]} |\{ i \in [n]:~y_i = k \}| $. From Assumption \ref{sGMM-assumption-regularity} and the estimate \eqref{eqn-thm-sGMM-kmeans-1}, we get $\mathrm{rank}(\bar\bG) = r$,
\begin{align}
&\bar{\Delta} =  \lambda_r (\bar\bG) -  \lambda_{r+1} (\bar\bG) \geq  \frac{ n \bar{s}^2 }{  2 \kappa_0^2 }
\qquad\text{and}\qquad \kappa   \leq  \kappa_0^2.
\label{eqn-kmeans-eigengap}
\end{align}

Let $T_k = \{ i \in [n]:~y_i = k \}$ for any $k \in [K]$. Because $\bar\bx_i = \bmu_k$ holds for all $i \in T_k$, $\{ \bar\bU_i \}_{i \in T_k}$ are all the same. Then
\begin{align}
& K = \| \bar\bU \|_{\mathrm{F}}^2 \geq \sum_{i \in T_k} \|  \bar\bU_i \|_2^2, \notag \\
& \|  \bar\bU_i \|_2 \leq \sqrt{K / |T_k|} \leq \sqrt{K / (n/\kappa_0)} = \sqrt{K \kappa_0 /n} \leq \kappa_0 / \sqrt{n}, \qquad \forall i \in T_k , \notag \\
& \|  \bar\bU \|_{2,p} \leq n^{1/p} \|  \bar\bU \|_{2,\infty} \leq \kappa_0 n^{1/p - 1/2} , \qquad\forall p \geq 1.
\label{eqn-kmeans-u}
\end{align}

To apply \Cref{KPCA-corollary-main}, we first verify Assumption \ref{KPCA-assumption-incoherence}. By definition,
\begin{align*}
&\| \bar\bX \|_{2, \infty}^2 = \max_{k \in [K]} \| \bmu_k \|^2 \leq  \| \bB \|_2  . 
\end{align*}
By \eqref{eqn-thm-sGMM-kmeans-2},
\begin{align*}
&\| \bar\bX \|_{\mathrm{op}}^2 = \| \bar\bG \|_2 \geq \lambda_r(\bar\bG) \geq \lambda_K(\bY^{\top} \bY) \| \bB \|_2 / \kappa_0 \geq n \| \bB \|_2 / \kappa_0^2 .
\end{align*}
Therefore,
\begin{align}
& \frac{\| \bar\bX \|_{2, \infty}}{\| \bar\bX \|_{\mathrm{op}}} \sqrt{\frac{n}{r}}
\leq \frac{ \sqrt{\| \bB\|_2} }{ \sqrt{n \| \bB \|_2} / \kappa_0 }
\cdot \sqrt{\frac{n}{r}} \leq \kappa_0.
\label{eqn-thm-sGMM-kmeans-3}
\end{align}
In light of \eqref{eqn-kmeans-eigengap} and \eqref{eqn-thm-sGMM-kmeans-3}, Assumption \ref{KPCA-assumption-incoherence} holds when $\kappa_0^3 \sqrt{r / n} \leq \gamma \ll 1$. In addition, we have $1 \leq \mu \leq \kappa_0$.

Next, we verify Assumption \ref{KPCA-assumption-concentration}. From \eqref{eqn-kmeans-eigengap},
\begin{align*}
& \frac{ \kappa \| \bSigma \|_{\mathrm{op}} }{  \bar\Delta }
\leq \frac{ 2 \kappa_0^2 \| \bSigma \|_{\mathrm{op}} }{  \bar{s}^2 n } =
\frac{ 2 \kappa_0^3 \| \bSigma \|_{\mathrm{op}} }{  \bar{s}^2 n  } 
\qquad\text{and}\qquad
  \frac{ \| \bSigma \|_{\mathrm{HS}} }{ \bar\Delta } \leq 
\frac{ 2 \kappa_0  \| \bSigma \|_{\mathrm{HS}} }{  \bar{s}^2 n } 
\end{align*}
we see that Assumption \ref{KPCA-assumption-concentration} holds when
\[
\gamma \geq 2 \kappa_0^{3/2} \max \bigg\{ 
 \frac{ \sqrt{ \| \bSigma \|_{\mathrm{op}} } }{  \bar{s} }  ,~
\frac{  \| \bSigma \|_{\mathrm{HS}} }{ \sqrt{n} \bar{s}^2 } 
\bigg\} = \frac{2 \kappa_0^{3/2} }{ \sqrt{ \mathrm{SNR} } } .
\]

To sum up, all the assumptions in \Cref{KPCA-corollary-main} hold if
\begin{align}
&\gamma = 2 \kappa_0^{3} \max \bigg\{
\frac{1 }{ \sqrt{ \mathrm{SNR} } } , ~ \sqrt{\frac{r}{n}}
\bigg\} 
\qquad\text{and}\qquad
2 \leq p \leq  \mathrm{SNR} \wedge  n .
\label{eqn-kmeans-proof-0}
\end{align}
We adopt the regime in (\ref{eqn-kmeans-proof-0}). By \Cref{KPCA-corollary-main},
\begin{align}
& \| \bU \bLambda^{1/2} \sgn(\bH) - \bG \bar\bU \bar\bLambda^{-1/2} \|_{2,p} = o_{\PP} ( \| \bar\bU \|_{2,p} \| \bar\bLambda^{1/2} \|_2 ;~p) , 
\label{eqn-kmeans-proof-1}\\
& \| \bU \bLambda^{1/2} \sgn(\bH) - [ \bar\bU \bar\bLambda^{1/2} + \cH (\bZ \bX^{\top}) \bar{\bU} \bar\bLambda^{-1/2} ] \|_{2,p} = o_{\PP} ( \| \bar\bU \|_{2,p} \| \bar\bLambda^{1/2} \|_2 ;~p).
\label{eqn-kmeans-proof-2}
\end{align}

Suppose that the approximate $k$-means step is conducted on the rows of $\bU \bLambda^{1/2} \sgn(\bH)$ rather than $\bU \bLambda^{1/2}$. This does not affect the outcomes $\{ \hat{y}_i \}_{i=1}^n$ but facilitates analysis. We need the following lemma on the recovery error of approximate $k$-means clustering, whose proof is in \Cref{sec-lem-kmeans-key-proof}.

\begin{lemma}[Approximate $k$-means]\label{lem-kmeans-key}
	Let $\{ \bx_i \}_{i=1}^n \subseteq \RR^d$ and define
	\[
	L( \{ \bmu_j \}_{j=1}^K , \{ y_i \}_{i=1}^n ) = \sum_{i=1}^{n} \| \bx_i - \bmu_{y_i} \|_2^2,
	\qquad \forall \{ \bmu_j \}_{j=1}^K \subseteq \RR^d ,~~ \{ y_i \}_{i=1}^n \subseteq [K].
	\]
	Let $ \{ \bmu_j^{\star} \}_{j=1}^K, \{ \hat\bmu_j \}_{j=1}^K  \subseteq \RR^d$, $\{ y_i^{\star} \}_{i=1}^n, \{ \hat y_i  \}_{i=1}^n \subseteq [K]$, $s = \min_{j \neq k} \| \bmu_j^{\star} - \bmu_k^{\star} \|_2$ and $n_{\min} =  \min_{j \in [K]}  |\{ i :~ y_i^{\star} = j \}|$. Suppose that the followings hold for some $C > 0$ and $\delta \in (0, s / 2 )$:
	\begin{enumerate}
		\item (Near-optimality) $L ( \{ \hat\bmu_j \}_{j=1}^K, \{ \hat y_i \}_{i=1}^n ) \leq C \cdot L( \{ \bmu_j^{\star} \}_{j=1}^K , \{ y_i^{\star} \}_{i=1}^n )$ and $\hat y_i  \in \argmin_{j \in [K]} \| \bx_i - \hat\bmu_j \|_2$ for all $i \in [n]$;
		\item (Low-noise condition) $ L( \{ \bmu_j^{\star} \}_{j=1}^K , \{ y_i^{\star} \}_{i=1}^n ) \leq  \delta^2 n_{\min}  / [(1 + \sqrt{C})^2 K ]$.
	\end{enumerate}
	Then, there exists a permutation $\tau:~[K] \to [K]$ such that
	\begin{align*}
	& \sum_{j = 1}^K \| \bmu_j^{\star} - \hat\bmu_{\tau(j)} \|_2^2 \leq  \frac{(1 + \sqrt{C})^2 K}{n_{\min}} L ( \{ \bmu_j^{\star} \}_{j=1}^K, \{  y_i^{\star} \}_{i=1}^n ) , \\ 
	& \{i:~ \hat{y}_i \neq \tau(y_i^{\star}) \} \subseteq  \{ i:~ \| \bx_i - \bmu_{y_i^{\star}}^{\star} \|_2 \geq s / 2 - \delta \} , \\
	& | \{i:~ \hat{y}_i \neq \tau(y_i^{\star}) \} | \leq |  \{ i:~ \| \bx_i - \bmu_{y_i^{\star}}^{\star} \|_2 \geq s / 2 - \delta \} |
	\leq \frac{ L( \{ \bmu_j^{\star} \}_{j=1}^K , \{ y_i^{\star} \}_{i=1}^n ) }{ (s / 2 - \delta)^2 }.
	\end{align*}
\end{lemma}

To apply \Cref{lem-kmeans-key}, take $\bx_i = [ \bU \bLambda^{1/2} \sgn(\bH) ]_i$ and $y^{\star}_i = y_i$ for $i \in [n]$. Define $\bmu^{\star}_j = (\bar\bU \bar\bLambda^{1/2})_i$ for any $i \in \{ l :~y_l = j \}$. The near-optimality condition holds with $C = 1 + \varepsilon$. Moreover, we use Assumption \ref{sGMM-assumption-regularity} to get
\begin{align}
& s = \min_{j \neq k} \| \bmu_j^{\star} - \bmu_k^{\star} \|_2 = \bar{s}
\qquad\text{and}\qquad
n_{\min} =  \min_{j \in [K]}  |\{ i :~ y_i^{\star} = j \}| \geq n / \kappa_0.
\label{eqn-kmeans-s}
\end{align}

We now verify the low-noise condition with $\delta =  \bar{s} / 4 $. By definition,
\[
L( \{ \bmu_j^{\star} \}_{j=1}^K , \{ y_i^{\star} \}_{i=1}^n ) = \sum_{i=1}^{n} \| [\bU \bLambda \sgn(\bH)]_i - (\bar\bU \bar\bLambda^{1/2})_i \|_2^2
= \| \bU \bLambda \sgn(\bH) - \bar\bU \bar\bLambda^{1/2}  \|_{\mathrm{F}}^2.
\]
Since
\[
\bar\bU \bar\bLambda^{1/2}  = \bar\bG \bar\bU \bar\bLambda^{-1/2} = \bG \bar\bU \bar\bLambda^{-1/2} - ( \bG - \bar\bG ) \bar\bU \bar\bLambda^{-1/2} ,
\]
we have
\begin{align*}
& \sqrt{L( \{ \bmu_j^{\star} \}_{j=1}^K , \{ y_i^{\star} \}_{i=1}^n )}
= \| \bU \bLambda^{1/2} \sgn(\bH) - \bar\bU \bar\bLambda^{1/2} \|_{\mathrm{F}} \\
& \leq \| \bU \bLambda^{1/2} \sgn(\bH) - \bG  \bar\bU \bar\bLambda^{-1/2} \|_{\mathrm{F}}
+ \| \bG - \bar\bG \|_2 \| \bar\bU \bar\bLambda^{-1/2}  \|_{\mathrm{F}}\\
& \leq n^{1/2-1/p} \| \bU \bLambda^{1/2} \sgn(\bH) - \bG  \bar\bU \bar\bLambda^{-1/2} \|_{2,p} + \| \bG - \bar\bG \|_2 \| \bar\bU \|_{\mathrm{F}} \| \bar\bLambda^{-1/2}  \|_2 \\
& \overset{\mathrm{(i)}}{=} n^{1/2-1/p} \cdot o_{\PP} ( \| \bar\bU \|_{2,p} \| \bar\bLambda^{1/2} \|_2 ;~p) + O_{\PP} (\gamma\bar{\Delta} ;~n) \cdot \sqrt{K} \cdot \bar\Delta^{-1/2} \\
&  \overset{\mathrm{(ii)}}{=}  o_{\PP} (   \bar\Delta^{1/2}  ;~p) 
\overset{\mathrm{(iii)}}{=}  o_{\PP} (  \sqrt{n} \bar{s}  ;~p) . 
\end{align*}
Here we used $\mathrm{(i)}$ \Cref{eqn-kmeans-proof-1} and \Cref{KPCA-lemma-L2}; $\mathrm{(ii)}$ \Cref{eqn-kmeans-u} and Assumption \ref{sGMM-assumption-regularity}; $\mathrm{(iii)}$ \Cref{eqn-kmeans-eigengap}.
On the other hand, \Cref{eqn-kmeans-s} and $\delta = \bar{s} / 4 $ imply that
\begin{align*}
\delta^2 n_{\min}  / [(1 + \sqrt{C})^2 K ] \gtrsim n \bar{s}^2.
\end{align*}
As a result, there exists $\eta_n \to 0$ such that
\begin{align}
\PP \bigg(
L( \{ \bmu_j^{\star} \}_{j=1}^K , \{ y_i^{\star} \}_{i=1}^n ) \leq \eta_n  \delta^2 n_{\min}  / [(1 + \sqrt{C})^2 K ]
\bigg) \geq 1 - e^{-p}.
\label{eqn-kmeans-low-noise}
\end{align}
Consequently, the low-noise condition holds with probability at least $1 - e^{-p}$.

Now that the regularity conditions are verified, on the event in \Cref{eqn-kmeans-low-noise} we use \Cref{lem-kmeans-key} to get
\begin{align}
\min_{\tau \in S_K} | \{i:~ \hat{y}_i \neq \tau(y_i^{\star}) \} | & \leq |  \{ i:~ \| [ \bU \bLambda^{1/2} \sgn(\bH) ]_i - (\bar\bU \bar\bLambda^{1/2})_i \|_2 \geq s / 2 - \delta \} | \notag\\
& \leq |  \{ i:~ \| [ \bU \bLambda^{1/2} \sgn(\bH) ]_i - (\bar\bU \bar\bLambda^{1/2})_i \|_2 \geq \bar{s} / 4  \} | \notag\\
& \leq \frac{ \| \bU \bLambda^{1/2} \sgn(\bH)  - \bar\bU \bar\bLambda^{1/2}  \|_{2,p}^p }{ ( \bar{s} / 4 )^p } \notag \\
& = n \bigg(  \frac{ \| \bU \bLambda^{1/2} \sgn(\bH)  - \bar\bU \bar\bLambda^{1/2}  \|_{2,p}  }{ n^{1/p} \bar{s} / 4 }  \bigg)^p .
\label{eqn-kmeans-proof-markov}
\end{align}
So far we have not specified the choice of $p$. The results above hold for any $p$ satisfying \Cref{eqn-kmeans-proof-0}.

Next, we will find some constant $C_0 \in (0, 1]$ such that
\begin{align}
\PP \bigg(
\| \bU \bLambda^{1/2} \sgn(\bH)  - \bar\bU \bar\bLambda^{1/2}  \|_{2,p}/ n^{1/p} \leq \frac{ \bar{s} }{4e} 
\bigg) \leq 1 - e^{-p}
\label{eqn-kmeans-proof-p}
\end{align}
holds for $p = C_0 ( \mathrm{SNR} \wedge n )$. If that is true, then \Cref{eqn-kmeans-low-noise,eqn-kmeans-proof-markov,eqn-kmeans-proof-p} imply that
\begin{align*}
\EE \cM( \hat\by , \by ) \leq 3 e^{-C_0 ( \mathrm{SNR} \wedge n )} .
\end{align*}
When $\mathrm{SNR} \geq 2 C_0^{-1} \log n$, for large $n$ we have $\EE \cM( \hat\by , \by ) \leq 3 n^{-2}$. Then
\[
\PP [ \cM( \hat\by , \by )  > 0 ] = \PP [ \cM( \hat\by , \by )  \geq 1 /n  ] \leq \EE \cM( \hat\by , \by ) \leq 3 / n \to 0.
\]
When $\mathrm{SNR} \leq 2 C_0^{-1} \log n$, for large $n$ we have $\EE \cM( \hat\by , \by ) \leq 3 e^{-C_0 \mathrm{SNR} } $ and then
\[
\limsup_{n\to\infty} \mathrm{SNR}^{-1} \log \EE \cM ( \hat{\by} , \by ) < - C_0.
\]

It remains to find $C_0 \in (0, 1]$ and prove \Cref{eqn-kmeans-proof-p} for $p = C_0 ( \mathrm{SNR} \wedge n )$. By \Cref{eqn-kmeans-proof-2}, there exists $\xi_n \to 0 $ such that
\begin{align}
& \PP \Big( \| \bU \bLambda^{1/2} \sgn(\bH) - [ \bar\bU \bar\bLambda^{1/2} + \cH (\bZ \bX^{\top}) \bar{\bU} \bar\bLambda^{-1/2} ] \|_{2,p} 
\geq \xi_n \| \bar\bU \|_{2,p} \| \bar\bLambda^{1/2} \|_2 
\Big) 
\leq e^{-2p}
\label{eqn-kmeans-final-1}
\end{align} 
holds for large $n$. By Lemma \ref{KPCA-lemma-Lp_prelim}, Claim \ref{claim-p-n} and Claim \ref{claim-assumptions},
\begin{align*}
& \|\cH(\bZ \bX^{\top}) \bar\bU \|_{2,p} = O_{\PP} (
\sqrt{p} \mu \gamma \bar\Delta \| \bar\bU \|_{2,p}
;~p ) .
\end{align*}
There exists a constant $C_1$ such that
\begin{align*}
\PP \bigg(
\|\cH(\bZ \bX^{\top}) \bar\bU \|_{2,p} \geq C_1 \sqrt{p} \mu \gamma \bar\Delta \| \bar\bU \|_{2,p}
\bigg) \leq e^{-2p}
\end{align*}
when $n$ is large. We have
\begin{align}
\PP \bigg(
\|\cH(\bZ \bX^{\top}) \bar\bU \bar\bLambda^{-1/2} \|_{2,p} \geq C_1 \sqrt{p} \mu \gamma \bar\Delta^{1/2} \| \bar\bU \|_{2,p}
\bigg) \leq e^{-2p}.
\label{eqn-kmeans-final-2}
\end{align}

By \Cref{eqn-kmeans-final-1,eqn-kmeans-u,eqn-kmeans-final-2,eqn-kmeans-eigengap}, there exists a constant $C_2$ such that
\begin{align}
\PP \bigg(
\| \bU \bLambda^{1/2} \sgn(\bH) - \bar\bU \bar\bLambda^{1/2} \|_{2,p} 
\geq C_2 \max\{ \sqrt{p \gamma^2} , \xi_n \} \cdot  \bar{s} n^{1/p}
\bigg) \leq e^{-2p}
\end{align}
holds for large $n$. Recall that $\xi_n \to 0$. Hence, the desired inequality (\ref{eqn-kmeans-proof-p}) holds so long as $C_2 \sqrt{p \gamma^2} \leq \sqrt{2} / (4e)$, which is equivalent to
\[
p \leq \frac{2}{(4 C_2 e)^2 \gamma^2} = \frac{1}{8 C_2^2 e^2 \gamma^2} = 
\frac{1}{ 8 C_2^2 e^2 \kappa_0^{6} }
\min \{
\mathrm{SNR}  , ~ n/K
\}.
\]
Hence, it suffices to take $p = C_0 (\mathrm{SNR} \wedge n)$ with
\[
C_0 = \frac{1}{ 8 C_2^2 e^2 \kappa_0^{6} K }.
\]

\subsection{Proof of \Cref{lem-kmeans-key}}\label{sec-lem-kmeans-key-proof}
Define $\bX = (\bx_1,\cdots,\bx_n)^{\top} \in \RR^{n \times d}$, $\bM^{\star} = (\bmu_1^{\star} , \cdots, \bmu_K^{\star}) \in \RR^{d \times K}$, $\hat\bM = (\hat\bmu_1 , \cdots, \hat\bmu_K )  \in \RR^{d \times K}$, $\bY^{\star} = ( \be_{y_1^{\star}} , \cdots, \be_{y_n^{\star}} )^{\top}  \in \RR^{n \times K}$ and $\hat\bY = ( \be_{\hat y_1 } , \cdots, \be_{\hat y_n } )^{\top}  \in \RR^{n \times K}$. Then 
\begin{align*}
\| \hat\bY \hat\bM^{\top} - \bY^{\star} \bM^{\star \top} \|_{\mathrm{F}}
&\leq \| \hat\bY \hat\bM^{\top} - \bX \|_{\mathrm{F}}
+ \| \bX -  \bY^{\star} \bM^{\star \top} \|_{\mathrm{F}} \\
& = \sqrt{ L ( \{ \hat\bmu_j \}_{j=1}^K, \{ \hat y_i \}_{i=1}^n ) } + \sqrt{ L ( \{ \bmu_j^{\star} \}_{j=1}^K, \{  y_i^{\star} \}_{i=1}^n )  } \\
& \leq (1 + \sqrt{C}) \sqrt{ L ( \{ \bmu_j^{\star} \}_{j=1}^K, \{  y_i^{\star} \}_{i=1}^n )  }
\end{align*}
and
\begin{align}
\| \hat\bY \hat\bM^{\top} - \bY^{\star} \bM^{\star \top} \|_{\mathrm{F}}^2
\leq (1 + \sqrt{C})^2 L ( \{ \bmu_j^{\star} \}_{j=1}^K, \{  y_i^{\star} \}_{i=1}^n )
\leq \frac{ n_{\min} \delta^2 }{  K } < \frac{ n_{\min} s^2 }{ 4  K }.
\label{eqn-lem-kmeans-dk-1}
\end{align}

\begin{claim}
	There exists a permutation $\tau:~[K] \to [K]$ such that
	\begin{align*}
	\sum_{j = 1}^K \| \bmu_j^{\star} - \hat\bmu_{\tau(j)} \|_2^2 \leq 
	\frac{ K }{n_{\min}} \| \hat\bY \hat\bM^{\top} - \bY^{\star} \bM^{\star \top} \|_{\mathrm{F}}^2  .
	\end{align*}
\end{claim}
For any $j,k \in [K]$, define $S_{jk} = |\{ i \in [n] :~ y_i^{\star} = j,~ \hat y_i = k \}|$ and $\tau(j) = \argmax_{k \in [K]} S_{jk}$ (break any tie by selecting the smaller index). We have $S_{j \tau(j)} \geq n_{\min} / K$ for all $j$ and
\begin{align}
\| \hat{\bY} \hat{\bM}^{\top} - \bY^{\star} \bM^{\star\top} \|_{\mathrm{F}}^2
= \sum_{j, k \in [K]} S_{jk} \| \bmu_j^{\star} - \hat\bmu_k \|_2^2.
\label{eqn-lem-kmeans-dk-0}
\end{align}

We first prove by contradiction that $\tau:~[K] \to [K]$ must be a permutation (bijection). Suppose there exist distinct $j$ and $k$ such that $\tau(j) = \tau(k) = \ell$. By the triangle's inequality,
\begin{align*}
& \| \bmu_j^{\star} - \hat\bmu_{\ell} \|_2 + \| \hat\bmu_{\ell}- \bmu_k^{\star} \|_2 \geq \| \bmu_j^{\star} - \bmu_k^{\star} \|_2 \geq s ,\\
&  \| \bmu_j^{\star} - \hat\bmu_{\ell} \|_2^2 + \| \hat\bmu_{\ell}- \bmu_k^{\star} \|_2^2 \geq s^2 / 2.
\end{align*}
By \Cref{eqn-lem-kmeans-dk-0} and the facts that $S_{j\ell} \geq n_{\min} / K$ and $S_{k\ell} \geq n_{\min} / K$,
\begin{align*}
\| \hat{\bY} \hat{\bM}^{\top} - \bY^{\star} \bM^{\star\top} \|_{\mathrm{F}}^2 \geq
S_{j \ell} \| \bmu_j^{\star} - \hat\bmu_{\ell} \|_2^2
+ S_{k \ell} \| \bmu_k^{\star} - \hat\bmu_{\ell} \|_2^2
\geq \frac{n_{\min} s^2}{2K} ,
\end{align*}
which contradicts \Cref{eqn-lem-kmeans-dk-1}. Now that $\tau$ is a permutation, we derive from \Cref{eqn-lem-kmeans-dk-0} that
\begin{align*}
& \| \hat{\bY} \hat{\bM}^{\top} - \bY^{\star} \bM^{\star\top} \|_{\mathrm{F}}^2
\geq \sum_{j = 1}^K S_{j\tau(j)} \| \bmu_j^{\star} - \hat\bmu_{\tau(j)} \|_2^2
\geq \frac{n_{\min}}{K} \sum_{j = 1}^K \| \bmu_j^{\star} - \hat\bmu_{\tau(j)} \|_2^2 .
\end{align*}

Now that the claim has been proved, we obtain that
\begin{align*}
\sum_{j = 1}^K \| \bmu_j^{\star} - \hat\bmu_{\tau(j)} \|_2^2 \leq 
\frac{ K }{n_{\min}} \| \hat\bY \hat\bM^{\top} - \bY^{\star} \bM^{\star \top} \|_{\mathrm{F}}^2 
\leq \frac{(1 + \sqrt{C})^2 K L ( \{ \bmu_j^{\star} \}_{j=1}^K, \{  y_i^{\star} \}_{i=1}^n )}{n_{\min}}
\leq \delta^2 .
\end{align*}
Hence $\max_{j \in [K]} \|  \bmu_j^{\star} - \hat\bmu_{\tau(j)} \|_2 \leq \delta$ and $\min_{j \neq k} \| \hat\bmu_j - \hat\bmu_k \|_2 \geq s - 2 \delta$. 
For any $j \neq y_i^{\star}$, we have
\begin{align*}
& \| \bx_i - \hat\bmu_{\tau(j)} \|_2 - \| \bx_i - \hat\bmu_{\tau(y_i^{\star})} \|_2 \\
& \geq ( \| \bx_i - \bmu_{j}^{\star} \|_2 - \| \bmu_{j}^{\star} - \hat\bmu_{\tau(j)} \|_2 ) - ( \| \bx_i - \bmu_{y_i^{\star}}^{\star} \|_2 + \| \bmu_{y_i^{\star}}^{\star} - \hat\bmu_{\tau(y_i^{\star})} \|_2 ) \\
& \geq \| \bx_i - \bmu_{j}^{\star} \|_2 - \| \bx_i - \bmu_{y_i^{\star}}^{\star} \|_2 - 2 \delta \\
& \geq  (\| \bmu_{j}^{\star} - \bmu_{y_i^{\star}}^{\star} \|_2 -
\| \bx_i - \bmu_{y_i^{\star}}^{\star} \|_2 ) - \| \bx_i - \bmu_{y_i^{\star}}^{\star} \|_2  - 2 \delta \\
& = \| \bmu_{j}^{\star} - \bmu_{y_i^{\star}}^{\star} \|_2 - 2 \| \bx_i - \bmu_{y_i^{\star}}^{\star} \|_2 - 2 \delta  \geq s - 2 ( \| \bx_i - \bmu_{y_i^{\star}}^{\star} \|_2 + \delta ).
\end{align*}
If $\| \bx_i - \bmu_{y_i^{\star}}^{\star} \|_2 < s / 2 - \delta $, then $\| \bx_i - \hat\bmu_{\tau(j)} \|_2 > \| \bx_i - \hat\bmu_{\tau(y_i^{\star})} \|_2$ for all $j \neq y_i^{\star}$. The assumption $\hat y_i  \in \argmin_{j \in [K]} \| \bx_i - \hat\bmu_j \|_2$, $\forall i \in [n]$ forces
\begin{align*}
\{i:~ \hat{y}_i \neq \tau(y_i^{\star}) \} & \subseteq \{ i:~ \| \bx_i - \hat\bmu_{\tau(j)} \|_2 \leq \| \bx_i - \hat\bmu_{\tau(y_i^{\star})} \|_2 \text{ for some } j \neq  y_i^{\star} \} \\
& \subseteq \{ i:~ \| \bx_i - \bmu_{y_i^{\star}}^{\star} \|_2 \geq s / 2 - \delta \}.
\end{align*}
The desired bound on $ | \{i:~ \hat{y}_i \neq \tau(y_i^{\star}) \} |$ then becomes obvious.

\section{Proofs of Section \ref{KPCA-sec-GMM-GMM}}\label{KPCA-sec-GMM-proof}

\subsection{Useful lemmas}

We first prove a lemma bridging $\ell_p$ approximation and misclassification rates.

\begin{lemma}\label{KPCA-lemma-GMM-error-rate}
	Suppose that $\bv = \bv_n, \bw = \bw_n$ and $\bar\bv =\bar{\bv}_n$ are random vectors in $\RR^n$, $\min_{i \in [n]} |\bar v_i| = \delta_n > 0$, and $p = p_n \to \infty$. If $\min_{s = \pm 1} \| s \bv - \bar\bv - \bw \|_p = o_{\PP} ( n^{1/p} \delta_n;~ p)$, then
	\begin{align*}
	& \limsup_{n\to\infty} p^{-1} \log \bigg(
	\frac{1}{n}
	\EE \min_{s = \pm 1} \left|  \{ i \in [n]:~ s \sgn(v_i) \neq \sgn(\bar v_i) \} \right| 
	\bigg)
	\\ &
	\leq
	\limsup_{\varepsilon \to 0}
	\limsup_{n\to\infty} p^{-1} \log \bigg( \frac{1}{n} \sum_{i=1}^{n}  \PP  \left(
	-w_i \sgn(\bar v_i) \geq (1-\varepsilon) |\bar v_i|
	\right)
	\bigg).
	\end{align*}
\end{lemma}

\begin{proof}[\bf Proof of Lemma \ref{KPCA-lemma-GMM-error-rate}]
	
	Let $S_n = \{ i \in [n]:~ \sgn(v_i) \neq \sgn(\bar v_i) \}$ and $\br = \bv - \bar\bv - \bw$. For notational simplicity, we will prove the upper bound for $
	\limsup_{n\to\infty} p^{-1} \log ( \EE |S_n| / n )$ under a stronger assumption $\| \br \|_p = o_{\PP} ( n^{1/p} \delta_n;~ p)$. Otherwise we just redefine $\bv$ as $ (\argmin_{s = \pm 1} \| s \bv - \bar\bv - \bw \|_p ) \bv$ and go through the same proof.
	
	As a matter of fact,
	\begin{align*}
	S_n 
	& \subseteq \{ i \in [n]:~ - ( v_i -\bar v_i) \sgn (\bar v_i ) \geq | \bar v_i |  \}
	= \{ i \in [n]:~ - ( w_i + r_i ) \sgn (\bar v_i ) \geq | \bar v_i |  \}.
	\end{align*}
	For any $\varepsilon \in (0,1)$,
	\begin{align*}
	&\{ i \in [n]:~ -r_i \sgn(\bar v_i) < \varepsilon |\bar v_i| \text{ and } -w_i \sgn(\bar v_i) < (1-\varepsilon) |\bar v_i| \} \\
	& \subseteq  \{ i \in [n]:~ - ( w_i + r_i ) \sgn (\bar v_i ) < | \bar v_i |  \}.
	\end{align*}
	Hence
	\begin{align*}
	&S_n  \subseteq \{ i \in [n]:~ -r_i \sgn(\bar v_i) \geq \varepsilon |\bar v_i| \text{ or } -w_i \sgn(\bar v_i) \geq (1-\varepsilon) |\bar v_i| \} \notag \\
	& \subseteq \{ i \in [n]:~ |r_i| \geq \varepsilon |\bar v_i| \} \cup \{ i \in [n]:~ -w_i \sgn(\bar v_i) \geq (1-\varepsilon) |\bar v_i| \}.
	\end{align*}
	Let $ q_n (\varepsilon) = \frac{1}{n}\sum_{i=1}^{n} \PP  \left(
	-w_i \sgn(\bar v_i) \geq (1-\varepsilon) |\bar v_i|
	\right)$. We have $\EE | S_n | \leq \EE \left| \{ i \in [n]:~ |r_i| \geq \varepsilon |\bar v_i| \} \right| + n q_n (\varepsilon)$.
	
	To study $\{ i \in [n]:~ |r_i| \geq \varepsilon |\bar v_i| \}$, we define $\mathcal{E}_n = \{ \| \br \|_p < \varepsilon^2 n^{1/p} \delta_n \}$. Since $\| \br \|_p = o_{\PP} ( n^{1/p} \delta_n ;~ p )$, there exist $C_1, N\in \ZZ_+$ such that $\PP ( \mathcal{E}_n^c ) \leq C_1 e^{- p / \varepsilon }$, $\forall n \geq N$. When $\mathcal{E}_n$ happens,
	\begin{align*}
	\left| \{ i \in [n]:~ |r_i| \geq \varepsilon |\bar v_i| \} \right|
	\leq \left| \{ i \in [n]:~ |r_i| \geq \varepsilon \delta_n \} \right|
	\leq \frac{\| \br \|_p^p }{ (\varepsilon \delta_n)^p}
	\leq \frac{ (\varepsilon^2 n^{1/p} \delta_n )^p }{ (\varepsilon \delta_n)^p}
	= n \varepsilon^p.
	\end{align*}
	Then by $\log t = \log (1 + t - 1) \leq t - 1 < t$ for $t \geq 1$, we have $\log (1/\varepsilon) \leq 1 / \varepsilon$,
	\begin{align*}
	&n^{-1} \EE\left| \{ i \in [n]:~ |r_i| \geq \varepsilon |\bar v_i| \} \right|
	\leq \varepsilon^p \PP (\mathcal{E}_n) + 1 \cdot \PP (\mathcal{E}_n^c) \\
	& \leq e^{ - p  \log (1/ \varepsilon) } + C_1 e^{-p/\varepsilon}
	\leq (C_1 \vee 1) e ^{-p \log (1/\varepsilon)},
	\end{align*}
	and $n^{-1} \EE |S_n| \leq (C_1 \vee 1) e^{ - p  \log (1/ \varepsilon ) }  +q_n (\varepsilon)$. As a result,
	\begin{align*}
	\log (\EE |S_n| / n )
	& \leq \log ( (C_1 \vee 1) e^{ - p  \log (1/ \varepsilon ) }  +q_n (\varepsilon) )
	\leq  \log [ 2 \max \{ (C_1 \vee 1) e^{ - p  \log (1/ \varepsilon ) } , q_n (\varepsilon) \} ] \\
	& \leq  \log 2 +\max\{ \log (C_1 \vee 1) - p  \log (1/ \varepsilon ),~ \log q_n (\varepsilon)
	\}.
	\end{align*}
	The assumption $p = p_n \to \infty$ leads to
	\begin{align*}
	\limsup_{n\to\infty} p^{-1} \log (\EE |S_n| / n ) \leq \max\{ - \log (1/\varepsilon),~ \limsup_{n\to\infty} p^{-1} \log q_n (\varepsilon) \}, \qquad\forall \varepsilon \in (0,1).
	\end{align*}
	By letting $\varepsilon \to 0$ we finish the proof.
\end{proof}

The following lemma will be used in the analysis of misclassification rates.

\begin{lemma}\label{KPCA-lemma-GMM-norms}
Consider the Gaussian mixture model in Definition \ref{KPCA-defn-GMM} with $d \geq 2$. Let $R = \| \bmu \|_2$ and $ p = \mathrm{SNR} = R^4 / (R^2 + d / n)$. If $n \to \infty$ and $\mathrm{SNR} \to \infty$, then for any fixed $i$ we have
$\| \hat\bmu^{(-i)} - \bmu \|_2 = O_{\PP} ( \sqrt{(d \vee p)/n} ;~p )$, $\left| 
\| \hat\bmu^{(-i)} \|_2 - \sqrt{R^2 + d / (n-1)} \right|
= O_{\PP} ( \sqrt{ p / n } ;~p )$, $\| \bx_i \|_2 = O_{\PP} ( R \vee \sqrt{d};~p )$, $\langle \hat\bmu^{(-i)} - \bmu , \bx_i \rangle = \sqrt{p / n} O_{\PP} (  R \vee \sqrt{d};~p)$ and $\langle \hat\bmu^{(-i)} , \bx_i \rangle = O_{\PP} (R^2;~ p)$.
\end{lemma}

\begin{proof}[\bf Proof of Lemma \ref{KPCA-lemma-GMM-norms}]
Let $\bw_i = \sum_{j\neq i} \bz_j  y_j$ and note that $(n-1) \hat{\bmu}^{(-i)} = \sum_{j \neq i} \bx_j  y_j = \sum_{j \neq i} (\bmu  y_j + \bz_j )  y_j = (n-1) \bmu + \bw_i$. From $\bw_i \sim N ( \mathbf{0}, (n-1) \bI_d )$ we get $\| \bw_i \|_2^2 / (n-1) \sim \chi_d^2$, and Lemma \ref{KPCA-lemma-chi-square} leads to $\| \bw_i \|_2^2 / (n-1) - d = O_{\PP} ( p \vee \sqrt{pd} ;~p)$. Then
\begin{align*}
& \| \hat\bmu^{(-i)} - \bmu \|_2^2 = (n-1)^{-2} \| \bw_i \|_2^2 = \frac{ d + O_{\PP} ( p \vee \sqrt{pd} ;~p) }{n - 1} = O_{\PP} ( (d \vee p) / n ;~p ),
\end{align*}
and $\| \hat\bmu^{(-i)} - \bmu \|_2 = O_{\PP} ( \sqrt{(d \vee p)/n} ;~p )$. To study $\| \hat\bmu^{(-i)}  \|_2$, we start from the decomposition
\begin{align*}
\| \hat\bmu^{(-i)} \|_2^2 = \| \bmu \|_2^2 + 2 (n-1)^{-1} \langle \bmu, \bw_i \rangle  +(n-1)^{-2}  \| \bw_i \|_2^2.
\end{align*}
Since $\langle \bmu, \bw_i \rangle \sim N ( 0, (n-1) R^2 )$, Lemma \ref{KPCA-lemma-Lp-gaussian} yields $ \langle \bmu, \bw_i \rangle  = O_{\PP} ( R \sqrt{np};~p )$.
We use these and $\sqrt{p} \leq R$ to derive
\begin{align*}
\| \hat\bmu^{(-i)} \|_2^2 
& = R^2 + \frac{2 \cdot O_{\PP} ( R \sqrt{np};~p )  }{n-1} + \frac{ d + O_{\PP} ( p \vee \sqrt{pd} ;~p) }{n - 1} \\
& = R^2 + \frac{d}{n-1} + \frac{ \max\{
	R\sqrt{np},~ p ,~ \sqrt{pd}
	\} }{n} O_{\PP} ( 1 ;~p )\\
& = R^2 + \frac{d}{n-1} + \frac{ \max\{
	R\sqrt{np} ,~ \sqrt{pd}
	\} }{n} O_{\PP} ( 1 ;~p ) \\
& = R^2 + \frac{d}{n-1} + \sqrt{ \frac{p}{n} } O_{\PP} ( R \vee \sqrt{d / n} ;~p ).
\end{align*}
Based on this and $\sqrt{R^2 + d / (n-1)} \geq \sqrt{R^2 + d / n} \asymp R \vee \sqrt{d / n}$, 
\begin{align*}
& \left| 
\| \hat\bmu^{(-i)} \|_2 - \sqrt{R^2 + d / (n-1)} \right| = \frac{
	\left|  \| \hat\bmu^{(-i)} \|_2^2 - [ R^2 + d / (n-1) ] 
	\right|
}{
	\| \hat\bmu^{(-i)} \|_2 + \sqrt{R^2 + d / (n-1)}
} \\
& \leq \frac{  \sqrt{ p / n } O_{\PP} ( R \vee \sqrt{d / n} ;~p ) }{ \sqrt{R^2 + d / (n-1)} }
= O_{\PP} ( \sqrt{ p / n } ;~p ).
\end{align*}

From $\| \bz_i \|_2^2 \sim \chi_d^2$ and Lemma \ref{KPCA-lemma-chi-square} we get $\| \bz_i \|_2^2 = d + O_{\PP} ( \sqrt{pd} \vee p ;~p ) = O_{\PP} ( p \vee d;~p )$. Hence
$\| \bx_i \|_2 \leq \| \bmu \|_2 + \| \bz_i \|_2 = R + O_{\PP} ( \sqrt{p \vee d} ;~p ) = O_{\PP} ( R \vee \sqrt{d};~p )$ as $R \geq \sqrt{p}$.

Now we study $\langle \hat\bmu^{(-i)} - \bmu , \bx_i \rangle = \langle \hat\bmu^{(-i)} - \bmu , \bmu \rangle  y_i + \langle \hat\bmu^{(-i)} - \bmu , \bz_i \rangle$.
On the one hand, $ \langle \hat\bmu^{(-i)} - \bmu , \bmu \rangle  =(n-1)^{-1} \langle \bw_i , \bmu \rangle \sim N(0, R^2 / (n-1))$ and Lemma \ref{KPCA-lemma-Lp-gaussian} imply that $\langle \hat\bmu^{(-i)} - \bmu , \bmu \rangle = O_{\PP} ( R \sqrt{ p / n } ;~p )$.
On the other hand, $ \langle \hat\bmu^{(-i)} - \bmu , \bz_i \rangle / \| \hat\bmu^{(-i)} - \bmu \|_2 \sim N(0,1)$ leads to $ \langle \hat\bmu^{(-i)} - \bmu , \bz_i \rangle / \| \hat\bmu^{(-i)} - \bmu \|_2 = O_{\PP} (\sqrt{p};~p)$. Since $\| \hat\bmu^{(-i)} - \bmu \|_2 = O_{\PP} ( \sqrt{(d \vee p)/n} ;~p )$, we have $ \langle \hat\bmu^{(-i)} - \bmu , \bz_i \rangle  = \sqrt{p / n} O_{\PP} ( \sqrt{p \vee d};~p)$.
As a result,
\begin{align*}
\langle \hat\bmu^{(-i)} - \bmu , \bx_i \rangle = \sqrt{p / n} O_{\PP} (  R \vee \sqrt{d};~p).
\end{align*}

Note that $\left| \langle \bmu , \bx_i \rangle \right| \leq  \left| \|\bmu\|_2^2  y_i  + \langle \bmu , \bz_i \rangle \right| \leq R^2 + \left| \langle \bmu , \bz_i \rangle \right|$. From $\langle \bmu , \bz_i \rangle \sim N( 0, R^2 )$ we obtain that $  \langle \bmu , \bz_i \rangle = O_{\PP} ( R \sqrt{p};~p )$. The fact $\sqrt{p} \leq R$ leads to $ \langle \bmu , \bx_i \rangle  = O_{\PP} ( R^2;~p )$ and
\begin{align*}
\langle \hat\bmu^{(-i)} , \bx_i \rangle  = 
\langle \bmu , \bx_i \rangle +  \langle \hat\bmu^{(-i)} - \bmu , \bx_i \rangle =
O_{\PP} ( R^2 + \sqrt{p/n} (R \vee \sqrt{d}) ;~p ) = O_{\PP} (R^2;~ p),
\end{align*}
where we also applied $\sqrt{pd / n} = R^2 \sqrt{d / n} / \sqrt{ R^2 + d / n } \leq R^2$.
\end{proof}

\subsection{Proof of Theorem \ref{KPCA-thm-GMM}}\label{KPCA-thm-GMM-proof}

We supress the subscripts of $\lambda_1$, $\bar\lambda_1$, $\bu_1$ and $\bar\bu_1$. First, suppose that $\mathrm{SNR} > C \log n$ for some constant $C>0$. We have
\begin{align*}
&\PP [\cM (\hat\by , \by ) > 0] \leq \PP \bigg(
\min_{c = \pm 1} \| c \bu - \bar\bu \|_2 \geq \frac{1}{\sqrt{n}}
\bigg)  \notag\\
& \leq \PP \bigg(
\min_{c = \pm 1} \| c \bu - \bar\bu - \cH ( \bZ \bX^{\top} ) \bar\bu / \bar\lambda \|_{\infty} \geq \frac{1}{2 \sqrt{n}}
\bigg) + \PP \bigg( \| \cH ( \bZ \bX^{\top} ) \bar\bu / \bar\lambda \|_{\infty} \geq \frac{1}{2 \sqrt{n}}
\bigg) .
\end{align*}
By \Cref{KPCA-corollary-main-inf},
\[
\min_{c = \pm 1} \| c \bu - \bar\bu - \cH ( \bZ \bX^{\top} ) \bar\bu / \bar\lambda \|_{\infty} = o_{\PP} (\| \bar\bu \|_{\infty} ;~ \log n ).
\]
Therefore, for sufficiently large $n$ we have
\begin{align*}
\PP \bigg(
\min_{c = \pm 1} \| c \bu - \bar\bu - \cH ( \bZ \bX^{\top} ) \bar\bu / \bar\lambda \|_{\infty} \geq \frac{1}{2 \sqrt{n}}
\bigg) \leq \frac{1}{n}.
\end{align*}
We are going to show that when $C$ is large,
\begin{align}
\PP \bigg( \| \cH ( \bZ \bX^{\top} ) \bar\bu / \bar\lambda \|_{\infty} \geq \frac{1}{2 \sqrt{n}}
\bigg) \to 0 .
\label{eqn-gmm-new-0}
\end{align}
If that is true, then $\PP [\cM (\hat\by , \by ) = 0] \to 1$ provided that $\mathrm{SNR} > C \log n$. Note that
\begin{align}
\PP \bigg( \| \cH ( \bZ \bX^{\top} ) \bar\bu / \bar\lambda \|_{\infty} \geq \frac{1}{2 \sqrt{n}}
\bigg)
& = \PP \bigg( \frac{1}{n \| \bmu \|_2^2} \max_{i \in [n]} 
\bigg| \bigg\langle
\bz_i , \sum_{j \neq i} \bx_j \bar\bu_j
\bigg\rangle \bigg|
\geq \frac{1}{2 \sqrt{n}}
\bigg) \notag\\
& \leq \sum_{i=1}^{n}
 \PP \bigg( 
\bigg| \bigg\langle
\bz_i , \sum_{j \neq i} \bx_j \bar\bu_j
\bigg\rangle \bigg|
\geq \frac{\sqrt{n} \| \bmu \|_2^2}{2}
\bigg).
\label{eqn-gmm-new-1}
\end{align}

Note that $\bar\bu_j = y_j / \sqrt{n}$ and $\bx_j = y_j \bmu + \bz_j$, we have
\begin{align*}
\sum_{j \neq i} \bx_j \bar\bu_j = \sum_{j \neq i} ( \bmu + y_j \bz_j) = \sqrt{\frac{n-1}{n}} (
\sqrt{n-1}  \bmu + \bw_i ),
\end{align*}
where $\bw_i = \frac{1}{\sqrt{n-1}} \sum_{j \neq i} y_j \bz_j \sim N(\bm{0} , \bI)$. Hence
\begin{align}
& \PP  \bigg(
\bigg| \bigg\langle \bz_i , \sum_{j \neq i} \bx_j \bar\bu_j \bigg\rangle 
\bigg | \geq \frac{ \sqrt{n} \| \bmu \|_2^2 }{ 2 }
\bigg) 
\notag\\ &
 =\PP  \bigg(
\bigg| \bigg\langle \bz_i , \frac{ \sqrt{n-1} \bmu + \bw_i }{ \|   \sqrt{n-1} \bmu + \bw_i  \|_2 } \bigg\rangle 
\bigg | \geq
\frac{ \sqrt{n} \| \bmu \|_2^2  }{2 \|  \sqrt{n-1} \bmu + \bw_i  \|_2 }
\bigg).
\label{KPCA-eqn-sGMM-1}
\end{align}

By the triangle's inequality,
\begin{align*}
\|   \sqrt{n-1} \bmu + \bw_i   \|_2 \leq \sqrt{n}  \| \bmu \|_2 + \| \bw_i \|_2.
\end{align*}
Since Lemma \ref{KPCA-lem-concentration-gram} yields $\| \bw_i \|_2^2 = O_{\PP} ( d \vee n ;~ n )$. There exist constants $c_1, c_2>0$ such that
\begin{align*}
\PP ( \|  \bw_i \|_2 \geq c_1 \sqrt{d \vee n} ) < c_2 e^{-n}.
\end{align*}
The assumption $\mathrm{SNR} \geq C \log n$ yields $\| \bmu \|_2 \gg 1$ and thus
\begin{align*}
\PP \Big( \|  \sqrt{n-1} \bmu + \bw_i   \|_2 \geq (c_1+1) \max\{ \sqrt{d} ,~ \sqrt{n} \| \bmu \|_2 \}  \Big) < c_2 e^{-n}.
\end{align*}
Hence, there is a constant $c_1'$ such that
\begin{align*}
\PP \Big( \|  \sqrt{n-1} \bmu + \bw_i   \|_2 \geq c_1' \sqrt{n \| \bmu \|_2^2 + d }
 \Big) < c_2 e^{-n}.
\end{align*}

By (\ref{KPCA-eqn-sGMM-1}) and the definition of $\mathrm{SNR} $,
\begin{align*}
& \PP  \bigg(
\bigg| \bigg\langle \bz_i , \sum_{j \neq i} \bx_j \bar\bu_j \bigg\rangle 
\bigg | \geq  \sqrt{n} \| \bmu \|_2^2 / 2
\bigg) \notag\\
&\leq  \PP  \bigg(
\bigg| \bigg\langle \bz_i , \frac{ \sqrt{n-1} \bmu + \bw_i }{ \|   \sqrt{n-1} \bmu + \bw_i  \|_2 } \bigg\rangle 
\bigg | \geq
\frac{ \sqrt{n} \| \bmu \|_2^2  }{2  c_1' \sqrt{n \| \bmu \|_2^2 + d } }
\bigg) 
+
\PP \Big( \|  \sqrt{n-1} \bmu + \bw_i   \|_2 \geq c_1' \sqrt{n \| \bmu \|_2^2 + d }
\Big) \notag \\
& \leq \PP  \bigg(
\bigg| \bigg\langle \bz_i , \frac{ \sqrt{n-1} \bmu + \bw_i }{ \|  \sqrt{n-1} \bmu + \bw_i \|_2 } \bigg\rangle 
\bigg | \geq
\frac{ \sqrt{\mathrm{SNR}} }{2 c_1' }
\bigg) + c_2 e^{-n}.
\end{align*}
From the fact
\[
\bigg\langle \bz_i , \frac{ \sqrt{n-1} \bmu + \bw_i }{ \|  \sqrt{n-1} \bmu + \bw_i \|_2 } \bigg\rangle  \sim N(\bm{0} , \bI)
\]
we obtain that
\begin{align*}
& \PP  \bigg(
\bigg| \bigg\langle \bz_i , \sum_{j \neq i} \bx_j \bar\bu_j \bigg\rangle 
\bigg | \geq  \sqrt{n} \| \bmu \|_2^2 / 2
\bigg) 
 \leq e^{-c_3 \mathrm{SNR}} + c_2 e^{-n}.
\end{align*}
Here $c_3$ is some constant. Without loss of generality, assume that $c_3 < 1/2$. Therefore, (\ref{eqn-gmm-new-1}) leads to
\begin{align*}
\PP \bigg( \| \cH ( \bZ \bX^{\top} ) \bar\bu / \bar\lambda \|_{\infty} \geq \frac{1}{2 \sqrt{n}}
\bigg)  \leq n (e^{-c_3 \mathrm{SNR}} + c_2 e^{-n}).
\end{align*}
The right-hand side tends to zero if $\mathrm{SNR} > 2 c_3^{-1} \log n$, proving (\ref{eqn-gmm-new-0}). Consequently, when $\mathrm{SNR} >  2 c_3^{-1}  \log n > 4 \log n$ we have $\PP [\cM (\hat\by , \by ) = 0] \to 1$.

Next, consider the regime $1 \ll \mathrm{SNR} \leq 2 c_3^{-1} \log n$ and take $p = \mathrm{SNR}$. By Theorem \ref{KPCA-corollary-main},
\[
\min_{c = \pm 1} \| c \bu - \bar\bu - \cH ( \bZ \bX^{\top} ) \bar\bu / \bar\lambda \|_p = o_{\PP} (\| \bar\bu \|_p ;~ p ).
\]
Since $\bar\bu = n^{-1/2} \by $ , Lemma \ref{KPCA-lemma-GMM-error-rate} asserts that
\begin{align*}
& 
\limsup_{n\to\infty} p^{-1} \log 
\EE \cM [ \sgn(\bu) ]  =
\limsup_{n\to\infty} p^{-1} \log \bigg(
\frac{1}{n}
\EE \min_{s = \pm 1} \left|  \{ i \in [n]:~ s \sgn(u_i) \neq \sgn(\bar u_i) \} \right| 
\bigg)
\notag \\ &
\leq
\limsup_{\varepsilon \to 0}
\limsup_{n\to\infty} p^{-1} \log \bigg( \frac{1}{n} \sum_{i=1}^{n}  \PP  \left(
-[\cH ( \bZ \bX^{\top} ) \bar\bu / \bar\lambda]_i \sgn(\bar u_i) \geq (1-\varepsilon) |\bar u_i|
\right)
\bigg).
\end{align*}
From $[\cH ( \bZ \bX^{\top} ) \bar\bu]_i = \sum_{j \neq i} \langle \bz_i , \bx_j  \rangle \bar\bu_j$ and $\bar\lambda = n \| \bmu \|_2^2$ we obtain that
\begin{align*}
& \PP  \left(
-[\cH ( \bZ \bX^{\top} ) \bar\bu / \bar\lambda]_i \sgn(\bar u_i) \geq (1-\varepsilon) |\bar u_i|
\right)
\leq \PP  \left(
| [\cH ( \bZ \bX^{\top} ) \bar\bu / \bar\lambda]_i | \geq (1-\varepsilon) /\sqrt{n}
\right) \notag \\
& \leq \PP  \bigg(
\bigg| \bigg\langle \bz_i , \sum_{j \neq i} \bx_j \bar\bu_j \bigg\rangle 
\bigg | \geq (1-\varepsilon) \sqrt{n} \| \bmu \|_2^2 
\bigg) , \qquad \forall \varepsilon \in (0,1].
\end{align*}
The estimates above yield
\begin{align}
& 
\limsup_{n\to\infty} p^{-1} \log 
\EE \cM ( \sgn(\bu) , \by ) \notag \\
& \leq
\limsup_{\varepsilon \to 0}
\limsup_{n\to\infty} p^{-1} \log \PP  \bigg(
\bigg| \bigg\langle \bz_i , \sum_{j \neq i} \bx_j \bar\bu_j \bigg\rangle 
\bigg | \geq (1-\varepsilon) \sqrt{n} \| \bmu \|_2^2 
\bigg).
\label{KPCA-eqn-GMM-0}
\end{align}

Since $\sum_{j \neq i} \bx_j \bar\bu_j = (n-1) \hat{ \bmu }^{-i} / \sqrt{n}$, we get
\begin{align}
\PP  \bigg(
\bigg| \bigg\langle \bz_i , \sum_{j \neq i} \bx_j \bar\bu_j \bigg\rangle 
\bigg | \geq (1-\varepsilon) \sqrt{n} \| \bmu \|_2^2 
\bigg)
 \leq \PP  \bigg(
\bigg| \frac{ \langle \bz_i , \hat{ \bmu }^{-i} \rangle }{ \| \hat{ \bmu }^{-i} \|_2 }
\bigg | \geq \frac{( 1-\varepsilon) \| \bmu \|_2^2 }{  \| \hat{ \bmu }^{-i} \|_2 }
\bigg).
\label{KPCA-thm-GMM-0}
\end{align}

Let $R = \| \bmu \|_2$. Lemma \ref{KPCA-lemma-GMM-norms} yields $\left| 
\| \hat\bmu^{(-i)} \|_2 - \sqrt{R^2 + d / (n-1)} \right|
= O_{\PP} ( \sqrt{ p / n } ;~p )$.
Hence there exist constants $C_1$, $C_2$ and $N$ such that
\begin{align}
\PP ( \| \hat\bmu^{(-i)} \|_2 - \sqrt{R^2 + d / (n-1)} \geq C_1 \sqrt{p / n} ) \leq C_2 e^{-p}, \qquad \forall n \geq N.
\label{KPCA-thm-GMM-1}
\end{align}
On the one hand, $\sqrt{R^2 + d / (n-1)} = [1 + o(1)] \sqrt{R^2 + d / n} = [1 + o(1)] R^2 / \sqrt{p}$. On the other hand,
\[
\frac{R^2 / \sqrt{p}}{ \sqrt{p/n} } = \frac{\sqrt{n} R^2}{p} = \frac{\sqrt{n} R^2}{ R^4 / (R^2 + d / n) } = \frac{\sqrt{n} (R^2 + d / n)}{R^2} \geq \sqrt{n}.
\]
As a result, (\ref{KPCA-thm-GMM-1}) implies that for any constant $\delta > 0$, there exists a constant $N'$ such that
\begin{align}
\PP ( \| \hat\bmu^{(-i)} \|_2 \geq (1 + \delta) R^2 / \sqrt{p} ) \leq C_2 e^{-p}, \qquad \forall n \geq N'.
\label{KPCA-thm-GMM-2}
\end{align}
By (\ref{KPCA-thm-GMM-0}) and (\ref{KPCA-thm-GMM-2}),
\begin{align}
& \PP  \bigg(
\bigg| \bigg\langle \bz_i , \sum_{j \neq i} \bx_j \bar\bu_j \bigg\rangle 
\bigg | \geq (1-\varepsilon) \sqrt{n} \| \bmu \|_2^2 
\bigg) \notag \\
& \leq \PP  \bigg(
\bigg| \frac{ \langle \bz_i , \hat{ \bmu }^{-i} \rangle }{ \| \hat{ \bmu }^{-i} \|_2 }
\bigg | \geq \frac{( 1-\varepsilon) \| \bmu \|_2^2 }{ (1 + \delta) R^2 / \sqrt{p} }
\bigg) + C_2 e^{-p} \notag \\
& = \PP  \bigg(
\bigg| \frac{ \langle \bz_i , \hat{ \bmu }^{-i} \rangle }{ \| \hat{ \bmu }^{-i} \|_2 }
\bigg | \geq \frac{ 1-\varepsilon}{1 + \delta} \sqrt{p}
\bigg) + C_2 e^{-p}, \qquad \forall n \geq N'.
\label{KPCA-thm-GMM-3}
\end{align}

The independence between $\bz_i $ and $\hat{ \bmu }^{-i}$ yields $ \langle \bz_i , \hat{ \bmu }^{-i} \rangle / \| \hat{ \bmu }^{-i} \|_2 \sim N(0,1)$. Then we get 
\begin{align}
& 
\limsup_{n\to\infty} p^{-1} \log 
\EE \cM [ \sgn(\bu) ]  
\leq -1/2.
\label{KPCA-thm-GMM-4}
\end{align}
from (\ref{KPCA-thm-GMM-0}), (\ref{KPCA-thm-GMM-3}), standard tail bounds for Gaussian random variable and the fact that $\varepsilon$, $\delta$ are arbitrary.

Finally, when $ (2+\varepsilon) \log n < \mathrm{SNR} \leq 2c_3^{-1} \log n$ for some constant $\varepsilon > 0$, (\ref{KPCA-thm-GMM-4}) implies the existence of positive constants $\varepsilon'$ and $N''$ such that
\begin{align*}
\EE  \cM ( \hat\by , \by ) =
\EE  \cM ( \sgn(\bu) , \by )  \leq n^{-1 - \varepsilon'}, \qquad \forall n \geq N''.
\end{align*}
Then we must have $\PP [ \cM ( \sgn(\bu) , \by ) = 0] \to 1 $ because
\[
\PP [ \cM ( \hat\by , \by ) > 0] = \PP [ \cM ( \hat\by , \by ) \geq 1/n] \leq  \EE  \cM ( \hat\by , \by ) / n^{-1} \leq n^{-\varepsilon'} \to 0.
\]

\subsection{Proof of Theorem \ref{KPCA-thm-GMM-Lp}}\label{KPCA-thm-GMM-Lp-proof}
It is easily checked that Assumptions \ref{KPCA-assumption-incoherence}, \ref{KPCA-assumption-noise} and \ref{KPCA-assumption-concentration} hold with $\bSigma = \bI$, $\kappa = 1$, $\mu = 1$ and $\gamma \asymp \mathrm{SNR}$. Theorem \ref{KPCA-corollary-main} then yields the desired result.

\section{Proof of Section \ref{KPCA-sec-CSBM}}\label{KPCA-sec-CSBM-proof}

Define
\begin{align*}
I( t, a, b, c ) = \frac{a}{2} [ 1 - (a/b)^{t} ] +  \frac{b}{2} [1 - (b/a)^{t} ]  - 2 c ( t + t^2 )
\end{align*}
for $( t, a, b, c ) \in \RR \times (0,+\infty)^3$. It is easily seen that both $a (a/b)^t + b(b/a)^t$ and $t+t^2$ are convex and achieve their minima at $-1/2$. Then
\begin{align*}
I^*( a, b, c )  = I(-1/2, a, b, c) = \sup_{ t \in \RR } I(t, a, b, c).
\end{align*}

\subsection{Useful lemmas}

We present three useful lemmas. The first one finds an $\ell_{\infty}$ approximation of the aggregated spectral estimator $\hat\bu$. The second one concerns large deviation probabilities. The third one relates genie-aided estimators to fundamental limits of clustering.

\begin{lemma}\label{KPCA-lemma-CSBM-inf}
	Let $\bar\bu = \by / \sqrt{n}$ and
	\[
	\bw = \log(a/b) \bA \bar\bu + \frac{2 R^2}{n R^2 + d } \bG \bar\bu.
	\]
	For $\hat\bu$ defined by \eqref{eqn-CSBM-aggregation-1},
	there exist some $\varepsilon_n \to 0$ and constant $C>0$ such that 
	\[
	\PP ( \min_{c = \pm 1}  \| c \hat\bu - \bw \|_{\infty} < \varepsilon_n n^{-1/2} \log n ) > 1 - C n^{-2}.
	\]
\end{lemma}

\begin{proof}[\bf Proof of Lemma \ref{KPCA-lemma-CSBM-inf}]
	Define, as in \eqref{eqn-CSBM-aggregation},
	\begin{align*}
	& \bv = \frac{n(\alpha-\beta)}{2} \log \bigg( \frac{\alpha}{\beta} \bigg) \bu_2(\bA) + \frac{2n R^4}{n R^2 + d } \bu_1(\bG).
	\end{align*}
	Then
	\begin{align}
	\| \bv - \bw \|_{\infty} & \leq \log(a/b) \| [n(\alpha - \beta)/2] \bu_2(\bA) - \bA \bar\bu \|_{\infty} \notag\\
	& + \frac{2 R^2}{n R^2 + d } \| (nR^2) \bu_1(\bG) - \bG \bar\bu  \|_{\infty},
	\label{KPCA-theorem-CSBM-error-rate-0} \\
	\| \hat\bu - \bv \|_{\infty} & \leq \bigg|
	\lambda_2(\bA) \log \bigg( \frac{\lambda_1(\bA) + \lambda_2(\bA)}{\lambda_1(\bA) - \lambda_2(\bA)} \bigg) - \frac{n(\alpha-\beta)}{2} \log \bigg( \frac{\alpha}{\beta} \bigg)
	\bigg| \| \bu_2(\bA) \|_{\infty} \notag\\
	& + \bigg|
	\frac{2 \lambda_1^2(\bG) }{n \lambda_1(\bG) + n d } - \frac{2n R^4}{n R^2 + d }
	\bigg| \| \bu_1(\bG) \|_{\infty} . \label{KPCA-theorem-CSBM-error-rate-1}
	\end{align}
	
	For simplicity, suppose that $\langle \bu_1(\bG) , \bar\bu \rangle \geq 0 $ and $\langle \bu_2(\bA) , \bar\bu \rangle \geq 0 $. By Lemma \ref{KPCA-lemma-L2} and Theorem \ref{KPCA-corollary-main-inf}, we have
	\begin{align*}
	& | \lambda_1(\bG) - nR^2 | = o_{\PP} (1;~ n), \\
	& \| \bu_1(\bG) - \bG \bar\bu / (n R^2) \|_{\infty} = o_{\PP} ( n^{-1/2} ;~\log n ), \\
	& \| \bu_1(\bG) \|_{\infty} = O_{\PP} ( n^{-1/2} ;~\log n ).
	\end{align*}
	Hence there exists $\varepsilon_{1n} \to 0$ and a constant $C_1$ such that
	\begin{align}
	\PP ( &
	| \lambda_1(\bG) / nR^2 - 1 | < \varepsilon_{1n}, \notag \\
	& \| \bu_1(\bG) - \bG \bar\bu / (n R^2) \|_{\infty} < \varepsilon_{1n} / \sqrt{n}, ~ 
	\| \bu_1(\bG) \|_{\infty} < C_1 / \sqrt{n} ) > 1 - n^{-2}.
	\label{KPCA-theorem-CSBM-error-rate-2}
	\end{align}
	By mimicking the proof of Corollary 3.1 in \cite{AFW17} and applying Lemma 6 therein, we get $\varepsilon_{2n} \to 0$ and a constant $C_2$ such that
	\begin{align}
	\PP (&
	\max\{ | \lambda_1(\bA) - n(\alpha+\beta) / 2|,~| \lambda_2(\bA) - n(\alpha-\beta) / 2|\} < \varepsilon_{2n} \sqrt{\log n}, \notag \\
	& \| \bu_2(\bA) - \bA \bar\bu / [ n (\alpha - \beta) / 2 ] \|_{\infty} < \varepsilon_{2n} / \sqrt{n}, ~ 
	\| \bu_2(\bA) \|_{\infty} < C_2 / \sqrt{n} ) > 1 - n^{-2}.
	\label{KPCA-theorem-CSBM-error-rate-3}
	\end{align}
	Inequalities (\ref{KPCA-theorem-CSBM-error-rate-0}), (\ref{KPCA-theorem-CSBM-error-rate-1}), (\ref{KPCA-theorem-CSBM-error-rate-2}) and (\ref{KPCA-theorem-CSBM-error-rate-3}) yield some $\varepsilon_n \to 0$ and constant $C>0$ such that
	\begin{align*}
	\PP ( \| \hat\bu - \bw \|_{\infty} < \varepsilon_n n^{-1/2} \log n ) > 1 - C n^{-2}.
	\end{align*}
\end{proof}

\begin{lemma}\label{KPCA-lemma-CSBM-LDP}
Let Assumption \ref{KPCA-assumption-CSBM-q} hold and define 
	\begin{align*}
	W_{ni} = \Bigg( \frac{2 R^2}{n R^2 + d } \sum_{j \neq i} \langle \bx_i , \bx_j \rangle y_j
	  + \log(a/b) \sum_{j \neq i} A_{ij} y_j \Bigg) y_i.
	\end{align*}
For any fixed $i$,
	\begin{align*}
	\lim_{n \to \infty} q_n^{-1} \log \PP (W_{ni} \leq \varepsilon q_n)  = - \sup_{ t \in \RR } \{ \varepsilon t + I(t, a, b, c) \},
	\qquad \forall \varepsilon < \frac{a-b}{2} \log (a/b) + 2 c.
	\end{align*}
As a result, for any $\varepsilon < \frac{a-b}{2} \log (a/b) + 2 c$ and $\delta > 0$ there exists $N>0$ such that
\begin{align*}
\PP ( W_{ni} \leq \varepsilon q_n) \leq e^{- q_n [ \sup_{ t \in \RR } \{ \varepsilon t + I(t,a,b,c) \} - \delta ] } ,\qquad\forall n \geq N
\end{align*}
\end{lemma}

\begin{proof}[\bf Proof of Lemma \ref{KPCA-lemma-CSBM-LDP}]
We will invoke Lemma \ref{KPCA-lemma-LDP} to prove Lemma \ref{KPCA-lemma-CSBM-LDP}, starting from the calculation of $\EE e^{t W_{ni}}$. Conditioned on $ y_i$, $ \sum_{j \neq i} \langle \bx_i , \bx_j \rangle y_j$ and $\sum_{j \neq i} A_{ij}  y_j$ are independent. Hence
\begin{align*}
\EE ( e^{t W_{ni}} |  y_i ) = \EE \bigg[
\exp \bigg(
t \cdot \frac{2 R^2}{n R^2 + d } \sum_{j \neq i} \langle \bx_i , \bx_j \rangle y_j y_i
\bigg)
\bigg|  y_i \bigg] \cdot \EE \bigg[
\exp \bigg(
t \log(a/b)   \sum_{j \neq i} A_{ij}  y_j y_i
\bigg)
\bigg|  y_i \bigg].
\end{align*}
We claim that for any fixed $t \in \RR$, there exists $N > 0$ such that when $n > N$,
\begin{align}
& \log \EE \bigg[
\exp \bigg(
t \cdot \frac{2 R^2}{n R^2 + d } \sum_{j \neq i} \langle \bx_i , \bx_j \rangle y_j y_i
\bigg)
\bigg|  y_i \bigg]
=
\log \EE \bigg[ \exp \bigg(
t \cdot \frac{2 R^2}{n R^2 + d } \sum_{j \neq i} \langle \bx_i , \bx_j \rangle y_j y_i
\bigg)
\bigg] \notag\\
& = 2 c (t + t^2)  [ 1 + o(1) ] q_n , \label{KPCA-lemma-CSBM-LDP-1} \\
&\log \EE \bigg[
\exp \bigg(
t \log(a/b)  y_i \sum_{j \neq i} A_{ij}  y_j
\bigg)
\bigg|  y_i \bigg] = \log \EE \bigg[
\exp \bigg(
t \log(a/b)  y_i \sum_{j \neq i} A_{ij}  y_j
\bigg) \bigg] \notag \\
& = \frac{ a  [ (a/b)^{t} - 1 ] + b  [ (b/a)^{t} -1 ]  }{2}   [ 1 + o(1) ] q_n.
\label{KPCA-lemma-CSBM-LDP-2}
\end{align}
If (\ref{KPCA-lemma-CSBM-LDP-1}) and (\ref{KPCA-lemma-CSBM-LDP-2}) hold, then 
\[
\EE ( e^{t W_{ni}} |  y_i ) = \EE \bigg[
\exp \bigg(
t \cdot \frac{2 R^2}{n R^2 + d } \sum_{j \neq i} \langle \bx_i , \bx_j \rangle y_j y_i 
\bigg) \bigg] \cdot \EE
\exp \bigg(
t \log(a/b)  y_i \sum_{j \neq i} A_{ij}  y_j
\bigg)
\]
does not depend on $ y_i$, and
\begin{align*}
q_n^{-1} \log \EE e^{t W_{ni}} & = q_n^{-1} \log \EE \bigg[
\exp \bigg(
t \cdot \frac{2 R^2}{n R^2 + d } \sum_{j \neq i} \langle \bx_i , \bx_j \rangle y_j y_i 
\bigg) \bigg] \\
&  + q_n^{-1} \log \EE
\bigg[ \exp \bigg(
t \log(a/b)  y_i \sum_{j \neq i} A_{ij}  y_j
\bigg) \bigg]\\
& = \bigg(  \frac{a}{2} [ (a/b)^{t} - 1 ] + \frac{b}{2}  [ (b/a)^{t} -1 ]  + 2 c (t + t^2) \bigg) [ 1 + o(1) ]\\
& = - I(t, a, b, c) [ 1 + o(1) ].
\end{align*}
Lemma \ref{KPCA-lemma-LDP} implies that for $\varepsilon < -\frac{\partial}{\partial t} I(t,a,b,c) |_{t = 0} =  \frac{a-b}{2} \log (a/b) + 2c$, 
\begin{align*}
\lim_{n \to \infty} q_n^{-1}  \log \PP (W_{ni} \leq \varepsilon q_n) = - \sup_{ t \in \RR } \{ \varepsilon t + I(t, a, b, c) \}.
\end{align*}

Below we prove (\ref{KPCA-lemma-CSBM-LDP-1}) and (\ref{KPCA-lemma-CSBM-LDP-2}). From $\bx_i  = \bmu  y_i + \bz_i $ we see that given $ y_i$, $\bx_i  y_i \sim N(\bmu, \bI_d)$ is independent of $\sqrt{n-1}\hat\bmu^{(-i)} \sim N( \sqrt{n-1} \bmu,  \bI_d )$. Lemma \ref{KPCA-lemma-chi-square} asserts that
\begin{align*}
& \log \EE ( e^{t \langle \bx_i, \hat\bmu^{(-i)} \rangle  y_i } |  y_i )
 = \log \EE ( e^{ (t / \sqrt{n-1}) \langle \bx_i  y_i , \sqrt{n - 1} \hat\bmu^{(-i)} \rangle } |  y_i ) \\
& = \frac{ (\frac{t}{\sqrt{ n - 1 }})^2 }{ 2[ 1 - (\frac{t}{\sqrt{ n - 1 }})^2 ] } ( \| \bmu \|_2^2 + (n-1) \| \bmu \|_2^2 ) \\
&~~ + \frac{ \frac{t}{\sqrt{ n - 1 }} }{ 1 - (\frac{t}{\sqrt{ n - 1 }})^2  } \langle \bmu, \sqrt{n-1} \bmu \rangle
- \frac{d}{2} \log \bigg[
1 - \bigg( \frac{t}{\sqrt{n - 1 }} \bigg)^2
\bigg] \\
& =  \frac{  t R^2 }{1 - t^2 / (n-1)} \bigg( 1 + \frac{nt}{2(n-1)} \bigg) - \frac{d}{2} \log \bigg( 1 - \frac{t^2}{n - 1} \bigg), \qquad \forall t \in ( - \sqrt{n-1} , \sqrt{ n - 1} ).
\end{align*}
Since the right hand side does not depend on $ y_i$, $ \log \EE e^{t \langle \bx_i, \hat\bmu^{(-i)} \rangle  y_i }$ is also equal to it.
Now we fix any $t \in \RR$ and let $s = 2 t p / R^2 = 2 t / [1 + d / (nR^2)]$. Since $|s| < 2 |t|$, we have $|s| < \sqrt{ n - 1}$ for large $n$. In that case, we obtain from the equation above that 
\begin{align*}
& \log \EE \bigg[ \exp \bigg( t \cdot \frac{2 \langle \hat\bmu^{(-i)} , \bx_i \rangle  y_i }{1 + d / (nR^2) } \bigg) \bigg]
=  \log \EE 
e^{ s \langle \bx_i, \hat\bmu^{(-i)} \rangle  y_i }  \\
&=\frac{  s R^2 }{1 - s^2 / (n-1)} \bigg( 1 + \frac{ns}{2(n-1)} \bigg) - \frac{d}{2} \log \bigg( 1 - \frac{s^2}{n - 1} \bigg).
\\& 
= [1 + o(1) ] s R^2  (1 + s / 2) + \frac{d}{2} \cdot \frac{s^2}{n - 1} [1+o(1)]
= \bigg[ 2tp 
\bigg( 1 + \frac{tp}{R^2} \bigg)
+ \frac{d}{2n} \cdot \frac{4t^2 p^2}{R^4}
\bigg][1+o(1)] \\
&=2pt \bigg[ 1 + \frac{tp}{R^2} \bigg( 1 + \frac{d}{nR^2} \bigg)  \bigg] [1 + o(1)] = 2pt(1+t) [1+o(1)],
\end{align*}
where we used $p = R^4 / ( R^2 + d / n )$.
It then follows from the results above and the assumption $p = c q_n$ that
\begin{align*}
\log \EE \bigg[ \exp \bigg( t \cdot \frac{2 \langle \hat\bmu^{(-i)} , \bx_i \rangle  y_i }{1 + d / (nR^2) }  \bigg) \bigg]
= c q_n p^{-1} \log \EE \bigg[ \exp \bigg( t \cdot \frac{2 \langle \hat\bmu^{(-i)} , \bx_i \rangle  y_i }{1 + d / (nR^2) }   \bigg)\bigg] =  2 c  (t + t^2)  [ 1+ o(1) ] q_n,
\end{align*} 
which leads to (\ref{KPCA-lemma-CSBM-LDP-1}).

On the other hand,
\begin{align*}
& \EE ( e^{ t A_{ij}  y_i  y_j } |  y_i )  = \frac{1}{2} \EE ( e^{ t A_{ij} } |   y_i  y_j = 1) + \frac{1}{2} \EE ( e^{ - t A_{ij} } |   y_i  y_j = - 1) \\
&=\frac{1}{2} [ u e^{t} + (1-u) ] + \frac{1}{2} [ v e^{-t} + (1-v) ]
= 1 + \frac{ u (e^{t} - 1 ) + v (e^{-t} -1) }{2}, \qquad \forall  t \in \RR.
\end{align*}
Conditioned on $ y_i$, $\{ A_{ij}  y_i  y_j \}_{j\neq i}$ are i.i.d. random variables. Hence
\begin{align*}
\EE \bigg[
\exp \bigg(
t \log(a/b)  y_i \sum_{j \neq i} A_{ij}  y_j
\bigg) \bigg|  y_i \bigg] = \bigg( 1 + \frac{ u [ (a/b)^{t} - 1 ] + v [ (b/a)^{t} -1 ] }{2} \bigg)^{n-1}.
\end{align*}
Again, the right-hand side does not depend on $ y_i$. By substituting $u = a q_n / n$ and $v = b q_n / n$,
\begin{align*}
& \log \EE \bigg[ \exp \bigg(
t \log(a/b)  y_i \sum_{j \neq i} A_{ij}  y_j
\bigg) \bigg]  = (n-1) \log \bigg( 1 + \frac{ u [ (a/b)^{t} - 1 ] + v [ (b/a)^{t} -1 ] }{2} \bigg) \\
& =  (n-1) \log \bigg( 1 + \frac{ a q_n [ (a/b)^{t} - 1 ] + b q_n [ (b/a)^{t} -1 ] }{2n} \bigg) \\
& =  \frac{ a  [ (a/b)^{t} - 1 ] + b  [ (b/a)^{t} -1 ]  }{2} \cdot  [ 1 + o(1) ] q_n .
\end{align*}
We get (\ref{KPCA-lemma-CSBM-LDP-2}) and thus finish the proof.
\end{proof}

\begin{lemma}[Fundamental limit via genie-aided approach]\label{KPCA-lemma-fundamental}
	Suppose that $\cS$ is a Borel space and $(\by, \bX)$ is a random element in $\{ \pm 1 \}^n \times \cS$. Let $\cF$ be a family of Borel mappings from $\cS$ to $\{ \pm 1 \}^n$. Define
	\begin{align*}
	&\cM ( \bu, \bv ) = \min \bigg\{
	\frac{1}{n} \sum_{i=1}^{n} \mathbf{1}_{ \{ u_i \neq v_i \} },~
	\frac{1}{n} \sum_{i=1}^{n} \mathbf{1}_{ \{ - u_i \neq v_i \} }
	\bigg\},\qquad
	\forall \bu,~ \bv \in \{ \pm 1 \}^n,\\
	&f(\cdot | \tilde\bX, \tilde{\by}_{-i}   ) = \PP ( y_i = \cdot | \bX = \tilde\bX, \by_{-i} = \tilde{\by}_{-i}   ),\qquad
	\forall i \in [n],~ \tilde\bX \in \cS,~ \tilde{\by}_{-i} \in \{ \pm 1 \}^{n-1}.
	\end{align*}
	We have
	\begin{align*}
	\inf_{\hat{ \by } \in \cF }
	\EE \cM ( \hat{ \by }, \by )
	\geq \frac{n-1}{3n-1} \cdot \frac{1}{n} \sum_{i=1}^{n}  \PP  \left[
	f ( y_i | \bX, \by_{-i} )
	< f ( - y_i | \bX, \by_{-i} )
	\right].
	\end{align*}
\end{lemma}

\begin{proof}[\bf Proof of Lemma \ref{KPCA-lemma-fundamental}]
	
	For $\bu, \bv \in \{ \pm 1 \}^m$ with some $m \in \ZZ_+$, define the sign
	$$s(\bu, \bv ) = \argmin_{c = \pm 1} \|c \bu - \bv \|_1$$ with any tie-breaking rule. As a matter of fact, $s(\bu, \bv ) = \sgn( \langle \bu , \bv \rangle )$ if $ \langle \bu , \bv \rangle \neq 0$. When $| \langle \bu , \bv \rangle | > 1$, we have $s(\bu_{-i}, \bv_{-i} ) = s(\bu, \bv ) $  for all $i$. 
	Hence for $\hat{ \by }  \in \cF$ (we drop the dependence of $\hat{ \by }$ on $\bX$),
	\begin{align*}
	& \EE \cM ( \hat{ \by }, \by )
	 \geq  \EE \left( \frac{1}{n} \sum_{i=1}^{n} \mathbf{1}_{ \{  s(\hat{ \by }, \by ) \hat y_i  \neq  y_i \} }
	\mathbf{1}_{
		\{ |\langle \hat{ \by }, \by \rangle| > 1  \}
	}
	\right) \\
	& = \EE \left( \frac{1}{n} \sum_{i=1}^{n} \mathbf{1}_{ \{  s(\hat{ \by }_{-i} , \by_{-i} )  \hat y_i   \neq  y_i \} }
	\mathbf{1}_{
		\{ |\langle \hat{ \by }, \by \rangle| > 1  \}
	}
	\right) \\
	& = \EE \left( \frac{1}{n} \sum_{i=1}^{n} \mathbf{1}_{ \{  s(\hat{ \by }_{-i}, \by_{-i} ) \hat y_i  \neq  y_i \} }
	\right) -
	\EE \left( \frac{1}{n} \sum_{i=1}^{n} \mathbf{1}_{ \{  s(\hat{ \by }_{-i} , \by_{-i} ) \hat y_i  \neq  y_i \} }
	\mathbf{1}_{
		\{ |\langle \hat{ \by }, \by \rangle| \leq 1  \}
	}
	\right)
	\\
	& \geq \frac{1}{n} \sum_{i=1}^{n}  \PP  \left(
	s(\hat{ \by }_{-i}, \by_{-i} ) \hat y_i  \neq  y_i 
	\right) - \PP ( |\langle \hat{ \by }, \by \rangle| \leq 1 ).
	\end{align*}
	
	Define $\cF_{\varepsilon} = \{ 
	\hat{ \by } \in \cF:~ \PP (
	|\langle \hat{ \by }, \by \rangle| \leq 1
	) \leq \varepsilon
	\}$ for $\varepsilon \in [0,1]$.
	If $\cF_{\varepsilon} \neq \varnothing$, then
	\begin{align*}
	\inf_{\hat{ \by } \in \cF_{\varepsilon} }
	\EE \cM ( \hat{ \by }, \by )
	& \geq \frac{1}{n} \sum_{i=1}^{n} \inf_{\hat{ \by } \in \cF } \PP  \left(
	s(\hat{ \by }_{-i}, \by_{-i} ) \hat y_i  \neq  y_i 
	\right) - \varepsilon.
	\end{align*}
	Define $\cG$ be the family of Borel mappings from $\cS \times \{ \pm 1 \}^{n-1} \to \{ \pm 1 \}$. 
	For any fixed $\hat{\by} \in \cF$, the mapping $(\bX, \by_{-i} ) \mapsto s(\hat{ \by }_{-i}, \by_{-i} ) \hat y_i$ belongs to $\cG$. Then
	\begin{align*}
	\inf_{\hat{ \by } \in \cF } \PP  \left(
	s(\hat{ \by }_{-i}, \by_{-i} ) \hat y_i \neq  y_i 
	\right)
	\geq \inf_{\hat{ \bm\ell } \in \cG } \PP  \left(
	\hat{ \bm\ell } ( \bX , \by_{-i} ) \neq  y_i 
	\right)
	\geq \PP  \left[
	f (  y_i | \bX, \by_{-i} )
	< f ( -  y_i | \bX, \by_{-i} )
	\right],
	\end{align*}
	where the last inequality follows from the Neyman-Pearson lemma \citep{NPe33}. Let $\delta = \frac{1}{n} \sum_{i=1}^{n}  \PP  [
	f (  y_i | \bX, \by_{-i} )
	< f ( -  y_i | \bX, \by_{-i} )
	]$. We have $\inf_{\hat{ \by } \in \cF_{\varepsilon} }
	\EE \cM ( \hat{ \by }, \by )
	\geq \delta - \varepsilon$ provided that $\cF_{\varepsilon} \neq \varnothing$.

	On the other hand, when $
	|\langle \hat{ \by }, \by \rangle| \leq 1
	$, we have
	\begin{align*}
	\cM ( \hat{ \by }, \by ) &= (4n)^{-1} \min_{c = \pm 1} \| c \hat{ \by } - \by \|_2^2 \\
	& = (4n)^{-1} \min_{c = \pm 1} \{
	\| \hat{ \by }  \|_2^2 - 2 c \langle \hat{ \by }, \by \rangle
	+ \| \by \|_2^2 \}
	\geq \frac{n - 1}{2n}.
	\end{align*}
	Hence if $\cF \backslash \cF_{\varepsilon} \neq \varnothing$,
	\begin{align*}
	\inf_{\hat{ \by } \in \cF \backslash \cF_{\varepsilon} }
	\EE \cM ( \hat{ \by }, \by ) \geq \frac{n-1}{2n}
	\inf_{\hat{ \by } \in \cF \backslash \cF_{\varepsilon} } \PP (|
	\langle
	\hat{ \by } (\bX), \by
	\rangle
	| \leq 1) \geq  \frac{n-1}{2n} \cdot \varepsilon = \frac{\varepsilon}{2} \bigg( 1 - \frac{1}{n} \bigg).
	\end{align*}
	
	Based on the deduction above, we have the followings for all $\varepsilon \in [0,1]$:
	\begin{enumerate}
		\item If $\cF_{\varepsilon} \neq \varnothing$ and $\cF \backslash \cF_{\varepsilon} \neq \varnothing$, then
		$\inf_{\hat{ \by } \in \cF }
		\EE \cM ( \hat{ \by }, \by )
		\geq \min \{ \delta - \varepsilon,~ \varepsilon (1 - n^{-1}) / 2 \}$;
		\item If $\cF_{\varepsilon} = \varnothing$, then $\inf_{\hat{ \by } \in \cF }
		\EE \cM ( \hat{ \by }, \by )
		\geq \varepsilon (1 - n^{-1}) / 2$.
		\item If $\cF \backslash \cF_{\varepsilon} = \varnothing$, then $\inf_{\hat{ \by } \in \cF }
		\EE \cM ( \hat{ \by }, \by )
		\geq \delta - \varepsilon$.
	\end{enumerate}
	As a result, $\inf_{\hat{ \by } \in \cF }
	\EE \cM ( \hat{ \by }, \by )
	\geq \sup_{\varepsilon \in [0,1]} \min \{ \delta - \varepsilon,~ \varepsilon (1 - n^{-1}) / 2 \}
	= \frac{n-1}{3n-1} \delta$.
\end{proof}

\subsection{Proof of Lemma \ref{KPCA-lemma-CSBM-linearization}}\label{KPCA-lemma-CSBM-linearization-proof}

The proof directly follows the Lemmas \ref{KPCA-lemma-GMM-linearization} and \ref{KPCA-lemma-SBM-linearization}, plus the conditional independence between $\bA$ and $\bX$ as well as the Bayes formula. See Appendices \ref{KPCA-lemma-GMM-linearization-proof} and \ref{KPCA-lemma-SBM-linearization-proof} for proofs of lemmas.

\begin{lemma}\label{KPCA-lemma-GMM-linearization}
	Denote by $p_X(\cdot | \tilde{\ell}_i, \tilde{\by}_{-i} )$ the conditional density function of $\bX$ given $y_i = \tilde{\ell}_i \in \{ \pm 1 \}$ and $\by_{-i} = \tilde{\by}_{-i} \in \{ \pm 1 \}^{n-1}$.
	Under Assumption \ref{KPCA-assumption-CSBM-q},
	\begin{align*}
	\bigg| y_i \log 
	\bigg(
	\frac{ p_X ( \bX |  y_i, \by_{-i}  ) }{ p_X ( \bX | - y_i, \by_{-i}  ) }
	\bigg)
	-  \frac{2 R^2}{n R^2 + d } \sum_{j \neq i} \langle \bx_i , \bx_j \rangle y_j  \bigg| 
	= o_{\PP} (q_n;~q_n),\qquad \forall i.
	\end{align*}
\end{lemma}

\begin{lemma}\label{KPCA-lemma-SBM-linearization}
	Denote by $p_A (\cdot | \tilde y_i, \tilde{\by}_{-i} )$ the conditional probability mass function of $\bA$ given $y_i = \tilde{\ell}_i$ and $\by_{-i} = \tilde{\by}_{-i}$.
	Under Assumption \ref{KPCA-assumption-CSBM-q},
	\begin{align*}
	\bigg| 	y_i \log \bigg( \frac{ p_A ( \bA |  y_i, \by_{-i}  ) }{ p_A ( \bA | - y_i, \by_{-i}  ) } \bigg) -
 \log \bigg( \frac{a}{b} \bigg) 
	\sum_{j \neq i} A_{ij} y_j
	\bigg| = o_{\PP} ( q_n;~q_n ) , \qquad \forall i.
	\end{align*}
\end{lemma}

\subsection{Proof of Lemma \ref{KPCA-lemma-GMM-linearization}}\label{KPCA-lemma-GMM-linearization-proof}
Let $p = p_n = R^4 / (R^2 + d / n)$. We have $p_n \asymp q_n$.
First of all, from the data generating model, we have
$$
p_X( \bX | \by  ) \propto \EE_{\bmu} \exp \Big( - \frac{1}{2} \sum_{j=1}^n \| \bx_j - y_j \bmu\|^2 \Big)  \propto \EE_{\bmu} \exp \Big( \Big\langle \sum_{j=1}^n  \bx_j y_j, \bmu \Big\rangle \Big),
$$
where $\propto$ hide quantities that do not depend on $\by$.
By defining
\begin{align*}
I(\balpha) = R^{d-1} \int_{ \SSS^{d-1} } e^{ R\langle \balpha, \tilde\bmu \rangle } \rho( \mathrm{d} \tilde\bmu ), 
\qquad \forall  \balpha \in \RR^{d},
\end{align*}
and using the uniform distribution of $\bmu$ on the sphere with radius $R$, we get
\begin{align}
\frac{ p_X( \bX | y_i =  s, \by_{-i}  ) }{ p_X ( \bX | y_i = - s, \by_{-i}  ) } 
= \frac{ I \left( (n-1)\hat\bmu^{(-i)} + \bx_i s \right) }{ I \left( (n-1)\hat\bmu^{(-i)} - \bx_i s \right) }.
\label{KPCA-eqn-genie-aided-1}
\end{align}

Let $
P (t, s) = \int_{0}^{\pi} e^{ t \cos \theta } (\sin \theta)^{s - 2} \rd \theta$ for
$ t \geq 0,~ s \geq 2$.  Then,
\begin{align*}
I(\balpha) 
\propto \int_{0}^{\pi} e^{ R \| \balpha \|_2 \cos \theta } (\sin \theta)^{d - 2} \rd \theta
= P ( R \| \balpha \|_2, d ),
\end{align*}
where $\propto $ only hides some factor that does not depend on $\balpha$.
Hence by (\ref{KPCA-eqn-genie-aided-1}) and $\hat\bmu^{(-i)} = \frac{1}{n - 1} \sum_{j \neq i} y_j \bx_j$,
\begin{align*}
& \log \bigg( \frac{ p_X ( \bX |   y_i, \by_{-i}  ) }{ p_X ( \bX | -  y_i, \by_{-i}  ) } \bigg) \\
& = \log P \left( R \| (n-1) \hat\bmu^{(-i)} + \bx_i  y_i  \|_2 , d \right) - \log P \left( R \| (n-1) \hat\bmu^{(-i)} - \bx_i  y_i  \|_2 , d \right).
\end{align*}

We will linearize the functional above, and invoke Lemma \ref{KPCA-lemma-integral-derivative} to control the approximation error. Take $t_0 = (n-1) R \sqrt{  R^2 + d / (n-1) }$, $t_1 = R \| (n-1) \hat\bmu^{(-i)} + \bx_i  y_i  \|_2 $, $t_2 = R \| (n-1) \hat\bmu^{(-i)} - \bx_i  y_i  \|_2 $. We first claim that
\begin{align}
&t_0 = nR \sqrt{R^2 + d / n } [ 1 + o(1) ] =[1+o(1)] nR^3 / \sqrt{p} \asymp n R (R \vee \sqrt{d / n}) , \label{KPCA-eqn-genie-aided-2.5}\\
&\max\{ 1/t_0,~ d^2 / t_0^3,~ |t_2 - t_0| /t_0,~ |t_1 - t_0| /t_0 \} = o_{\PP} (1;~p).
\label{KPCA-eqn-genie-aided-3}
\end{align}
Equation (\ref{KPCA-eqn-genie-aided-2.5}) is obvious and it leads to $1 / t_0 = o(1)$. From $t_0 \gtrsim R \sqrt{nd}$ and the assumption $R \gg 1 \vee (d/n)^{1/4}$ we get
\begin{align*}
\frac{d^2}{ t_0^3 } \lesssim \frac{ d^2 }{ (R \sqrt{nd})^3 } = \bigg(
\frac{d^4}{(R^2 nd)^3}
\bigg)^{1/2}
=\bigg(
\frac{d}{nR^4} \cdot \frac{1}{n^2 R^2}
\bigg)^{1/2} = o(1).
\end{align*}
By the triangle's inequality and $\| \bx_i \|_2 = O_{\PP} ( R \vee \sqrt{d};~p )$ in Lemma \ref{KPCA-lemma-GMM-norms},
\begin{align*}
\left| 
|t_1 - t_0| - \left| R \| (n-1) \hat\bmu^{(-i)} \|_2 - t_0\right| 
\right|
& \leq R \| \bx_i  y_i  \|_2 \leq R (R \vee \sqrt{d}) O_{\PP} ( 1 ;~p ).
\end{align*}
By $\left| 
\| \hat\bmu^{(-i)} \|_2 - \sqrt{R^2 + d / (n-1)} \right|
= O_{\PP} ( \sqrt{ p / n } ;~p )$ in Lemma \ref{KPCA-lemma-GMM-norms},
$\left| R \| (n-1) \hat\bmu^{(-i)} \|_2 - t_0\right| 
= O_{\PP} ( R \sqrt{np} ;~p )$. Hence
$|t_1 - t_0| / R = O_{\PP} ( R \vee \sqrt{d} \vee \sqrt{np} ;~p ) = O_{\PP} ( \sqrt{n} R \vee \sqrt{d} ;~p )$ as $\sqrt{p} \leq R$. Then $t_0 \asymp n R (R \vee \sqrt{d / n}) $ forces 
\begin{align*}
|t_1 - t_0| / t_0 = \frac{|t_1 - t_0| / R}{|t_0| / R} = \frac{
	O_{\PP} ( \sqrt{n} R \vee \sqrt{d} ;~p )
}{
	nR \vee \sqrt{nd } 
}= o_{\PP} (1;~p).
\end{align*}
Similarly, $|t_2 - t_0| / t_0 = o_{\PP} (1;~p)$.

Now that (\ref{KPCA-eqn-genie-aided-3}) has been justified, Lemma \ref{KPCA-lemma-integral-derivative} and Fact \ref{KPCA-fact-transform} assert that
\begin{align}
& \bigg| \frac{ \log p_X ( \bX |   y_i, \by_{-i}  ) - \log p_X ( \bX | -  y_i, \by_{-i}  )    }{ g(t_0, d) (t_2 - t_1) } - 1 \bigg| \notag \\
& =
\bigg| \frac{ \log P(t_2, d) - \log P(t_1, d) }{ g(t_0, d) (t_2 - t_1) } - 1 \bigg|  = o_{\PP} (1;~p),
\label{KPCA-eqn-genie-aided-4}
\end{align}
where
\begin{align*}
g (t_0 , d) & = \frac{\sqrt{ [ (d-2) / t_0 ]^2 + 4} - (d-2) / t_0 }{2} = \frac{ \sqrt{ (d-2)^2 + 4 t_0^2 } - (d-2) }{2 t_0}.
\end{align*}
By (\ref{KPCA-eqn-genie-aided-2.5}), we have $t_0 = [1+o(1)] nR^3 / \sqrt{p}$ and
\begin{align*}
g (t_0 , d) & = \frac{ \sqrt{p} [ 1 + o(1) ] }{ 2 n R^3 } [ \sqrt{ (d-2)^2 + 4 n^2 R^6 / p } - (d-2) ] \\
& = \frac{\sqrt{p}  [ 1 + o(1) ]  }{ 2 R } \cdot 
\bigg[
\sqrt{ \bigg( \frac{d-2}{ n R^2 } \bigg)^2 + \frac{4 R^2}{p}
}
-
\frac{d-2}{ n R^2 }
\bigg].
\end{align*}
Since $p = R^4 / ( R^2 + d / n )$ and
\begin{align*}
\bigg( \frac{d-2}{ n R^2 } \bigg)^2 + \frac{4 R^2}{p}
= \bigg( \frac{d-2}{ n R^2 } \bigg)^2 + 4 \bigg( 1 + \frac{d}{nR^2} \bigg)
= \bigg( \frac{d-2}{ n R^2 } + 2 \bigg)^2 + \frac{8}{nR^2},
\end{align*}
we have 
\begin{align*}
\frac{d-2}{ n R^2 } + 2
\leq
\sqrt{ \bigg( \frac{d-2}{ n R^2 } \bigg)^2 + \frac{4 R^2}{p} }
\leq \frac{d-2}{ n R^2 } + 2 + \sqrt{\frac{8}{nR^2}}
\end{align*}
and $g (t_0 , d) =  [ 1 + o(1)] \sqrt{p} /  R $.

To further simplify (\ref{KPCA-eqn-genie-aided-4}), we first note that
\begin{align*}
t_1 - t_2 & = \frac{t_1^2 - t_2^2}{t_1 + t_2}
= \frac{ 4 R^2 (n-1) \langle \hat\bmu^{(-i)} , \bx_i \rangle  y_i }{ t_1 + t_2 }
=  \frac{4 R^2 (n-1) \langle \hat\bmu^{(-i)} , \bx_i \rangle  y_i  [1 + o_{\PP} (1;~p) ] }{2 t_0}
\\&
= \frac{ 4 R^2 (n-1) \langle \hat\bmu^{(-i)} , \bx_i \rangle  y_i [1 + o_{\PP} (1;~p) ] }{ 2 n R^3 / \sqrt{p} }
= \bigg( \frac{2 \sqrt{p}}{nR} \sum_{j \neq i} \langle \bx_i , \bx_j \rangle y_j \bigg) y_i [1 + o_{\PP} (1;~p) ],
\end{align*}
where we used $t_0 = [1+o(1)] nR^3 / \sqrt{p}$ in (\ref{KPCA-eqn-genie-aided-2.5}). Then
\[
g(t_0, d) (t_1 - t_2) = \bigg( \frac{ 2 p}{n R^2} \sum_{j \neq i} \langle \bx_i , \bx_j  \rangle y_j \bigg) y_i [1 + o_{\PP} (1;~p) ].
\]
By $\langle \hat\bmu^{(-i)} , \bx_i \rangle = O_{\PP} (R^2;~ p)$ in Lemma \ref{KPCA-lemma-GMM-norms}, $g(t_0, d)
( t_1 - t_2 ) = O_{\PP} ( p;~p )$.
The proof is completed by plugging these estimates into (\ref{KPCA-eqn-genie-aided-4}).

\subsection{Proof of Lemma \ref{KPCA-lemma-SBM-linearization}}\label{KPCA-lemma-SBM-linearization-proof} 

Define $T_i = \{ j \in [n] \backslash \{i\} :~ y_i  y_j = 1 \}$ and $S_i = \{ j \in [n]  \backslash \{i\} :~A_{ij} = 1 \}$ for $i \in [n]$. By definition,  
\begin{align*}
&p_A ( \bA |  y_i, \by_{-i} ) \propto \alpha^{ |T_i \cap S_i| }
(1-\alpha)^{ | T_i \backslash S_i | }
\beta^{ | S_i \backslash T_i | }
(1-\beta)^{ [n] \cap \{ i \}^c \cap T_i^c \cap S_i^c }, \\
&p_A ( \bA | - y_i, \by_{-i} ) \propto \alpha^{ | S_i \backslash T_i | }
(1-\alpha)^{ [n] \cap \{ i \}^c \cap T_i^c \cap S_i^c }
\beta^{|T_i \cap S_i|  }
(1-\beta)^{  | T_i \backslash S_i | },
\end{align*}
where both $\propto$'s hide the same factor that does not involve $\{ A_{ij} \}_{j = 1}^n$ or $ y_i$. Hence
\begin{align}
\log \bigg(
\frac{
	p_A ( \bA |  y_i, \by_{-i} )
}{
	p_A ( \bA | - y_i, \by_{-i} )
}
\bigg)
& = (|T_i \cap S_i| - | S_i \backslash T_i | ) \log (\alpha/\beta) \notag\\
& + (  | T_i \backslash S_i | - |[n] \cap \{ i \}^c \cap T_i^c \cap S_i^c| ) \log \bigg(
\frac{1-\beta}{1-\alpha}
\bigg).
\label{KPCA-eqn-SBM-1}
\end{align}

The facts $|T_i| - |S_i| \leq | T_i \backslash S_i | \leq |T_i|$ and $n- 1 - |T_i| - |S_i| \leq |[n] \cap \{ i \}^c \cap T_i^c \cap S_i^c| \leq n - 1 - |T_i|$ yield
\begin{align*}
\left|
| T_i \backslash S_i | - |[n] \cap \{ i \}^c \cap T_i^c \cap S_i^c|
\right|
\leq \left| 2 |T_i| - (n-1) \right| + |S_i|
\end{align*}
For any independent random variables $\{ \xi_i \}_{i=1}^n$ taking values in $[-1, 1]$, Hoeffding's inequality \citep{Hoe63} asserts
$\PP (
| \sum_{i}^{n} \xi_i - \sum_{i}^{n} \EE  \xi_i | \geq n t 
) \leq 2 e^{- nt^2 / 2 }$, $\forall t \geq 0$. Hence $| \sum_{i}^{n} \xi_i - \sum_{i}^{n} \EE  \xi_i |  = O_{\PP} ( \sqrt{n q};~q )$. This elementary fact leads to
\begin{align*}
& \left| 2 |T_i| - (n-1) \right|  = O_{\PP} (
\sqrt{n q} ;~q ),  \\
& |S_i| \leq |\EE S_i|  + |S_i - \EE S_i| \leq O(q) + O_{\PP} (
\sqrt{n q} ;~q ) = O_{\PP} (
\sqrt{n q} ;~q ). \notag
\end{align*}
As a result, $ \left|
| T_i \backslash S_i | - |[n] \cap \{ i \}^c \cap T_i^c \cap S_i^c|
\right|
=
O_{\PP} (
\sqrt{n q} ;~q )$.
This bound, combined with
\begin{align*}
0 \leq \log \bigg(
\frac{1-\beta}{1-\alpha}
\bigg) = \log \bigg(
1 + 
\frac{\alpha-\beta}{1-\alpha}
\bigg)
\leq \frac{\alpha-\beta}{1-\alpha} = \frac{ (a-b) q / n }{1 - a q/ n } \lesssim \frac{q}{n},
\end{align*}
(\ref{KPCA-eqn-SBM-1}) and $\log (\alpha/\beta) = \log (a/b)$, implies that
\begin{align*}
\left| \log \bigg(
\frac{
	p_A ( \bA |  y_i, \by_{-i} )
}{
	p_A ( \bA | - y_i, \by_{-i} )
}
\bigg) - (|T_i \cap S_i| - | S_i \backslash T_i | ) \log (a/b) \right| 
= O_{\PP} ( \sqrt{ nq } \cdot q / n ; ~q) = o_{\PP} ( q ; ~ q).
\end{align*}
The proof is completed by $$|T_i \cap S_i| - | S_i \backslash T_i | = \sum_{j \in T_i} A_{ij} - \sum_{j \in [n] \cap \{ i \}^c \cap T_i^c } A_{ij} =  y_i \sum_{j \neq i} A_{ij}  y_j .$$

\subsection{Proof of Theorem \ref{KPCA-theorem-CSBM-eaxct}}\label{KPCA-theorem-CSBM-eaxct-proof}
Lemma \ref{KPCA-lemma-CSBM-inf} asserts the existence of some $\varepsilon_n \to 0$ and constant $C>0$ such that
\begin{align}
\PP ( \min_{c = \pm 1} \| c \hat\bu - \bw \|_{\infty} < \varepsilon_n n^{-1/2} \log n ) > 1 - C n^{-2}.
\label{KPCA-theorem-CSBM-eaxct-1}
\end{align}
Let $\hat{ c} =  \argmin_{c = \pm 1} \| c \hat\bu - \bw \|_{\infty}$ and $\bv = \hat{ c} \hat\bu$.
Hence
\begin{align}
\PP [ \cM(\sgn(\hat\bu) , \by ) = 0 ] & \geq \PP ( \sgn(\hat\bv) = \by ) \notag\\
& \geq \PP ( \min_{i \in [n]} w_i y_i > \varepsilon_n n^{-1/2} \log n ,~ \| \bv - \bw \|_{\infty} < \varepsilon_n n^{-1/2} )  \notag\\
& \geq  \PP ( \min_{i \in [n]} w_i y_i > \varepsilon_n n^{-1/2}\log n ) - \PP ( \| \bv - \bw \|_{\infty} < \varepsilon_n n^{-1/2} )\notag \\
& \geq 1 - \sum_{i=1}^{n} \PP (  w_i y_i \leq \varepsilon_n n^{-1/2}\log n ) - C n^{-2} \notag\\
& =  1 - n \PP (  w_i y_i \leq \varepsilon_n n^{-1/2}\log n ) - C n^{-2}.
\label{KPCA-theorem-CSBM-eaxct-2}
\end{align}
where we used (\ref{KPCA-theorem-CSBM-eaxct-1}), union bounds and symmetry.

Take any $0 < \varepsilon < \frac{a-b}{2} \log (a/b) + 2c$. By Lemma \ref{KPCA-lemma-CSBM-LDP}, for any $\delta > 0$ there exists a large $N$ such that when $n \geq N$, $\varepsilon_n < \varepsilon$ and
\begin{align*}
\PP ( w_i y_i \leq \varepsilon_n n^{-1/2} \log n) \leq n^{-  \sup_{ t \in \RR } \{ \varepsilon t + I(t,a,b,c) \} + \delta }.
\end{align*}
This and (\ref{KPCA-theorem-CSBM-eaxct-2}) lead to
\begin{align*}
\PP [ \cM(\sgn(\hat\bu) , \by ) = 0 ] \geq 1 - n^{1 - \sup_{ t \in \RR } \{ \varepsilon t + I(t,a,b,c) \} + \delta } - C n^{-2}, \qquad \forall n \geq N.
\end{align*}
When $I^*(a,b,c) = \sup_{ t \in \RR } I(t,a,b,c) > 1$, by choosing small $\varepsilon$ and $\delta$ we get $\PP [ \cM(\sgn(\hat\bu) , \by ) = 0 ] \to 1$.

The converse result for $I^*(a,b,c) = \sup_{ t \in \RR } I(t,a,b,c) < 1$ follows from the large deviation Lemma \ref{KPCA-lemma-CSBM-LDP} and the proof of Theorem 1 in \cite{ABH16}.

\subsection{Proof of Theorem \ref{KPCA-theorem-CSBM-error-rate}}\label{KPCA-theorem-CSBM-error-rate-proof}

Lemma \ref{KPCA-lemma-CSBM-inf} asserts the existence of some $\varepsilon_n \to 0$ and constant $C>0$ such that
\begin{align}
\PP ( \min_{c = \pm 1} \| c \hat\bu - \bw \|_{\infty} < \varepsilon_n n^{-1/2} \log n ) > 1 - C n^{-2}.
\label{KPCA-theorem-CSBM-error-rate-4}
\end{align}
Let $\hat{ c} =  \argmin_{c = \pm 1} \| c \hat\bu - \bw \|_{\infty}$ and $\bv = \hat{ c} \hat\bu$.

By definition, $\EE \cM ( \sgn(\hat\bu) , \by ) \leq \frac{1}{n} \sum_{i=1}^{n} \PP ( v_i y_i < 0)$.
By union bounds and (\ref{KPCA-theorem-CSBM-error-rate-4}),
\begin{align}
\PP ( v_i y_i < 0) & \leq \PP ( v_i y_i < 0, ~ \| \bv - \bw \|_{\infty} < \varepsilon_n n^{-1/2} \log n) + \PP (\| \bv - \bw \|_{\infty} \geq \varepsilon_n n^{-1/2} \log n) \notag \\
&\leq \PP ( w_i y_i < \varepsilon_n n^{-1/2} \log n) + C n^{-2}.
\label{KPCA-theorem-CSBM-error-rate-5}
\end{align}
Take any $0 < \varepsilon < \frac{a-b}{2} \log (a/b) + 2c$. By Lemma \ref{KPCA-lemma-CSBM-LDP}, for any $\delta > 0$ there exists a large $N$ such that when $n \geq N$, $\varepsilon_n < \varepsilon$ and
\begin{align}
\PP ( w_i y_i < \varepsilon_n n^{-1/2} \log n) \leq n^{-  \sup_{ t \in \RR } \{ \varepsilon t + I(t,a,b,c) \} + \delta }.
\label{KPCA-theorem-CSBM-error-rate-6}
\end{align}
From (\ref{KPCA-theorem-CSBM-error-rate-5}) and (\ref{KPCA-theorem-CSBM-error-rate-6}) we obtain that
\begin{align*}
\EE\cM( \sgn(\hat\bu) , \by ) \leq n^{-  \sup_{ t \in \RR } \{ \varepsilon t + I(t,a,b,c) \} + \delta } + C n^{-2}, \qquad \forall n \geq N.
\end{align*}
The proof is completed using $I^*(a,b,c) = \sup_{ t \in \RR } I(t,a,b,c) \leq 1 $ and letting $\varepsilon$, $\delta$ go to zero.

\subsection{Proof of Theorem \ref{KPCA-thm-CSBM-lower}}\label{KPCA-thm-CSBM-lower-proof}
Define $f(\cdot | \tilde\bA, \tilde\bX, \tilde{\by}_{-i} ) = \PP ( y_i =  \cdot | \bA = \tilde\bA, \bX = \tilde\bX, \by_{-i} = \tilde{\by}_{-i} )$. By Lemma \ref{KPCA-lemma-fundamental} and symmetries, for any estimator $\hat\by$ we have
\begin{align*}
\EE \cM ( \hat{ \by } , \by ) & \geq \frac{n-1}{3n - 1} \PP [ f(y_1 | \bA, \bX, \by_{-1} ) < f( -y_1 | \bA, \bX, \by_{-1} ) ] .
\end{align*}
Denote by $\cA$ the event on the right hand side. Let
\begin{align*}
&\cB_{\varepsilon} = \bigg\{ 
\bigg| \log \bigg( \frac{ f(y_1 | \bA, \bX, \by_{-1} ) }{ f(-y_1 | \bA, \bX, \by_{-1} ) } \bigg) -
\bigg( \log (a/b) (\bA \by)_1  + \frac{2R^2}{n R^2 + d } (\bG \by)_1 \bigg) y_1
\bigg| 
< \varepsilon q_n
\bigg\}
\\
& \cC_{\varepsilon} = \bigg\{ 
\bigg( \log (a/b) (\bA \by)_1  + \frac{2R^2}{n R^2 + d } (\bG \by)_1 \bigg) y_1 \leq - \varepsilon q_n
\bigg\}
\end{align*}
By the triangle's inequality, $\cC_{\varepsilon} \cap \cB_{\varepsilon} \subseteq \cA$. Hence
\begin{align}
\EE \cM ( \hat{ \by } , \by ) \gtrsim \PP( \cA) \geq \PP (\cC_{\varepsilon} \cap \cB_{\varepsilon}) \geq \PP (\cC_{\varepsilon}) - \PP(\cB_{\varepsilon}^c).
\label{KPCA-thm-CSBM-lower-1}
\end{align}

Since $\frac{a-b}{2} \log (a/b) + 2 c > 0$, Lemma \ref{KPCA-lemma-CSBM-LDP} asserts that
\begin{align*}
\lim_{n \to \infty} q_n^{-1} \log \PP ( \cC_{\varepsilon} )  = - \sup_{ t \in \RR } \{ -\varepsilon t + I(t, a, b, c) \}.
\end{align*}
By Lemma \ref{KPCA-lemma-CSBM-linearization} and the property of $o_{\PP}(\cdot;~\cdot)$,
\begin{align*}
\lim_{n \to \infty} q_n^{-1} \log \PP ( \cB_{\varepsilon}^c ) = - \infty.
\end{align*}
These limits and (\ref{KPCA-thm-CSBM-lower-1}) lead to
\begin{align*}
\liminf_{n \to \infty} q_n^{-1} \log \EE \cM ( \hat{ \by } , \by ) \geq - \sup_{ t \in \RR } \{ - \varepsilon t + I(t, a, b, c) \}.
\end{align*}
Taking $\varepsilon \to 0$ finishes the proof.

\subsection{Proof of \Cref{thm-CSBM-q}}\label{sec-proof-thm-CSBM-q}

We need the following lemma, whose proof can be found in \Cref{sec-KPCA-lemma-CSBM-q-proof}.

\begin{lemma}\label{KPCA-lemma-CSBM-q}
	Let Assumption \ref{KPCA-assumption-CSBM-q} hold. Define $\bar\bu = \by / \sqrt{n}$ and
\begin{align}
	\bw = \log(a/b) \bA \bar\bu + \frac{2 R^2}{n R^2 + d } \bG \bar\bu.
\label{eqn-KPCA-lemma-CSBM-q}
\end{align}
	For $\tilde\bu$ defined by \eqref{eqn-CSBM-aggregation-2}, we have
	\[
	 \min_{c = \pm 1}  \| c \tilde\bu - \bw \|_{q_n} = o_{\PP} (n^{-1/2 + 1/q_n} q_n ;~ q_n).
	\]
\end{lemma}

Let $\bw$ be the vector defined in (\ref{eqn-KPCA-lemma-CSBM-q}). We invoke \Cref{KPCA-lemma-GMM-error-rate} to control the misclassification error. Set the quantities $\bv$, $\bw $, $\bar\bv$ and $p$ therein be our $\tilde{\bu}$, $\bw - q_n n^{-1/2} \by$, $q_n n^{-1/2} \by$ and $q_n$. We have $\delta_n = q_n n^{-1/2}$. According to \Cref{KPCA-lemma-CSBM-q}, the assumption in \Cref{KPCA-lemma-GMM-error-rate} holds. Then
	\begin{align*}
& \limsup_{n\to\infty} q_n^{-1} \log 
\EE \cM ( \sgn(\tilde \bu ) , \by )
\\
& = \limsup_{n\to\infty} q_n^{-1} \log \bigg(
\frac{1}{n}
\EE \min_{c = \pm 1} \left|  \{ i \in [n]:~ c \sgn(\tilde u_i) \neq y_i \} \right| 
\bigg)
\\ &
\leq
\limsup_{\varepsilon \to 0}
\limsup_{n\to\infty} q_n^{-1} \log \bigg( \frac{1}{n} \sum_{i=1}^{n}  \PP  \left(
- ( w_i  - q_n n^{-1/2} y_i )
 y_i \geq  (1-\varepsilon) q_n n^{-1/2}
\right)
\bigg) 
\\ &
\leq
\limsup_{\varepsilon \to 0}
\limsup_{n\to\infty} q_n^{-1} \log  \PP   (
   w_i  y_i \leq  \varepsilon q_n n^{-1/2}
 ).
\end{align*}
By \Cref{KPCA-lemma-CSBM-LDP},
	\begin{align*}
\lim_{n \to \infty} q_n^{-1} \log \PP ( w_i y_i \leq \varepsilon q_n n^{-1/2})  = - \sup_{ t \in \RR } \{ \varepsilon t + I(t, a, b, c) \},
\qquad \forall \varepsilon < \frac{a-b}{2} \log (a/b) + 2 c.
\end{align*}
Then, we immediately get
	\begin{align*}
\limsup_{n \to q_n} q_n^{-1} \log \EE \cM ( \sgn( \tilde\bu ) , \by ) \leq
 - \sup_{ t \in \RR } I(t, a, b, c)  =  - I^*( a, b, c ) .
\end{align*}
If $ I^*( a, b, c ) > 1$ and $q_n = \log n$, then
\[
\PP [\cM ( \sgn( \tilde\bu ) , \by ) > 0] = \PP [\cM ( \sgn( \tilde\bu ) , \by ) \geq 1/n] \leq \frac{\EE \cM ( \sgn( \tilde\bu ) , \by )}{ 1/n } = n^{1- I^*( a, b, c ) +o(1)} \to 0.
\]

\subsection{Proof of Lemma \ref{KPCA-lemma-CSBM-q}}\label{sec-KPCA-lemma-CSBM-q-proof}

	Define an auxiliary quantity
	\begin{align*}
	& \bv = \frac{1}{\sqrt{n}} \log \bigg( \frac{\alpha}{\beta} \bigg) \bA \hat\by_G   + \frac{2n R^4}{n R^2 + d } \bu_1(\bG).
	\end{align*}
	Then $\| \tilde\bu - \bw \|_{q_n} \leq \| \tilde\bu - \bv \|_{q_n}  + \| \bv - \bw \|_{q_n}$,
	\begin{align}
	\| \bv - \bw \|_{q_n} & \leq \log(a/b) \|  \bA (\hat\by_G - \by  )  \|_{q_n}  / \sqrt{n}  + \frac{2 R^2}{n R^2 + d } \| (nR^2) \bu_1(\bG) - \bG \bar\bu  \|_{q_n},
	\label{KPCA-theorem-CSBM-error-rate-0-q} \\
	\| \tilde\bu - \bv \|_{q_n} & \leq \frac{1}{\sqrt{n}} \bigg|
	\log \bigg( \frac{
		\bm{1}^{\top} \bA \bm{1} + \hat{\by}_G^{\top} \bA \hat{\by}_G
	}{
		\bm{1}^{\top} \bA \bm{1} - \hat{\by}_G^{\top} \bA \hat{\by}_G
	} \bigg) -  \log \bigg( \frac{\alpha}{\beta} \bigg)
	\bigg| \| \bA \hat\by_G \|_{q_n} \notag\\
	& + \bigg|
	\frac{2 \lambda_1^2(\bG) }{n \lambda_1(\bG) + n d } - \frac{2n R^4}{n R^2 + d }
	\bigg| \| \bu_1(\bG) \|_{q_n} . \label{KPCA-theorem-CSBM-error-rate-1-q}
	\end{align}
	
	For simplicity, suppose that $\langle \bu_1(\bG) , \bar\bu \rangle \geq 0 $. By Lemma \ref{KPCA-lemma-L2} and Theorem \ref{KPCA-corollary-main-inf}, we have
	\begin{align*}
	& | \lambda_1(\bG) - nR^2 | = o_{\PP} (1;~ n), \\
	& \| \bu_1(\bG) - \bG \bar\bu / (n R^2) \|_{q_n} = o_{\PP} ( n^{-1/2+1/q_n} ;~q_n ), \\
	& \| \bu_1(\bG) \|_{q_n} = O_{\PP} ( n^{-1/2+1/q_n} ;~q_n ).
	\end{align*}
	Hence,
	\begin{align*}
&	 \frac{2 R^2}{n R^2 + d } \| (nR^2) \bu_1(\bG) - \bG \bar\bu  \|_{q_n}  = o_{\PP} (n^{-1/2+1/q_n} q_n ;~ q_n) , \\
& \bigg|
	\frac{2 \lambda_1^2(\bG) }{n \lambda_1(\bG) + n d } - \frac{2n R^4}{n R^2 + d }
	\bigg| \| \bu_1(\bG) \|_{q_n} = o_{\PP} (n^{-1/2+1/q_n} q_n ;~ q_n) . 
	\end{align*}
	According to (\ref{KPCA-theorem-CSBM-error-rate-0-q}) and (\ref{KPCA-theorem-CSBM-error-rate-1-q}), it remains to show that
		\begin{align*}
&  \log(a/b) \|  \bA (\hat\by_G - \by  )  \|_{q_n}  / \sqrt{n} = o_{\PP} (n^{-1/2+1/q_n} q_n ;~ q_n) ,
\\
& \frac{1}{\sqrt{n}} \bigg|
	\log \bigg( \frac{
		\bm{1}^{\top} \bA \bm{1} + \hat{\by}_G^{\top} \bA \hat{\by}_G
	}{
		\bm{1}^{\top} \bA \bm{1} - \hat{\by}_G^{\top} \bA \hat{\by}_G
	} \bigg) -  \log \bigg( \frac{\alpha}{\beta} \bigg)
	\bigg| \| \bA \hat\by_G \|_{q_n} = o_{\PP} (n^{-1/2+1/q_n} q_n ;~ q_n).
	\end{align*}
They are immediately implied by the followings:
		\begin{align}
&  \|  \bA (\hat\by_G - \by  )  \|_{q_n}   = o_{\PP} (n^{1/q_n} q_n ;~ q_n) ,
	\label{KPCA-theorem-CSBM-error-rate-3-q} \\
&\bigg|  \bm{1}^{\top} \bA \bm{1} / n - \frac{(a+b) q_n}{2} \bigg| = o_{\PP} (q_n;~q_n),
	\label{KPCA-theorem-CSBM-error-rate-4-q}\\
&\bigg|   \hat{\by}_G^{\top} \bA \hat{\by}_G / n - \frac{(a-b) q_n}{2} \bigg| = o_{\PP} (q_n;~q_n),
	\label{KPCA-theorem-CSBM-error-rate-5-q} \\
&\| \bA \hat\by_G \|_{q_n} = O_{\PP} (n^{1/q_n} q_n ;~ q_n) . 	\label{KPCA-theorem-CSBM-error-rate-6-q}
	\end{align}
We will tackle them one by one.

\noindent{\bf Proof of (\ref{KPCA-theorem-CSBM-error-rate-3-q}).} Let $S = \{ i:~ (\hat\by_G)_i \neq y_i \} =  \{ i:~ \sgn[ \bu_1(\bG) ]_i \neq \sgn(y_i) \}$. Then
\begin{align*}
|S| \leq \Big| \{ i:~ | [ \bu_1(\bG) -  \bar\bu ]_i | \geq 1/\sqrt{n} \} \Big|
\leq 
\frac{ \| \bu_1(\bG) -  \bar\bu\|_2^2  }{ (1/\sqrt{n})^2 } = O_{\PP} ( n / q_n ;~ n ).
\end{align*}
where we used the assumption $\langle \bu_1(\bG) , \bar\bu \rangle \geq 0 $ and \Cref{KPCA-lemma-L2}. Hence, there exists a constant $C$ such that $\PP ( |S| \geq C n / q_n ) \leq e^{-n}$ for large $n$. Thanks to Fact \ref{KPCA-fact-truncation}, it suffices to show that
\begin{align*}
\|  \bA (\hat\by_G - \by  )  \|_{q_n} \bm{1}_{
\{
|S| < C n / q_n
\}
}  = o_{\PP} (n^{1/q_n} q_n ;~ q_n)
\end{align*}
By Fact \ref{KPCA-fact-moment-tail}, this can be implied by
\begin{align*}
\EE^{1/q_n} \Big( \|  \bA (\hat\by_G - \by  )  \|_{q_n}^{q_n} \bm{1}_{
	\{
	|S| < C n / q_n
	\}
} \Big)  = o (n^{1/q_n} q_n ).
\end{align*}
Below we prove a stronger result
\begin{align}
\EE^{1/q_n} \Big( \|  \bA (\hat\by_G - \by  )  \|_{q_n}^{q_n}  \Big| 
	|S| < C n / q_n 
 \Big)  = o (n^{1/q_n} q_n ).
 \label{eqn-CSBM-moment}
\end{align}

From the fact
\[
| \bA_i (\hat\by_G - \by  ) | = \bigg| \sum_{j=1}^{n} A_{ij} (\hat\by_G - \by  )_j \bigg| \leq
2 \sum_{j \in S} A_{ij}
\]
we obtain that
\begin{align}
\EE \bigg(
\| \bA (\hat\by_G - \by  ) \|_{q_n}^{q_n}
\bigg|
|S| < C n / q_n
\bigg)
\leq 2^{q_n} \sum_{i=1}^{n} \EE \bigg[
\bigg( \sum_{j \in S} A_{ij} \bigg)^{q_n}
\bigg|
|S| < C n / q_n
\bigg].
\label{eqn-eqn-CSBM-moment-1}
\end{align}

Note that $\bA$ and $S$ are independent. By Corollary 3 in \cite{Lat97},
\begin{align*}
\EE^{1/q_n} \bigg[
\bigg( \sum_{j \in S} A_{ij} \bigg)^{q_n}
\bigg|
\by,~|S|
\bigg] \lesssim \frac{ q_n}{\log q_n} \sum_{j \in S} \EE (A_{ij} | \by)
\leq \frac{ q_n}{\log q_n} \cdot |S| \cdot \frac{(a+b) q_n}{n} .
\end{align*}
For any constant $C > 0$,
\begin{align*}
\EE^{1/q_n} \bigg[
\bigg( \sum_{j \in S} A_{ij} \bigg)^{q_n}
\bigg|
|S| \leq C n / q_n
\bigg]
\lesssim \frac{ q_n }{\log q_n}   .
\end{align*}
From this and (\ref{eqn-eqn-CSBM-moment-1}) we get
\begin{align}
\EE^{1/q_n} \bigg(
\| \bA (\hat\by_G - \by  ) \|_{q_n}^{q_n}
\bigg|
|S| < C n / q_n
\bigg)
\lesssim n^{1/q_n} \cdot \frac{ q_n }{\log q_n} = o (n^{1/q_n} q_n). 
\end{align}
and then derive (\ref{eqn-CSBM-moment}).

\noindent{\bf Proof of (\ref{KPCA-theorem-CSBM-error-rate-4-q}).} 
Let $\bar\bA = \EE (\bA | \by)$. On the one hand, we use Hoeffding's inequality \citep{Hoe63} to get
\begin{align*}
\PP \Big( |\bm{1}^{\top} \bA \bm{1} - \bm{1}^{\top} \bar\bA \bm{1} | /n \geq t \Big| \by \Big)
= \PP \bigg( \bigg| \sum_{1 \leq i,j \leq n} (A_{ij} - \bar{A}_{ij}) \bigg| \geq n t \bigg| \by \bigg) \leq 2 e^{-2 t^2} , \qquad\forall t \geq 0
\end{align*}
and thus
\begin{align}
 \bm{1}^{\top} \bA \bm{1} / n - \bm{1}^{\top} \bar\bA \bm{1}  /n  = O_{\PP} (\sqrt{q_n} ;~ q_n).
\label{eqn-CSBM-A-1}
\end{align}
On the other hand, we obtain from
\begin{align*}
\bar\bA = \frac{\alpha + \beta}{2} \bm{1} \bm{1}^{\top} + \frac{\alpha - \beta}{2} \by \by^{\top} .
\end{align*}
that $ \bm{1}^{\top} \bar\bA \bm{1} / n =  \frac{ a+b }{2} q_n  + \frac{a - b}{2} q_n (\by^{\top} \bm{1} )^2 / n^2$. Hoeffding's inequality yields
\begin{align*}
& \PP \bigg( \bigg| \bm{1}^{\top} \bar\bA \bm{1} /n - \frac{ a+b }{2} q_n \bigg| \geq t  \bigg)  = 
\PP \bigg( \frac{|a - b|}{2}  \cdot \frac{ q_n (\by^{\top} \bm{1} )^2 }{n^2} \geq t  \bigg) \\
&= \PP \bigg( \bigg|
\sum_{i=1}^{n} y_i
\bigg|  \geq n \sqrt{ \frac{2t}{|a - b| q_n} }  \bigg) \leq 2 \exp \bigg[
- \frac{2}{n} 
 \bigg(n \sqrt{ \frac{2t}{|a - b| q_n} } \bigg)^2 
\bigg]
 = 2 \exp \bigg(
- \frac{4nt}{|a - b| q_n} 
\bigg)
\end{align*}
and 
\begin{align}
\bm{1}^{\top} \bar\bA \bm{1} /n - \frac{ a+b }{2} q_n = O_{\PP} (q_n^2/n;~q_n) .
\label{eqn-CSBM-A-2}
\end{align}
The desired bound (\ref{KPCA-theorem-CSBM-error-rate-4-q}) follows from (\ref{eqn-CSBM-A-1}) and (\ref{eqn-CSBM-A-2}).

\noindent{\bf Proof of (\ref{KPCA-theorem-CSBM-error-rate-5-q}).}
Note that
\begin{align*}
& \frac{\hat\by_G^{\top} \bA \hat\by_G }{n} - \frac{a-b}{2} q_n 
 = \frac{ \hat\by_G^{\top} \bA \hat\by_G
- \by^{\top} \bA \by }{n} + \frac{ \by^{\top} \bA \by - \by^{\top} \bar\bA \by }{n} + \bigg(  \frac{ \by^{\top} \bar\bA \by }{n}  - \frac{a-b}{2} q_n \bigg) .
\end{align*}
Similar to (\ref{eqn-CSBM-A-1}) and (\ref{eqn-CSBM-A-2}), it is easy to show that
\begin{align*}
&  ( \by^{\top} \bA \by - \by^{\top} \bar\bA \by ) / n = o_{\PP} ( q_n;~q_n )
\qquad\text{and}\qquad
 \by^{\top} \bar\bA \by / n - \frac{a-b}{2} q_n = o_{\PP} ( q_n;~q_n ).
\end{align*}
By direct calculation and (\ref{KPCA-theorem-CSBM-error-rate-3-q}),
\begin{align*}
& (  \hat\by_G^{\top} \bA \hat\by_G
- \by^{\top} \bA \by )/n  = n^{-1}  ( \hat\by_G + \by )^{\top} \bA ( \hat\by_G - \by )
\leq n^{-1} ( \| \hat\by_G \|_2 + \| \by \|_2 ) \| \bA ( \hat\by_G - \by ) \|_2 \notag\\
& \leq n^{-1} \cdot 2 \sqrt{n} \cdot n^{1/2-1/q_n} \| \bA ( \hat\by_G - \by ) \|_{q_n}
= 2 n^{-1/q_n} o_{\PP} (n^{1/q_n} q_n ;~ q_n) =
o_{\PP} ( q_n ;~ q_n)  .
\end{align*}
Then (\ref{KPCA-theorem-CSBM-error-rate-5-q}) becomes obvious.

\noindent{\bf Proof of (\ref{KPCA-theorem-CSBM-error-rate-6-q}).}
By Theorem 1 in \cite{Lat97} and Assumption \ref{KPCA-assumption-CSBM-q}, $\EE^{1/q_n} (  \sum_{j=1}^{n} A_{ij}  )^{q_n} \lesssim q_n$. Hence
\begin{align*}
\EE^{1/q_n} \| \bA \by \|_{q_n}^{q_n}
& = \EE^{1/q_n} \bigg( \sum_{i=1}^{n} |\bA_i \by|^{q_n} \bigg)  \leq 
\EE^{1/q_n} \bigg( \sum_{i=1}^{n} \bigg| \sum_{j=1}^{n} A_{ij}  \bigg|^{q_n} \bigg) 
\lesssim n^{1/q_n} q_n.
\end{align*}
Fact \ref{KPCA-fact-moment-tail} leads to $\| \bA \by \|_{q_n} = O_{\PP} (  n^{1/q_n} q_n ;~ q_n )$. Then, (\ref{KPCA-theorem-CSBM-error-rate-6-q}) follows from the above and (\ref{KPCA-theorem-CSBM-error-rate-3-q}).

\section{Proofs of Section \ref{KPCA-sec-outlines}}

\subsection{Proof of Lemma \ref{thm-hollowing}}\label{thm-hollowing-proof}

Note that $s = 0$, $r = 1$, $\bar\Delta = \bar\lambda = n \| \bmu \|_2^2$ and $\kappa = 1$. Assumption \ref{KPCA-assumption-spectral} holds if $1 / \sqrt{n} \leq \gamma \ll 1$.
Assumption \ref{KPCA-assumption-noise} holds with $\bSigma = 2 \bI_d$ and in that case, Assumption \ref{KPCA-assumption-concentration} holds with
\begin{align*}
\gamma \geq 2 \max \bigg\{ 
\frac{1}{\| \bmu \|_2} ,~ \frac{\sqrt{d/n}}{ \| \bmu \|_2^2 }
\bigg\}.
\end{align*}
The right hand side goes to zero as $d / n \to \infty$ and $(n/d)^{1/4} \| \bmu \|_2  \to \infty$. Hence we can take
\begin{align*}
\gamma = 2 \max \bigg\{
\frac{1}{\sqrt{n}} , ~
\frac{1}{\| \bmu \|_2} ,~ \frac{\sqrt{d/n}}{ \| \bmu \|_2^2 }
\bigg \}
\end{align*}
to satisfy all the assumptions above. Then Lemma \ref{KPCA-lemma-L2} yields $| \langle \bu, \bar\bu \rangle | \overset{\PP}{\to} 1$.

To study $\hat\bu$, we first define $\tilde\bG = \EE (\bX \bX^{\top}) = d \bI_n + d \be_1 \be_1^{\top}$. Hence its leading eigenvector and the associated eigengap are $\tilde\bu = \be_1$ and $\tilde\Delta = d$. Observe that $\bG = \cH ( \bX \bX^{\top} )$ and
\begin{align}
& \| \bX \bX^{\top} - \tilde\bG \|_2 \leq \| \cH ( \bX \bX^{\top}  - \tilde\bG ) \|_2 + \max_{i\in[n]}
\left|
( \bX \bX^{\top} - \tilde\bG )_{ii}
\right| \notag \\
& \leq \| \cH( \bX \bX^{\top} ) - \bar\bG \|_2 + \| \bar\bG - \cH ( \tilde\bG ) \|_2 + \max_{i\in[n]} 
\left|
\| \bx_i \|_2^2 - \EE \| \bx_i \|_2^2
\right|
\label{thm-hollowing-1}
\end{align}
By Lemma \ref{KPCA-lemma-L2},
\begin{align}
\| \cH( \bX \bX^{\top} ) - \bar\bG \|_2 = o_{\PP} (\bar\Delta;~n) = o_{\PP} ( n \| \bmu \|_2^2 ;~n).
\label{thm-hollowing-2}
\end{align}
When $i \neq j$,
\begin{align*}
& \tilde\bG_{ij} = \EE \langle \bx_i , \bx_j \rangle = \EE \langle \bar\bx_i + \bz_i,  \bar\bx_j + \bz_j \rangle
= \EE  \langle \bar\bx_i ,  \bar\bx_j  \rangle = \bar\bG_{ij}.
\end{align*}
Hence $\cH( \bar\bG ) = \cH ( \tilde\bG )$, and
\begin{align}
\| \bar\bG - \cH ( \tilde\bG ) \|_2 = \max_{i \in [n]} |\bar{G}_{ii}| = \max_{i \in [n]} \| \bar\bx_i \|_2^2 = \| \bmu \|_2^2.
\label{thm-hollowing-3}
\end{align}
For the last term in (\ref{thm-hollowing}), we have
\begin{align*}
\| \bx_i \|_2^2 - \EE \| \bx_i \|_2^2 = \| \bar\bx_i + \bz_i \|_2^2 - ( \| \bar\bx_i \|_2^2 + \EE \| \bz_i \|_2^2 ) = 2 \langle \bar\bx_i , \bz_i \rangle + (\| \bz_i \|_2^2 - \EE \| \bz_i \|_2^2).
\end{align*}
From $\| \langle \bar\bx_i , \bz_i \rangle \|_{\psi_2} \lesssim \| \bar\bx_i \|_2 = \| \bmu \|_2$, Fact \ref{KPCA-fact-p-log} and Lemma \ref{KPCA-lemma-Lp-gaussian} we obtain that
\begin{align}
\max_{i \in [n]} | \langle \bar\bx_i , \bz_i \rangle| \lesssim \|
( \langle \bar\bx_1 , \bz_1 \rangle , \cdots, \langle \bar\bx_n , \bz_n \rangle )
\|_{\log n} = O_{\PP} ( \sqrt{\log n} \| \bmu \|_2 ;~ \log n )
\label{thm-hollowing-4}
\end{align}
For any $i \geq 2$, $\| \bx_i \|_2^2 \sim \chi^2_d$. Lemma \ref{KPCA-lemma-chi-square} forces
\begin{align*}
\PP( |  \| \bx_i \|_2^2  - d | \geq 2\sqrt{dt} + 2t ) \leq 2e^{-t},\qquad  \forall t \geq 0, ~~ i \geq 2.
\end{align*}
By the $\chi^2$-concentration above and union bounds, $\max_{2 \leq i \leq n} |  \| \bx_i \|_2^2  - \EE \| \bx_i \|_2^2 | = O_{\PP} ( \sqrt{dn} \vee n ;~ n ) = O_{\PP} ( \sqrt{dn} ;~ n )$. Since $\| \bx_1 \|_2^2 / 2 \sim \chi^2_d$, we get $\max_{i \in [n]} |  \| \bx_i \|_2^2  - \EE \| \bx_i \|_2^2 | = O_{\PP} ( \sqrt{dn} ;~ n )$.

Plugging this and (\ref{thm-hollowing-2}), (\ref{thm-hollowing-3}), (\ref{thm-hollowing-4}) into (\ref{thm-hollowing-1}), we get
\begin{align*}
\| \bX \bX^{\top} - \tilde\bG \|_2 
= O_{\PP} ( n \| \bmu \|_2^2 + \| \bmu \|_2^2 + \sqrt{\log n} \| \bmu \|_2 + \sqrt{dn} ;~ \log n )
= O_{\PP} ( n \| \bmu \|_2^2  ;~ \log n ).
\end{align*}
Here we used $\| \bmu \|_2 \gg (d/n)^{1/4} \gg 1$. The Davis-Kahan Theorem \citep{DKa70} then yields
\begin{align*}
\min_{c = \pm 1} \| s \hat\bu - \tilde\bu \|_2 \lesssim \| \bX \bX^{\top} - \tilde\bG \|_2 / \tilde\Delta
= O_{\PP} ( n \| \bmu \|_2^2  ;~ \log n ) / d
= o_{\PP} (1;~ \log n),
\end{align*}
since $\| \bmu \|_2 \ll  \sqrt{d/n}$. From $\tilde{\bu} = \be_1$ and $\langle \tilde{\bu} , \bar\bu \rangle = 1/\sqrt{n} \to 0$ we get $|\langle \hat{\bu} , \bar\bu \rangle | \overset{\PP}{\to} 0$.

\subsection{Proof of Lemma \ref{lemma-hollowing}}\label{lemma-hollowing-proof}

Lemma \ref{lemma-hollowing} directly follows from Lemma \ref{KPCA-lemma-L2} and thus we omit its proof.

\section{Technical lemmas}

\subsection{Lemmas for probabilistic analysis}

\begin{lemma}\label{KPCA-lem-concentration-gram}
Under Assumption \ref{KPCA-assumption-noise}, we have
\begin{align*}
&\| \cH ( \bZ \bZ^{\top} ) \|_2 = O_{\PP} \left(
\max\{ \sqrt{n} \| \bSigma \|_{\mathrm{HS}},~n \| \bSigma \|_{\mathrm{op}} \};~ n \right),\\
&\max_{i\in[n]} \| \bz_i \|^2 = O_{\PP} \left(
\max\{ \Tr(\bSigma),~n \| \bSigma \|_{\mathrm{op}} \};~n
\right),\\
& \| \bZ  \bZ^{\top} \|_2 = O_{\PP} \left(
\max\{ \Tr(\bSigma),~n \| \bSigma \|_{\mathrm{op}} \};~n
\right).
\end{align*}
\end{lemma}

\begin{proof}[\bf Proof of Lemma \ref{KPCA-lem-concentration-gram}]
By definition,
\begin{align*}
\| \cH ( \bZ  \bZ^{\top} ) \|_2 = \sup_{ \bu \in  \SSS^{n-1} } | \bu^{\top} \cH ( \bZ \bZ^{\top} ) \bu |
= \sup_{ \bu \in  \SSS^{n-1} } \bigg|
\sum_{i\neq j} u_i u_j \langle \bz_i, \bz_j \rangle \bigg|.
\end{align*}

Fix $\bu \in \SSS^{n-1}$, let $\bA= \bu \bu^{\top}$ and $S = \sum_{i\neq j} u_i u_j \langle \bz_i, \bz_j \rangle$. By Proposition 2.5 in \cite{CYa18}, there exists an absolute constant $C>0$ such that
\begin{align*}
\PP ( S \geq t ) \leq \exp \left(
-C \min \left\{
\frac{t^2}{\| \bSigma \|_{\mathrm{HS}}^2 },~ \frac{t}{\| \bSigma \|_{\mathrm{op}} }
\right\}
\right),\qquad \forall t > 0.
\end{align*}
When $t = \lambda \max\{ \sqrt{n} \| \bSigma \|_{\mathrm{HS}},~n \| \bSigma \|_{\mathrm{op}} \}$ for some $\lambda \geq 1$, we have $\min \{
t^2 / \| \bSigma \|_{\mathrm{HS}}^2,~ t / \| \bSigma \|_{\mathrm{op}}  \} \geq \lambda n$ and
$\PP ( S \geq t ) \leq e^{ -C \lambda n }$. Similarly, we get $\PP ( S \leq - t ) \leq e^{ -C \lambda n }$ and thus
\begin{align*}
\PP \bigg( 
\bigg|
\sum_{i\neq j} u_i u_j \langle \bz_i, \bz_j \rangle
\bigg|
\geq 
\lambda \max\{ \sqrt{n} \| \bSigma \|_{\mathrm{HS}},~n \| \bSigma \|_{\mathrm{op}} \}
 \bigg) \leq 2 e^{ -C \lambda n },
 \qquad \forall \lambda \geq 1.
\end{align*}
The bound on $\| \cH(\bZ \bZ^{\top}) \|_2$ then follows from a standard covering argument \citep[Section 5.2.2]{Ver10}.

Theorem 2.6 in \cite{CYa18} with $n = 1$ and $A = 1$ implies the existence of constants $C_1$ and $C_2$ such that for any $t \geq 0$,
\begin{align*}
\PP ( \| \bz_i \|^2 
\geq C_1 \Tr(\bSigma) + t ) 
\leq \exp\left(
- C_2 \min\left\{ \frac{t^2}{\| \bSigma \|_{\mathrm{HS}}^2}
,~\frac{t}{\| \bSigma \|_{\mathrm{op}}}
\right\}
\right).
\end{align*}
When $t = \lambda \max\{ \sqrt{n} \| \bSigma \|_{\mathrm{HS}},~n \| \bSigma \|_{\mathrm{op}} \}$ for some $\lambda \geq 1$, we have $\min \{
t^2 / \| \bSigma \|_{\mathrm{F}}^2,~ t / \| \bSigma \|_{\mathrm{op}}  \} \geq  \lambda n$. Hence
\begin{align*}
\PP ( \| \bz_i \|^2 
\geq C_1 \Tr(\bSigma) + \lambda \max\{ \sqrt{n} \| \bSigma \|_{\mathrm{HS}},~n \| \bSigma \|_{\mathrm{op}} \} ) 
\leq e^{- C_2 \lambda n},\qquad\forall \lambda \geq 1.
\end{align*}
Union bounds force
\begin{align*}
\max_{i\in[n]} \| \bz_i \|^2 = O_{\PP} \left(
\max\{ \Tr(\bSigma),~ \sqrt{n} \| \bSigma \|_{\mathrm{HS}},~n \| \bSigma \|_{\mathrm{op}} \};~n
\right).
\end{align*}
We can neglect the term $\sqrt{n} \| \bSigma \|_{\mathrm{HS}}$ above, since
\begin{align*}
\sqrt{n} \| \bSigma \|_{\mathrm{F}}  = \sqrt{ n  \| \bSigma \|_{\mathrm{F}}^2 }
\leq \sqrt{ ( n  \| \bSigma \|_{\mathrm{op}} ) \Tr(\bSigma) }
\leq \max\{ \Tr(\bSigma),~n \| \bSigma \|_{\mathrm{op}} \}.
\end{align*}
Finally, the bound on $\| \bZ \bZ^{\top} \|_2$ follows from $\| \bZ \bZ^{\top} \|_2 \leq  \| \cH(\bZ \bZ^{\top}) \|_2 + \max_{i\in[n]} \| \bz_i \|^2$.
\end{proof}

\begin{lemma}\label{KPCA-lemma-Z-product}
Let Assumption \ref{KPCA-assumption-noise} hold, $p \geq 2$ and $\{ \bV^{(m)} \}_{m=1}^n \subseteq \R^{n\times K}$ be random matrices such that $\bV^{(m)}$ is independent of $\bz_m$. Then,
\[
\bigg(
\sum_{m=1}^{n} \bigg\| \sum_{j\neq m}  \langle \bz_m,\bz_j \rangle \bV^{(m)}_j \bigg\|_2^p 
\bigg)^{1/p} = n^{1/p} \sqrt{Kp} \max_{m\in[n]} \| \bV^{(m)} \|_2 O_{\PP} \left(
 \max\{ \| \bSigma \|_{\mathrm{HS}},~\sqrt{n} \| \bSigma \|_{\mathrm{op}} \}
;~p \wedge n \right).
\]
\end{lemma}

\begin{proof}[\bf Proof of Lemma \ref{KPCA-lemma-Z-product}]
By Minkowski's inequality,
\begin{align*}
\bigg\| \sum_{j\neq m}  \langle \bz_m,\bz_j \rangle \bV^{(m)}_j \bigg\|_2^p 
& = \bigg( \sum_{k=1}^{K} \bigg| \sum_{j \neq m}  \langle \bz_m,\bz_j \rangle V^{(m)}_{jk} \bigg|^2  \bigg)^{p/2} \\
& \leq \bigg[ \bigg( \sum_{k=1}^{K} \bigg| \sum_{j \neq m}  \langle \bz_m,\bz_j \rangle V^{(m)}_{jk} \bigg|^p  \bigg)^{2/p} K^{1-2/p} \bigg]^{p/2} \\
& = K^{p/2 - 1} \sum_{k=1}^{K} \bigg| \sum_{j \neq m}  \langle \bz_m,\bz_j \rangle V^{(m)}_{jk} \bigg|^p
= K^{p/2 - 1} \sum_{k=1}^{K} |  \langle \bz_m,\bw_k^{(m)} \rangle |^p,
\end{align*}
where we define $\bw_k^{(m)} = \sum_{j\neq m} V^{(m)}_{jk} \bz_j  = \bZ^{\top} (\bI - \be_m \be_m^{\top}) \bv_k^{(m)}$, $\forall m \in [n]$, $k\in[K]$.
Observe that
\begin{align*}
\| \bSigma^{1/2} \bw_k^{(m)} \|^2 & = ( \bv_k^{(m)})^{\top} (\bI - \be_m \be_m^{\top}) \bZ \bSigma \bZ^{\top} (\bI - \be_m \be_m^{\top}) \bv_k^{(m)} \\
& \leq \| \bv_k^{(m)} \|_2^2 \| \bZ \bSigma \bZ^{\top} \|_2
\leq \| \bV^{(m)} \|_2^2 \| \bZ \bSigma \bZ^{\top} \|_2.
\end{align*}
As a result,
\begin{align*}
& \bigg\| \sum_{j\neq m}  \langle \bz_m,\bz_j \rangle \bV^{(m)}_j \bigg\|_2^p 
\leq K^{p/2 - 1}
\bigg( 
 \sum_{k=1}^{K} |  \langle \bz_m,\bw_k^{(m)} / \| \bSigma^{1/2} \bw_k^{(m)} \| \rangle |^p
 \bigg)
 \Big( 
\max_{m\in[n]} \| \bV^{(m)} \|_2 \cdot  \| \bZ \bSigma \bZ^{\top} \|_2^{1/2} \Big)^p .
\end{align*}
and
\begin{align}
\bigg(  \sum_{m=1}^{n} \bigg\| \sum_{j\neq m}  \langle \bz_m,\bz_j \rangle \bV^{(m)}_j \bigg\|_2^p \bigg)^{1/p}
& \leq \sqrt{K  \| \bZ \bSigma \bZ^{\top} \|_2 } \max_{m\in[n]} \| \bV^{(m)} \|_2 
\notag \\ & 
\cdot \bigg( K^{-1}
\sum_{m=1}^{n} \sum_{k=1}^{K} |  \langle \bz_m,\bw_k^{(m)} / \| \bSigma^{1/2} \bw_k^{(m)} \| \rangle |^p
\bigg)^{1/p}.
\label{KPCA-lemma-Z-product-1}
\end{align}

On the one hand, let $\tilde\bz_i = \bSigma^{1/2} \bz_i$, $\forall i \in [n]$ and $\tilde{\bZ} = (\tilde\bz_1,\cdots,\tilde{\bz}_n)^{\top}$. Note that $\{ \tilde\bz_i \}_{i=1}^n$ satisfy Assumption \ref{KPCA-assumption-noise} with $\bSigma$ replaced by $\bSigma^2$, because
\begin{align*}
\EE e^{ \langle \bu,\tilde\bz_i \rangle } = \EE e^{ \langle \bSigma^{1/2} \bu, \bz_i \rangle }
\leq e^{ \alpha^2 \langle \bSigma  \bSigma^{1/2} \bu,  \bSigma^{1/2} \bu\rangle }
= e^{ \alpha^2 \langle \bSigma^2 \bu, \bu\rangle },\qquad \forall \bu\in\HH,~~i \in [n].
\end{align*}
It is easily seen from $\bSigma \in \cT  (\HH) $ that $\bSigma^2 \in \cT  (\HH) $. Then Lemma \ref{KPCA-lem-concentration-gram} asserts that
\begin{align}
\| \bZ \bSigma \bZ^{\top} \|_2 & = \| \tilde\bZ \tilde\bZ^{\top} \|_2 = O_{\PP} \left(
\max\{ \Tr(\bSigma^2),~n \| \bSigma^2 \|_{\mathrm{op}} \};~n
\right) \notag \\
& = O_{\PP} \left(
\max\{ \| \bSigma \|_{\mathrm{HS}}^2,~n \| \bSigma \|_{\mathrm{op}}^2 \};~n
\right).
\label{KPCA-lemma-Z-product-3}
\end{align}

On the other hand, note that $\bz_m$ and $\bw_k^{(m)}$ are independent. According to Assumption \ref{KPCA-assumption-noise} on sub-Gaussianity of $\bz_m$, we have
\begin{align*}
& \EE \Big(  \langle \bz_m, \bw_k^{(m)} / \| \bSigma^{1/2} \bw_k^{(m)} \|  \rangle \Big| \bw_k^{(m)}  \Big)= 0, \\
& p^{-1/2} \EE^{1/p} \Big( | \langle \bz_m, \bw_k^{(m)} / \| \bSigma^{1/2} \bw_k^{(m)} \| \rangle |^p 
\Big| \bw_k^{(m)}
\Big) 
 \leq C
\end{align*}
for some absolute constant $C$. Then $ \EE | \langle \bz_m, \bw_k^{(m)} / \| \bSigma^{1/2} \bw_k^{(m)} \| \rangle |^p  \leq ( C \sqrt{p} )^p$.
We have
\begin{align*}
\sum_{m=1}^{n} \sum_{k=1}^{K} \EE |  \langle \bz_m,\bw_k^{(m)} / \| \bSigma^{1/2} \bw_k^{(m)} \| \rangle |^p
\leq n K ( C \sqrt{p} )^p = ( n^{1/p} K^{1/p} C \sqrt{p} )^p.
\end{align*}
By Fact \ref{KPCA-fact-moment-tail},
\begin{align}
\bigg(
\sum_{m=1}^{n} \sum_{k=1}^{K} |  \langle \bz_m,\bw_k^{(m)} / \| \bSigma^{1/2} \bw_k^{(m)} \| \rangle |^p
\bigg)^{1/p}
= O_{\PP} \left( 
n^{1/p} K^{1/p} C \sqrt{p} ;~p \right).
\label{KPCA-lemma-Z-product-2}
\end{align}
The final result follows from (\ref{KPCA-lemma-Z-product-1}), (\ref{KPCA-lemma-Z-product-3}) and (\ref{KPCA-lemma-Z-product-2}).
\end{proof}

\begin{lemma}\label{KPCA-lemma-Lp-gaussian}
Let $\bX \in \R^{n\times m}$ be a random matrix with sub-Gaussian entries, and define $\bM \in \R^{n\times m}$ through $M_{ij} = \| X_{ij} \|_{\psi_2}$. For any $p \geq q \geq 1$, we have $\| \bX \|_{q, p} = O_{\PP} ( \sqrt{p} \| \bM \|_{q, p} ;~p )$.
\end{lemma}
\begin{proof}[\bf Proof of Lemma \ref{KPCA-lemma-Lp-gaussian}]
By Minkowski's inequality,
\begin{align*}
\EE \| \bX \|_{q, p}^p = 
\sum_{i=1}^{n} \EE \bigg(  \sum_{j=1}^{n} | X_{ij} |^q \bigg)^{p/q}
\leq \sum_{i=1}^{n}
\bigg(
\sum_{j=1}^{n} \EE^{q/p} ( | X_{ij} |^q )^{p/q}
\bigg)^{p/q}
= \sum_{i=1}^{n}
\bigg(
\sum_{j=1}^{n} [ \EE^{1/p}  | X_{ij} |^p ]^q
\bigg)^{p/q}.
\end{align*}
Since $p^{-1/2} \EE^{1/p}  | X_{ij} |^p \leq \| X_{ij} \|_{\psi_2} = M_{ij}$, we have
\begin{align*}
\EE \| \bX \|_{q,p}^p \leq \sum_{i=1}^{n}
\bigg(
\sum_{j=1}^{n} (\sqrt{p} M_{ij} )^q
\bigg)^{p/q}
= p^{p/2} \sum_{i=1}^{n}\bigg(
\sum_{j=1}^{n} M_{ij}^q
\bigg)^{p/q}
= ( \sqrt{p} \| \bM \|_{q,p} )^p.
\end{align*}
By Fact \ref{KPCA-fact-moment-tail}, $\| \bX \|_{q,p} = O_{\PP} ( \sqrt{p} \| \bM \|_{q,p} ;~p )$.
\end{proof}

\begin{lemma}\label{KPCA-lemma-chi-square}
For independent random vectors $\bX \sim N(\bmu, \bI_d)$ and $\bY \sim N( \bnu, \bI_d )$, we have the followings:
\begin{enumerate}
	\item If $\bmu = \mathbf{0}$, then
	\begin{align*}
&\PP( |  \| \bX \|_2^2  - d | \geq 2\sqrt{dt} + 2t ) \leq 2e^{-t},\qquad  \forall t \geq 0, \\
& \log \EE e^{\alpha \| \bX \|_2^2 + \langle \bbeta, \bX \rangle } = - \frac{d}{2} \log ( 1 - 2 \alpha) + \frac{ \| \bbeta \|_2^2 }{ 2 (1 - 2 \alpha) } \qquad \forall \alpha < \frac{1}{2},~\bbeta \in \RR^d;
	\end{align*}
	\item For any $t \in (-1, 1)$,
	\begin{align*}
\log \EE e^{t \langle \bX, \bY \rangle } =  \frac{t^2}{2 (1-t^2)} ( \| \bmu \|_2^2 + \| \bnu \|_2^2 ) + \frac{t}{1 - t^2} \langle \bmu , \bnu \rangle - \frac{d}{2} \log (1-t^2).
	\end{align*}
\end{enumerate}
\end{lemma}
\begin{proof}[\bf Proof of Lemma \ref{KPCA-lemma-chi-square}]
When $\bmu = \mathbf{0}$, $\| \bX \|_2^2 \sim \chi^2_d$. The concentration inequality in the claim is standard, see Remark 2.11 in \cite{BLM13}. Note that $p(\bx) = (2\pi)^{-d/2} e^{ - \| \bx \|_2^2 / 2 }$ is the probability density function of $\bX$. With a new variable $\by = \sqrt{1 - 2 \alpha} \bx$, we have
\begin{align*}
& \alpha \| \bx \|_2^2 + \langle \bbeta, \bx \rangle - \frac{1}{2} \| \bx \|_2^2
=- \frac{ \| \by \|_2^2 }{2} + \langle \bbeta / \sqrt{1 - 2\alpha} , \by \rangle
= - \frac{1}{2}
\| \by - \bbeta / \sqrt{1-2 \alpha} \|_2^2  + \frac{ \| \bbeta \|_2^2 }{ 2 (1 - 2 \alpha) }
\end{align*}
and
\begin{align*}
& \EE e^{\alpha \| \bX \|_2^2 + \langle \bbeta, \bX \rangle } 
 = (2\pi)^{-d/2} \int_{\RR^d} \exp \bigg(
\alpha \| \bx \|_2^2 + \langle \bbeta, \bx \rangle - \frac{1}{2} \| \bx \|_2^2
\bigg) \rd \bx \\
& = (2\pi)^{-d/2} \int_{\RR^d} \exp \bigg(
- \frac{1}{2}
\| \by - \bbeta / \sqrt{1-2 \alpha} \|_2^2  + \frac{ \| \bbeta \|_2^2 }{ 2 (1 - 2 \alpha) }
\bigg) (1 - 2 \alpha)^{-d / 2} \rd \by  \\
&  = (1 - 2 \alpha)^{-d / 2} \exp \bigg( \frac{ \| \bbeta \|_2^2 }{ 2 (1 - 2 \alpha) } \bigg).
\end{align*}

Now we come to the second part. Given $\bY$, $\langle \bX, \bY \rangle \sim N( \langle \bmu , \bY \rangle, \| \bY\|_2^2 )$. Hence
$\EE ( e^{t \langle \bX, \bY \rangle } | \bY ) = e^{ \langle \bmu , \bY \rangle t + \| \bY\|_2^2 t^2 / 2 }$. Define $\bZ = \bY - \bnu$. From $\langle \bmu , \bY \rangle = \langle \bmu , \bnu \rangle + \langle \bmu , \bZ \rangle $ and $\| \bY\|_2^2 = \| \bnu \|_2^2 + 2 \langle \bnu, \bZ \rangle + \| \bZ \|_2^2$ we obtain that
\begin{align*}
 \log \EE e^{t \langle \bX, \bY \rangle } & =
\log \EE [ \EE ( e^{t \langle \bX, \bY \rangle } | \bY ) ]  \\
&= \log  \EE \exp \left[
(  \langle \bmu , \bnu \rangle + \langle \bmu , \bZ \rangle  ) t + ( \| \bnu \|_2^2 + 2 \langle \bnu, \bZ \rangle + \| \bZ \|_2^2 ) t^2 / 2
\right] \\
& =  \langle \bmu , \bnu \rangle  t + \| \bnu \|_2^2 t ^2 / 2 + \log \EE \exp \left(
\langle t \bmu + t^2 \bnu , \bZ \rangle + \| \bZ \|_2^2 t^2 / 2
\right) \\
&  =  \langle \bmu , \bnu \rangle  t + \| \bnu \|_2^2 t ^2 / 2 - \frac{d}{2} \log \bigg(
1 - 2 \cdot \frac{t^2}{2} 
\bigg) + \frac{ \| t \bmu + t^2 \bnu \|_2^2 }{ 2( 1 - 2 \cdot t^2 / 2 ) } \\
& =  \langle \bmu , \bnu \rangle  t + \frac{ \| \bnu \|_2^2 t ^2 }{2} - \frac{d}{2} \log (1-t^2) + \frac{ t^2 \| \bmu + t \bnu \|_2^2 }{ 2( 1 - t^2 ) } \\
& = \frac{t^2}{2 (1-t^2)} ( \| \bmu \|_2^2 + \| \bnu \|_2^2 ) + \frac{t}{1 - t^2} \langle \bmu , \bnu \rangle - \frac{d}{2} \log (1-t^2).
\end{align*}
\end{proof}

\begin{lemma}\label{KPCA-lemma-LDP}
	Let $\{ S_{n} \}_{n=1}^{\infty}$ be random variables such that $\Lambda_n (t) = \log \EE e^{ t S_{n} }$ exists for all $t \in [ -R_n, R_n ]$, where $\{ R_n \}_{n=1}^{\infty}$ is a positive sequence tending to infinity. Suppose there is a convex function $\Lambda: \RR \to \RR$ and a positive sequence $\{ a_n \}_{n=1}^{\infty}$ tending to infinity such that $ \lim_{n \to \infty} \Lambda_n (t) / a_n = \Lambda (t)$ for all $t \in \RR$. We have
	\begin{align*}
	\lim_{n\to\infty} a_n^{-1} \log \PP ( S_n \leq c a_n ) = - \sup_{ t \in \RR } \{ c t - \Lambda (t) \},\qquad \forall c < \Lambda' (0).
	\end{align*}
\end{lemma}
\begin{proof}[\bf Proof of Lemma \ref{KPCA-lemma-LDP}]
This result follows directly from the G{\"a}rtner-Ellis theorem \citep{Gar77,Ell84} for large deviation principles.
\end{proof}

%
%

\subsection{Other lemmas}

\begin{lemma}\label{KPCA-lemma-trig}
Let $x \in (0, \pi / 2)$, $\varepsilon \in (0,1)$ and $\delta = \frac{\varepsilon}{\pi} (  \frac{\pi}{2} - x )$. We have $\max_{|y| \leq 2 \delta}
|
\frac{\cos (x + y)}{\cos x} - 1
| \leq \varepsilon$. Moreover, if $x > 2 \delta$, then $\max_{|y| \leq \delta / 3 }
| \frac{ \sin^2 x }{ \sin^2 (x + y) } - 1 | \leq \frac{9}{16}$.
\end{lemma}

\begin{proof}[Proof of Lemma \ref{KPCA-lemma-trig}]
Recall the elementary identity $\cos ( x + y ) = \cos x \cos y - \sin x \sin y$. If $|y| \leq 2 \delta $, then
$| \sin y | \leq |y| \leq 2 \delta = \frac{2 \varepsilon}{\pi} ( \frac{\pi}{2} - x ) \leq \tan ( \frac{\pi}{2} - x )$ and
\begin{align*}
&\bigg|
\frac{ \cos (x+y) }{ \cos x }  - \cos y
\bigg|
\leq \frac{ \sin x  | \sin y | }{ \cos x }
= \frac{  | \sin y | }{ \tan (\frac{\pi}{2} - x ) }
\leq  \frac{ \frac{2 \varepsilon}{\pi} ( \frac{\pi}{2} - x ) }{ \tan (\frac{\pi}{2} - x ) }
\leq \frac{\varepsilon}{2 \pi}, \notag \\
&0 \leq 1 -  \cos y \leq \frac{ y^2 }{2} \leq \frac{ (2 \delta)^2 }{2}
= \frac{ [ \varepsilon ( 1 - 2 x / \pi  ) ]^2 }{ 2 } \leq \frac{\varepsilon^2}{2}.
\end{align*}
The result on $\max_{|y| \leq 2 \delta}
|
\frac{\cos (x + y)}{\cos x} - 1
|$ follows from the estimates above and $\frac{\varepsilon}{2 \pi} + \frac{\varepsilon^2}{2} = \frac{\varepsilon}{2} ( 1 / \pi + \varepsilon ) \leq \varepsilon$.

The identity $\sin ( x + y ) = \sin x \cos y + \cos x \sin y$ imply that if $ 2 \delta < x \leq \tan x$ and $|y| \leq \delta / 3$, then
\begin{align*}
&\bigg|
\frac{ \sin (x+y) }{ \sin x }  - \cos y
\bigg|
\leq \frac{ \cos x  | \sin y | }{ \sin x }
= \frac{  | \sin y | }{ \tan x }
\leq  \frac{ \delta / 3 }{ \tan x }
\leq \frac{ \delta / 3 }{x} \leq \frac{1}{6}, \notag \\
&0 \leq 1 -  \cos y \leq \frac{ y^2 }{2} \leq \frac{ (\delta / 3 )^2 }{2}
= \frac{ [ \frac{\varepsilon}{6} ( 1 - 2 x / \pi  ) ]^2 }{ 2 } \leq \frac{\varepsilon^2}{72} \leq \frac{1}{72}.
\end{align*}
Hence for $|y| \leq \delta / 3$, we have $ |
\frac{ \sin (x+y) }{ \sin x }  - 1
| \leq \frac{1}{6} + \frac{1}{72} = \frac{13}{72} < \frac{1}{5}$. Direct calculation yields $\frac{4}{5} \leq \frac{ \sin (x+y) }{ \sin x }  \leq \frac{6}{5}$, $\frac{25}{36} \leq \frac{ \sin^2 x }{ \sin^2 (x + y) }  \leq \frac{25}{16}$ and
$| \frac{ \sin^2 x }{ \sin^2 (x + y) } - 1 | \leq \frac{9}{16}$.
\end{proof}

\begin{lemma}\label{KPCA-lemma-integral}
For $t\geq 0$ and $s \geq 2$, define $P (t, s) = \int_{0}^{\pi} e^{ t \cos x } (\sin x)^{s - 2} \rd x$ and $a = (s-2) / t$. There exists a constant $c>0$ and a continuous, non-decreasing function $w:~[0, c] \mapsto [0,1)$ with $w(0) = 0$ such that when $\max\{ 1/t, s^2 / t^3 \} \leq c$,
\begin{align*}
\bigg| 
\frac{ \frac{\partial }{\partial t} [\log P(t, s)] }{
( \sqrt{a^2 + 4} - a ) / 2
} - 1
\bigg| \leq w ( \max\{ 1/t, s^2 / t^3 \} ).
\end{align*}
\end{lemma}

\begin{proof}[\bf Proof of Lemma \ref{KPCA-lemma-integral}]
It suffices to show that $\frac{ \frac{\partial}{\partial t } [\log P (t,s) ] }{ ( \sqrt{a^2 + 4} - a ) / 2} \to 1 $ as $t \to \infty$ and $t^3 / s^2 \to \infty$.
	
If $s = 2$, then $a = 0$, $P (t, s) = \int_{0}^{\pi} e^{ t \cos x } \rd x$ and $\frac{\partial}{\partial t} P (t, s) = \int_{0}^{\pi} \cos x e^{ t \cos x } \rd x$. A direct application of Laplace's method \citep{Lap86} yields $ \frac{\partial}{\partial t } [\log P (t,s) ] = [ \frac{\partial}{\partial t }  P (t,s) ] / P(t,s) \to 1$ as $t \to \infty$, proving the result. From now on we assume $s > 2$ and thus $a > 0$. Under our general setting, the proof is quite involved and existing results in asymptotic analysis, including the generalization of Laplace's method to two-parameter asymptotics \citep{Ful51} cannot be directly applied.

Define $f(x, a) = e^{\cos x} \sin^a x$ for $x \in [0,\pi]$. Then $P(t,s) = \int_{0}^{\pi} f^t (x, a) \rd x$ and $\frac{\partial}{\partial t} P (t, s) =  \int_{0}^{\pi} \cos x f^t (x, a) \rd x$. From $\log f (x,a) = \cos x + a \log \sin x$ we get
\begin{align}
\frac{\partial}{\partial x} [ \log f (x,a) ] = - \sin x + a \frac{\cos x}{\sin x}
\qquad \text{and} \qquad 
\frac{\partial^2 }{\partial x^2} [ \log f (x,a) ] = - \cos x - \frac{a}{\sin^2 x}.
\label{KPCA-eqn-lemma-integral-0}
\end{align}
Let $x^*$ be the solution to $\frac{\partial}{\partial x} [ \log f (x,a) ] = 0$ on $( 0,\pi )$. We have $x^* \in (0,\pi / 2)$,
\begin{align}
a = \frac{1}{\cos x^*} - \cos x^*,\qquad
\cos x^* =  \frac{ \sqrt{ a^2 + 4 } - a }{ 2 }
\qquad \text{and}
\qquad
\sin x^* =  \bigg( \frac{ a( \sqrt{ a^2 + 4 } - a ) }{ 2 } \bigg)^{1/2} 
.
\label{KPCA-eqn-lemma-integral-cos}
\end{align}
Moreover, $f (\cdot ,a)$ is strictly increasing in $[0, x^*)$ and strictly decreasing in $(x^*, \pi]$. Hence $x^*$ is its unique maximizer in $[0, \pi]$.

Fix any $\varepsilon \in (0, 1 / 32)$ and let $\delta = \frac{\varepsilon}{\pi} ( \frac{\pi}{2} - x^* )$. Define $I = [x^* - 2 \delta, x^* +2 \delta] \cap [0, \pi] $, $J = [x^* , x^* + \delta / 6 ] $ and $r(a) = \inf_{y \in J} f(y, a) /  \sup_{y \in [0, \pi] \backslash I} f(y, a)$. Then $J \subseteq I \subseteq [0,\pi/2)$ and $|J| = \delta / 6$. We have
\begin{align*}
\bigg|
\frac{P(t,s)}{ \int_I f^t (x,a) \rd x } - 1
\bigg|
& = \frac{ \int_{[0,\pi] \backslash I} f^t (x,a) \rd x  }{ \int_I f^t (x,a) \rd x }
\leq  \frac{ \int_{[0,\pi] \backslash I} f^t (x,a) \rd x  }{ \int_J f^t (x,a) \rd x }\\
& \leq  \frac{  \pi  [ \sup_{y \in [0, \pi] \backslash I} f(y, a) ]^t }{  (\delta / 6) [ \inf_{y \in J } f(y, a) ]^t }
= \frac{ 6 \pi }{  \delta  r^t(a)  }
\end{align*}
and
\begin{align*}
&\bigg|
\frac{ \frac{\partial}{\partial t} P(t,s)}{ \int_I \cos x f^t (x,a) \rd x } - 1
\bigg|
\leq \frac{ \int_{[0,\pi] \backslash I} | \cos x | f^t (x,a) \rd x  }{ \int_I \cos x f^t (x,a) \rd x }
\leq  \frac{ \int_{[0,\pi] \backslash I}  f^t (x,a) \rd x  }{  \int_J \cos x f^t (x,a) \rd x } \notag \\
& \leq  \frac{ \pi  [ \sup_{y \in [0, \pi] \backslash I} f(y, \alpha) ]^t }{ \cos (x^* + \delta) (\delta / 6) [ \inf_{y \in J } f(y, \alpha) ]^t } 
= \frac{6 \pi }{ \cos (x^* + \delta) \delta  r^t(a)  }
\leq  \frac{ 3 \pi^2 }{ \delta^2  r^t(a)  },
\end{align*}
where the last inequality follows from $x^* + 2 \delta < \pi / 2$ and $\cos (x^* + \delta) \geq \cos (\pi / 2 - \delta) = \sin \delta \geq 2 \delta / \pi$.
Consequently,
\begin{align*}
&\max\bigg\{
\bigg|
\frac{P(t,s)}{ \int_I f^t (x,a) \rd x } - 1
\bigg|,~
 \bigg|
\frac{ \frac{\partial}{\partial t} P(t,s)}{ \int_I \cos x f^t (x,a) \rd x } - 1
\bigg|
\bigg\}
\leq  \frac{ 3 \pi^2 }{ \delta^2  r^t(a)  }.
\end{align*}
Let $h(a, t)$ denote the right hand side. If $h(a, t) < 1$, the estimate above yields
\begin{align*}
& \frac{1 - h(a, t)}{1 + h(a, t)}
\leq 
\frac{ [ \frac{\partial}{\partial t }  P(t,s) ] / P(t,s) }{ \int_I \cos x f^t (x,a) \rd x /  \int_I f^t (x,a) \rd x } 
\leq \frac{1 + h(a, t)}{1 - h(a, t)}.
\end{align*}
According to Lemma \ref{KPCA-lemma-trig}, $| \cos x / \cos x^* - 1 | \leq \varepsilon$ holds for all $x \in I$. Hence
\begin{align*}
(1 - \varepsilon)
\frac{ 1 - h(a, t) }{1 + h(a, t)}
\leq 
\frac{ \frac{\partial}{\partial t } [\log P (t,s) ] }{ \cos x^* } \leq 
(1 + \varepsilon) \frac{  1 + h(a, t) }{1 - h(a, t)}.
\end{align*}
{\bf Note that our assumptions $t \to \infty$ and $t^3 / s^2 \to \infty$ imply that $t / (a \vee 1)^2 \to \infty$. Below we will prove $h(a, t) \to 0$ as $t / (a \vee 1)^2  \to \infty$ for any fixed $\varepsilon \in (0,1/32)$.} If that holds, then we get the desired result by letting $\varepsilon \to 0$.

The analysis of $h(a, t)$ hinges on that of $r(a) =  \inf_{y \in J} f(y, a) / \sup_{y \in [0, \pi] \backslash I} f(y, a)$.
The monotonicity of $f (\cdot ,a)$ in $[0, x^*)$ and $(x^*, \pi]$ yields $\inf_{y \in J } f(y, a) = f(x^* + \delta / 6, a)$,
\begin{align*}
\sup_{y \in [0, \pi] \backslash I} f(y, a) & = 
\max \{ f( x^* - 2\delta, a ) , ~ f(x^* + 2\delta, a) \}  \\
& \leq \max \{ f( x^* - \delta / 3, a ) , ~ f(x^* + \delta / 3, a) \} , \qquad \text{if } x^* > 2 \delta, \\
\sup_{y \in [0, \pi] \backslash I} f(y, a) & = f(x^* + 2\delta, a),
 \qquad \text{if } x^* \leq 2 \delta.
\end{align*}
The two cases $x^* > 2 \delta$ and $x^* \leq 2 \delta$ require different treatments. If we define $g(x) = 1/\cos x - \cos x$ for $x \in (0, \pi/2)$, then $a = g( x^*)$ and $\delta = \frac{\varepsilon}{\pi} ( \frac{\pi}{2} - x^* )$ yield the following simple fact.
\begin{fact}\label{KPCA-fact-lemma-integral-2}
If $x^* > 2 \delta$, then $x^* >  \frac{\varepsilon}{1 + 2 \varepsilon / \pi}$, $a > g( \frac{\varepsilon}{1 + 2 \varepsilon / \pi} )$ and $\delta < \frac{\pi \varepsilon}{2 \pi + 4 \varepsilon}$; if
$x^* \leq 2 \delta$, then $x^* \leq  \frac{\varepsilon}{1 + 2 \varepsilon / \pi}$, $a \leq g( \frac{\varepsilon}{1 + 2 \varepsilon / \pi} )$ and $\delta \geq \frac{\pi \varepsilon}{2 \pi + 4 \varepsilon}$.
\end{fact}

{\bf We first consider the case where $x^* > 2 \delta$, which is equivalent to $a > g( \frac{\varepsilon}{1 + 2 \varepsilon / \pi} )$.}
Let $ I' = [x^* - \delta / 3 , x^* + \delta / 3 ] $. For any $y \in I'$, there exists $\xi$ in the closed interval between $x^*$ and $y$ such that 
\begin{align*}
\log f (y ,a) = \log f (x^*,a) + \frac{\partial }{\partial x} [ \log f (x,a) ] |_{x = x^*} (y-x) + \frac{1}{2} \frac{\partial^2 }{\partial x^2} [ \log f (x,a) ] |_{x = \xi} (y-x)^2.
\end{align*}
By construction, $ \frac{\partial }{\partial x} [ \log f (x,a) ] |_{x = x^*} = 0$. From equation (\ref{KPCA-eqn-lemma-integral-0}) we get 
\begin{align*}
 \max_{y \in I'}
\bigg| 
\frac{ 
\frac{\partial^2 }{\partial x^2} [ \log f (x,a) ] |_{x = y}
 }{
\frac{\partial^2 }{\partial x^2} [ \log f (x,a) ] |_{x = x^*}
} - 1 \bigg| 
& \leq \max_{y \in I'}
\bigg| 
\frac{ 
\cos y
}{
\cos x^*
} - 1 \bigg|
+ \max_{y \in I' } \bigg| 
\frac{ 
	\sin^2 x^*
}{
	\sin^2 y
} - 1 \bigg| \\
& \leq \varepsilon + \frac{9}{16}
\leq \frac{1}{32} + \frac{9}{16} = \frac{19}{32},
\end{align*}
where we used Lemma \ref{KPCA-lemma-trig} and $\varepsilon \leq 1 / 32$. Therefore,
\begin{align*}
&\frac{ \inf_{y \in J } \log f(y, a) - \log f(x^*, a) }{
	\frac{1}{2}
\frac{\partial^2 }{\partial x^2} [ \log f (x,a) ] |_{x = x^*} 
} \leq \bigg(
1 + \frac{19}{32}
\bigg)   \bigg( \frac{\delta}{6} \bigg)^2 = \frac{51}{32} \cdot \frac{\delta^2}{36}, \\
&\frac{ \sup_{y \in [0 , \pi] \backslash I'} \log f(y, a) - \log f(x^*, a) }{
	\frac{1}{2}
	\frac{\partial^2 }{\partial x^2} [ \log f (x,a) ] |_{x = x^*} 
} 
 \geq  \bigg(
 1 - \frac{19}{32}
 \bigg)   \bigg( \frac{\delta}{3} \bigg)^2
  = \frac{13}{32} \cdot \frac{\delta^2}{9} = \frac{52}{32} \cdot \frac{\delta^2}{36}.
\end{align*}
Since $\frac{\partial^2 }{\partial x^2} [ \log f (x,a) ] |_{x = x^*} = - \cos x^* - a /  \sin^2 x^* < 0$,
\begin{align*}
\log r(a)
& = \inf_{y \in J } \log f(y, a) - \sup_{y \in [0, \pi] \backslash I} \log f(y, a)
\geq \inf_{y \in J } \log f(y, a) - \sup_{y \in [0,\pi] \backslash I'} \log f(y, a) \\
& \geq \frac{1}{2} \frac{\partial^2 }{\partial x^2} [ \log f (x,a) ] |_{x = x^*} 
\bigg(
 \frac{51}{32} \cdot \frac{\delta^2}{36} - \frac{52}{32} \cdot \frac{\delta^2}{36}
\bigg)
= \frac{ ( \cos x^* + a /  \sin^2 x^*) \delta^2 }{2 \times 32 \times 36} \\
& \gtrsim a \delta^2 / \sin^2 x^*.
\end{align*}
From this and $h(a, t) =  \frac{3  \pi^2 }{ \delta^2 r^t (a) }$ we get
\begin{align*}
- \log h(a, t) = - \log (3 \pi^2) + \log \delta^2 + t \log r(a)
\gtrsim - 1  + \log \delta^2 +  t a \delta^2 / \sin^2 x^*.
\end{align*}
From (\ref{KPCA-eqn-lemma-integral-cos}) we see that $\lim_{a \to \infty}\sin x^* = 1$, $\lim_{a \to \infty}a \cos x^* = 1$ and $\lim_{a \to \infty} a ( \frac{\pi}{2} - x^* ) = 1$. Since $\delta = \frac{\varepsilon}{\pi} ( \frac{\pi}{2} - x^* ) > 0$, we have $\lim_{a \to \infty} a \delta  = \frac{\varepsilon}{\pi}$. There exists $C_1 > 0$ determined by $\varepsilon$ such that for any $a > g( \frac{\varepsilon}{1 + 2 \varepsilon / \pi} )$, we have $\delta \geq C_1 / a$ and $a \delta^2 / \sin^2 x^* \geq C_1 / a$. As a result, for some $C_2$ determined by $\varepsilon$,
\begin{align}
- \log h(a, t) & \geq C_2 ( - 1  -  \log a +  t / a ) \notag\\
& \geq
C_2 [ - 1  -  \log (a \vee 1) +  t / (a \vee 1) ],\qquad \forall a > g \bigg( \frac{\varepsilon}{1 + 2 \varepsilon / \pi} \bigg).
\label{KPCA-eqn-lemma-integral-1}
\end{align}

{\bf We move on to the case where $x^* \leq 2 \delta$.} Recall that for $x \in (x^*,  x^* + 2 \delta) \subseteq (x^*, \pi / 2)$, we have $\frac{\partial }{\partial x } [\log f(x, a) ] < 0$ and
\begin{align*}
- \frac{\partial^2 }{\partial x^2 } [\log (x,a)] = \cos x + \frac{a}{ \sin^2 x } \geq  \cos x \geq \cos (x^* + 2 \delta)
\geq \cos ( 4 \delta ) \geq \cos (1/16),
\end{align*}
where we used $\delta \leq \varepsilon / 2 \leq 1 / 64$.
By Taylor expansion, there exists $\xi \in [ x^* + \delta / 6, x^* + 2\delta ]$ such that
\begin{align*}
\log r(a) & =  \inf_{y \in J } \log f(y, a) - \sup_{y \in [0, \pi] \backslash I} \log f(y, a) = 
 \log f(x^* + \delta / 6, a)- \log f(x^* + 2\delta, a)\\
&= - \bigg( \frac{\partial }{\partial x} [\log (x,a)] |_{ x = x^* + \delta / 6 } (2 \delta  - \delta / 6 )
+ \frac{1}{2} \frac{\partial ^2}{\partial x^2 } [\log (x,a)] |_{ x = \xi } (2 \delta  - \delta / 6 )^2 
\bigg)\\
& > \frac{1}{2} \inf_{x \in [x^*, x^* + 2 \delta]} \bigg( - \frac{\partial^2 }{\partial x^2 } [\log (x,a)] \bigg) ( 2 \delta - \delta / 6 )^2
\gtrsim \delta^2.
\end{align*}
Based on $h(a, t) =  \frac{3  \pi^2 }{ \delta^2 r^t (a) }$ and $\delta \geq \frac{\pi \varepsilon}{2 \pi + 4 \varepsilon}$ from Fact \ref{KPCA-fact-lemma-integral-2}, there exists some $C_3>0$ determined by $\varepsilon$ such that
$- \log h(a, t) \geq C_3 ( - 1 +  t ) \geq C_3 [ -1 + t / (a \vee 1) ]$ holds when $a \leq g ( \frac{\varepsilon}{1 + 2 \varepsilon / \pi} )$.

{\bf This bound,
(\ref{KPCA-eqn-lemma-integral-1}) and $\log (a \vee 1 ) \leq  a \vee 1 $ imply that}
\begin{align*}
- \log h(a, t) 
& \gtrsim -1 - \log ( a \vee 1 ) + \frac{t}{a \vee 1}
\geq -1 - ( a \vee 1 ) + \frac{t}{a \vee 1} \\
& =  -1 + (a \vee 1 ) \bigg(
\frac{t}{( a \vee 1 )^2} - 1
\bigg).
\end{align*}
As $t / (a \vee 1)^2 \to \infty$, we have $ - \log h(a, t) \to \infty$ and $ h(a, t) \to 0$.
\end{proof}

\begin{lemma}\label{KPCA-lemma-integral-derivative}
	For $t\geq 0$ and $s \geq 2$, define $a = (s-2) / t$ and $g(t, s) = ( \sqrt{a^2 + 4} - a ) / 2$.
	There exist a constant $c \in (0,1)$ and a function $w: [0,c] \to [0,1)$ such that when $\max\{ 1/t_0,~ d^2 / t_0^3,~ |t_2 - t_0| /t_0,~ |t_1 - t_0| /t_0 \} \leq c$,
	\begin{align*}
	& \bigg| \frac{ \log P(t_2, s) - \log P(t_1, s) }{ g(t_0, s) (t_2 - t_1) } - 1 \bigg| \leq w ( \max\{ 1/t_0,~ s^2 / t_0^3,~ |t_2 - t_0| /t_0,~ |t_1 - t_0| /t_0 \} ).
	\end{align*}
\end{lemma}

\begin{proof}[\bf Proof of Lemma \ref{KPCA-lemma-integral-derivative}]
	Let $h(a) = ( \sqrt{a^2 + 4} - a ) / 2$. Observe that $\frac{ \partial a}{ \partial t } = - (s - 2) / t^2 = - a / t$ and $h'(a) = \frac{1}{2} ( \frac{a}{\sqrt{a^2 + 4}} - 1 )
	= - h(a) / \sqrt{a^2 + 4}$. By the chain rule,
	\begin{align*}
	\frac{ \partial }{ \partial t } [ \log g(t, s) ] = \frac{ \rd }{ \rd a } [\log h(a) ] \cdot \frac{ \partial a}{ \partial t } = \frac{ h'(a) }{ h(a) } \cdot \frac{ \partial a}{ \partial t } = \frac{a}{t \sqrt{a^2 + 4} }.
	\end{align*}
	Hence $0 < \frac{ \partial }{ \partial t } [ \log g(t, s) ] \leq 1/t$. For any $t_2 \geq t_1 > 0$ there exists $\xi \in [t_1, t_2]$ such that
	\begin{align*}
	0 \leq \log \bigg(
	\frac{g(t_2, s)}{ g(t_1, s) }
	\bigg) =
	\log g(t_2, s) - \log g(t_1, s) = \frac{ \partial }{ \partial t } [ \log g(t, s) ] |_{t = \xi} (t_2 - t_1) \leq \frac{t_2 - t_1}{\xi} \leq \frac{ t_2 - t_1 }{t_1}.
	\end{align*}
	This leads to $| g(t_2, s) /  g(t_1, s) - 1 | \leq e^{|t_2 - t_1| / (t_1 \wedge t_2) } - 1 $ for any $t_1, t_2 > 0$.
	
	Let $c$ and $w$ be those defined in the statement of Lemma \ref{KPCA-lemma-integral}. Suppose that $t_0 > 0$ and $s \geq 2$ satisfies $\max\{ 1/t_0, s^2 / t_0^3 \} < c / 2$. When $t \geq t_0 / 2^{1/3}$, $\max\{ 1/t, s^2 / t^3 \} \leq 2 \max\{ 1/t_0, s^2 / t_0^3 \} < c$. Lemma \ref{KPCA-lemma-integral} and the non-decreasing property of $w$ force
	\begin{align*}
	\bigg| \frac{ \frac{\partial }{\partial t} [\log P(t, s)] }{
		g(t,s)
	} - 1 \bigg| & \leq  w ( \max\{ 1/ (t_0 / 2^{1/3}), s^2 / (t_0 / 2^{1/3})^3 \} ) \\
	& \leq w ( 2 \max\{ 1/ t_0, s^2 / t_0^3 \} ), \qquad \forall t \geq  t_0 / 2^{1/3}.
	\end{align*}
	When $|t - t_0| \leq t_0/ 5$, we have $t \geq 0.8 t_0 \geq t_0 / 2^{1/3}$ and $2 |t - t_0| \leq 0.4 t_0 \leq  t_0 / 2^{1/3}$. Then $|t - t_0| / (t_0 \wedge t) \leq 1/2$ and
	\begin{align*}
	| g(t, s) /  g(t_0, s) - 1 | \leq e^{|t - t_0| / (t_0 \wedge t) } - 1 \leq e^{1/2} \frac{|  t - t_0| }{t_0 \wedge t}
	\leq \frac{ \sqrt{e} |t - t_0| }{t_0 / 2^{1/3}}
	\leq \frac{ 3 |t - t_0| }{t_0} < 1.
	\end{align*}
	Hence when $t \in  [4t_0/5, 6t_0 / 5]$,
	\begin{align*}
&	[1 - w ( 2 \max\{ 1/ t_0, s^2 / t_0^3 \} ) ] \bigg(
	1 - \frac{ 3 |t - t_0| }{t_0} 
	\bigg) \leq
	\frac{ \frac{\partial }{\partial t} [\log P(t, s)] }{
		g(t_0, s)
	} \\
& \leq [ 1 + w ( 2 \max\{ 1/ t_0, s^2 / t_0^3 \} ) ] \bigg(
	1 + \frac{ 3 |t - t_0| }{t_0} 
	\bigg).
	\end{align*}
	
	We can find a constant $\tilde{c} \in (0,1)$ and construct a new function $\tilde{w}: [0,\tilde{c}] \to [0,1)$ such that for any distinct $t_1, t_2 \in  [ (1 - \tilde{c} ) t_0, (1 + \tilde{c} ) t_0]$,
	\begin{align*}
	& \bigg| \frac{ \log P(t_2, s) - \log P(t_1, s) }{ g(t_0, s) (t_2 - t_1) } - 1 \bigg| \leq \tilde{w} ( \max\{ 1/t_0,~ s^2 / t_0^3,~ |t_2 - t_0| /t_0,~ |t_1 - t_0| /t_0 \} ).
	\end{align*}
	The proof is completed by re-defining $c$ and $w$ as $\tilde{c} $ and $\tilde{w}$, respectively.
\end{proof}

{
\bibliographystyle{ims}
\bibliography{bib}
}

\end{document}